\documentclass{amsart}

\usepackage[all]{xy}
\usepackage{mathabx}
\usepackage{appendix}

\theoremstyle{plain}
\newtheorem{hypo}{Hypothesis}
\newtheorem{thm}{Theorem}[subsection]
\newtheorem{prop}[thm]{Proposition}
\newtheorem{lemma}[thm]{Lemma}

\newtheorem{secthm}{Theorem}[section]

\theoremstyle{definition}
\newtheorem{dfn}[thm]{Definition}

\newtheorem{convention}[secthm]{Convention}

\theoremstyle{remark}
\newtheorem{rem}[thm]{Remark}

\newcommand{\HH}{\mathrm{H}}

\newcommand{\fl}{\mathrm{f\/l}}
\newcommand{\spl}{\mathrm{sp}}
\newcommand{\LL}{\mathbf{L}}
\newcommand{\fin}{\mathrm{f\/in}}

\begin{document}

\title{Deformations of complexes for finite dimensional algebras}

\author{Frauke M. Bleher}
\address{F.B.: Department of Mathematics\\University of Iowa\\
14 MacLean Hall\\Iowa City, IA 52242-1419, U.S.A.}
\email{frauke-bleher@uiowa.edu}
\thanks{The first author was supported in part by NSF Grant DMS-1360621.}
\author{Jos\'{e} A. V\'{e}lez-Marulanda}
\address{J.V.: Department of Mathematics \& Computer Science\\Valdosta State University\\
2072 Nevins Hall\\Valdosta, GA 31698-0040, U.S.A.}
\email{javelezmarulanda@valdosta.edu}
\thanks{The second author was
supported by the Release Time for Research Scholarship of the Office of Academic Affairs at the Valdosta State University.}


\begin{abstract}
Let $k$ be a field and let $\Lambda$ be a finite dimensional $k$-algebra. We prove that every bounded complex $V^\bullet$
of finitely generated $\Lambda$-modules has a well-defined versal deformation ring $R(\Lambda,V^\bullet)$ which is
a complete local commutative Noetherian $k$-algebra with residue field $k$. We also prove that  nice two-sided tilting
complexes between $\Lambda$ and another finite dimensional $k$-algebra $\Gamma$ preserve these versal deformation
rings. Additionally, we investigate stable equivalences of Morita type between self-injective algebras in this context.
We apply these results to the derived equivalence classes of the members of a particular family of algebras of dihedral type
that were introduced by Erdmann and shown by Holm to be not derived equivalent to any block of a group algebra.
\end{abstract}

\maketitle


\section{Introduction}
\label{s:intro}
\setcounter{equation}{0}

The main objective of the theory of deformations of algebraic objects, such as modules or group representations, is
to study the behavior of these objects under perturbations. Suppose $\mathcal{O}$ is a complete local commutative Noetherian ring with residue field $k$, $\Lambda_{\mathcal{O}}$ is an $\mathcal{O}$-algebra 
and $\Lambda=k\otimes_{\mathcal{O}} \Lambda_{\mathcal{O}}$. If $V$ is a $\Lambda$-module of finite $k$-dimension,
the deformations of $V$ are defined to be isomorphism classes of lifts of $V$ over complete local commutative Noetherian 
$\mathcal{O}$-algebras $R$ with residue field $k$. Here a lift of $V$ over  $R$ is an
$R\otimes_{\mathcal{O}}\Lambda_{\mathcal{O}}$-module $M$ that is free over $R$, together with a 
$\Lambda$-module isomorphism $\phi:k\otimes_R M \to V$. 

Lifts and deformations of this kind were studied  by Green in \cite{green} in the 1950's  in  the case when $\mathcal{O}$ is a
ring of $p$-adic integers, for some prime number $p$, and $\Lambda_{\mathcal{O}}$ is the group algebra
of a group $G$ over $\mathcal{O}$. Green's work inspired Auslander, Ding and Solberg in \cite{ads} to
consider more general  $\mathcal{O}$-algebras $\Lambda_{\mathcal{O}}$ and more general lifting problems. 
In the 1970's Laudal developed a theory of formal moduli of algebraic structures, and he used Massey products
to describe deformations of $k$-algebras and their modules over complete local commutative Artinian $k$-algebras
with residue field $k$ (see \cite{laudal1983} and its references). In the 1980's Mazur developed a theory of
deformations of group representations to systematically study $p$-adic lifts of representations of profinite Galois groups
over finite fields of characteristic $p$ (see \cite{maz1,Maz}). Both Laudal and Mazur used Schlessinger's criteria in \cite{Sch}
for the pro-representability of functors of Artinian rings. One advantage of Mazur's approach is that
he uses a continuous deformation functor, which allows him to include the
deformations over arbitrary complete local commutative Noetherian $\mathcal{O}$-algebras with residue field $k$
directly in his functorial description and not just as inverse limits. In particular, Mazur proved in \cite{maz1} that a finite dimensional
Galois representation over a finite field always has a versal deformation ring, and in the case when the representation is
absolutely irreducible that this versal deformation ring is universal.
In \cite{blehervelez}, Mazur's approach was used by the authors
to study deformation rings and deformations of modules for arbitrary finite dimensional $k$-algebras $\Lambda$ when 
$\mathcal{O}=k$, and to provide additional structure theorems in the case when $\Lambda$ is self-injective or Frobenius.

Let now $k$ be a field of arbitrary characteristic, let $\mathcal{O}=k$, and let $\Lambda=\Lambda_{\mathcal{O}}$ 
be a finite dimensional $k$-algebra.
Our first goal is to generalize the deformation theory for finitely generated $\Lambda$-modules in \cite{blehervelez} 
to bounded complexes of finitely generated $\Lambda$-modules. We accomplish this goal in Section \ref{s:udr}.
Many of our techniques are based 
on the generalization of Mazur's deformation theory to bounded complexes of Galois representations in 
\cite{comptes,bcderived}. 

More precisely, let $D^-(\Lambda)$ be the derived category
of bounded above cochain complexes of pseudocompact $\Lambda$-modules (see Section \ref{s:setup} for
a review of pseudocompact rings and modules).  The following is our first main result (see Theorem \ref{thm:bigthm}
for a more precise statement):

\begin{secthm}
\label{thm:main1}
Let $V^\bullet$ be an object of $D^-(\Lambda)$ such that $V^\bullet$ only has finitely many non-zero cohomology groups,
all of which have finite $k$-dimension. Then $V^\bullet$ always has a versal deformation ring $R(\Lambda,V^\bullet)$ 
which is a complete local commutative Noetherian $k$-algebra with residue field $k$. Moreover, $R(\Lambda,V^\bullet)$ is 
universal if the endomorphism ring of $V^\bullet$ in $D^-(\Lambda)$ is isomorphic to $k$.
\end{secthm}

Additionally, we prove that the case of modules corresponds to the case
when $V^\bullet$ has precisely one non-zero cohomology group. 
The main challenge of the proof of Theorem \ref{thm:main1} is to ensure that the arguments in \cite{comptes,bcderived}
can be modified to work for arbitrary finite dimensional $k$-algebras $\Lambda$ that may be neither Frobenius nor self-injective.

Note that we provide more details concerning the continuity of our deformation functor than were provided 
for the deformation functor defined in \cite{bcderived}. This can also be used to better explain the arguments used to prove
\cite[Prop. 7.2]{bcderived}; see Remark \ref{rem:therefereeaskedforit}.

There has been a lot of interest in classifying finite dimensional $k$-algebras up to derived or stable equivalences.
One of the most famous conjectures in this context is Brou\'{e}'s conjecture that blocks of group rings of finite groups $G$
with an abelian defect group $D$ are derived equivalent to blocks of the normalizer of $D$ in $G$
(see \cite{broue,rickardICM} and their references).
This conjecture has spurred a lot of work on derived equivalences for more general algebras. 
For example, Rickard proved in \cite{rickard1} that 
two finite dimensional $k$-algebras $\Lambda$ and $\Gamma$ are derived equivalent if and only if there is a
derived equivalence between them that is given by the left derived tensor product functor with a so-called two-sided tilting complex.
Such a derived equivalence is also called a standard derived  equivalence.
It is then a natural question to ask whether standard derived equivalences preserve versal deformation rings of 
complexes $V^\bullet$ as in Theorem \ref{thm:main1}.

In \cite{rickard2}, Rickard showed that if $\Lambda$ and $\Gamma$ are derived equivalent block algebras for 
finite groups then one can choose the two-sided tilting complex providing the standard derived equivalence
to be particularly nice (namely, to be a split-endomorphism two-sided tilting complex). In \cite{derivedeq}, it was shown that such nice  two-sided tilting complexes indeed preserve
versal deformation rings when $\Lambda$ and $\Gamma$ are block algebras.

Our second goal, accomplished in Section \ref{s:derivedequivalences}, is to provide a variation on ``niceness'' 
of two-sided tilting complexes that works for arbitrary finite dimensional algebras (see Definition \ref{def:nice2sidedtilting}).
This leads to our second main result (see Theorem \ref{thm:deformations} for a more precise statement):

\begin{secthm}
\label{thm:main2}
Suppose $\Gamma$ is another finite dimensional $k$-algebra such that $\Lambda$ and $\Gamma$ are derived equivalent.
Then there exists a nice two-sided tilting complex $P^\bullet$ of finitely generated $\Gamma$-$\Lambda$-bimodules such that
if $V^\bullet$ is a bounded complex of finitely generated $\Lambda$-modules and 
${V'}^\bullet=P^\bullet\otimes^{\LL}_\Lambda V^\bullet$,
then the versal deformation rings $R(\Lambda,V^\bullet)$ and $R(\Gamma,{V'}^\bullet)$ are isomorphic.
\end{secthm}

To prove Theorem \ref{thm:main2}, the main challenge is again to modify the arguments in \cite{derivedeq} so that they 
work for arbitrary finite dimensional $k$-algebras $\Lambda$ that may be  neither Frobenius nor self-injective.

In \cite{rickard1}, Rickard proved that if $\Lambda$ and $\Gamma$ are self-injective derived equivalent $k$-algebras, 
then there is a stable equivalence of Morita type between them, as introduced by Brou\'{e} in \cite{broue1}. 
Such stable equivalences of Morita type provide especially well-behaved equivalences between the stable module
categories of $\Lambda$ and $\Gamma$. On the other hand, not every stable equivalence of Morita type between
self-injective algebras is induced by a 
derived equivalence (see, for example, \cite{dugas} and its references).
Therefore, we prove in Proposition \ref{prop:stabmordef} that arbitrary stable equivalences of Morita type between
{self-injective} algebras preserve versal deformation rings of modules. 

In Section \ref{s:examples}, we show how the main results
from Sections \ref{s:udr} and \ref{s:derivedequivalences} can be applied to the derived equivalence
classes of the members of a particular family of algebras of dihedral type that was introduced by Erdmann in \cite{erd} and denoted by $D(3\mathcal{R})$. Note
that Holm showed in \cite[Sect. 3.2]{holm} that none of the algebras in the family $D(3\mathcal{R})$ is derived equivalent to a block of
a group algebra. Let $\Lambda_0$ be a particular algebra in the family $D(3\mathcal{R})$. Theorems \ref{thm:derivedex2},
\ref{thm:derivedex3} and Proposition \ref{prop:derivedexvelez} demonstrate how the knowledge of the universal
deformation rings of certain $\Lambda_0$-modules can be used to determine the universal deformation
rings of modules for another algebra $\Lambda$ that is just known to be derived equivalent to $\Lambda_0$.

We would like to thank the referees for helpful comments which helped improve the readability of the paper.

\medskip

\begin{convention}
\label{con:complexes} 
Throughout this paper ``complex'' means ``cochain complex.'' The degree $n$ term of a 
complex $C^\bullet$ is denoted by $C^n$ and its degree $n$ differential is denoted by 
$\delta^n=\delta^n_C: C^n\to C^{n+1}$.

Even if we apply contravariant functors to cochain complexes, we shall assume, without saying this 
explicitly, that the terms of the resulting complexes are renumbered in order to regain cochain complexes. 
Thus, for example, the $k$-dual $\mathrm{Hom}_k(V^\bullet, k)$ of a bounded complex $V^\bullet$ of $\Lambda$-modules
will have degree $n$ term given by $\mathrm{Hom}_k(V^{-n}, k)$.

If $C^\bullet$ is a complex and $i$ is an integer, then $C^\bullet[i]$ is the complex obtained by
``shifting $C^\bullet$ to the left by $i$ places.'' More precisely, the degree $n$ term of $C^\bullet[i]$ is
$C^{n+i}$ and the degree $n$ differential of $C^\bullet[i]$ is $(-1)^i\delta^{n+i}_C$.

Our complexes are either bounded above, i.e. complexes $C^\bullet$ with $C^n=0$ for $n>>0$, or
bounded, i.e. complexes $C^\bullet$ with $C^n=0$ for all but finitely many $n$.
\end{convention}


\section{Versal deformation rings for complexes over finite dimensional algebras}
\label{s:udr}
\setcounter{equation}{0}

In \cite{comptes,bcderived}, it was proved that if $k$ is a field of positive characteristic,
$G$ is a profinite group satisfying a certain finiteness
condition and $V^\bullet$ is quasi-isomorphic to a bounded complex of pseudocompact 
$[[kG]]$-modules, then $V^\bullet$ always has a versal deformation ring.
Moreover, it was proved that if the endomorphism ring of $V^\bullet$ 
in the derived category of bounded above complexes of pseudocompact $[[kG]]$-modules 
is isomorphic to $k$, then this versal deformation ring is universal. 

It is the goal of this section to prove an analogous result when $k$ is  an arbitrary field
and $[[kG]]$ is replaced by an arbitrary finite dimensional $k$-algebra $\Lambda$.
In Section \ref{s:setup}, we will recall some results on pseudocompact rings and modules,
define quasi-lifts and deformations of complexes, and state our main result of this section, Theorem \ref{thm:bigthm}.
In Section \ref{s:quasilifts}, we analyze the structure of quasi-lifts over Artinian rings. In Section \ref{s:prelims},
we provide some more results on complexes and quasi-lifts. In Section
\ref{s:proofbigthm}, we prove Theorem \ref{thm:bigthm}. In Section \ref{s:12terms}, 
we consider quasi-lifts of one-term and two-term complexes, and completely split complexes.


\subsection{Pseudocompact rings and modules, quasi-lifts, and deformations of complexes}
\label{s:setup}

Recall from \cite{brumer} 
that a pseudocompact ring $\Omega$ is a complete Hausdorff topological ring that admits a system 
of open neighborhoods of $0$ consisting of two-sided ideals $I$ for which $\Omega/I$ is an Artinian ring.
In particular, every finite dimensional algebra $\Lambda$ over any field $k$
is a pseudocompact ring, by choosing all two-sided ideals of $\Lambda$ as an 
open neighborhood basis of $0$. In other words, $\Lambda$ is a pseudocompact ring with the discrete
topology. Moreover, since $k$ is a commutative pseudocompact ring (again with the discrete topology),
$\Lambda$ is a pseudocompact $k$-algebra.
A $\Lambda$-module $M$ is called pseudocompact if $M$ is a complete Hausdorff topological 
$\Lambda$-module that has a basis of open neighborhoods  of $0$ consisting of submodules $N$ for which
$M/N$ has finite length as $\Lambda$-module. In particular, each finitely generated $\Lambda$-module
$M$ is pseudocompact, by choosing all submodules of $M$ as an open neighborhood basis of $0$
(i.e. by giving $M$ the discrete topology). The category of pseudocompact $\Lambda$-modules is an abelian
category; see below for more details.
Let $D^-(\Lambda)$ be the derived category of bounded above complexes of pseudocompact
$\Lambda$-modules.

\begin{hypo}
\label{hypo:fincoh}
Throughout this paper, we assume that $k$ is an arbitrary field, $\Lambda$ is a finite dimensional
$k$-algebra, and $V^\bullet $ is a
complex in $D^-(\Lambda)$ that has  only finitely many non-zero cohomology
groups, all of which have finite $k$-dimension.
\end{hypo}

Define $\hat{\mathcal{C}}$ to be the category of all complete local commutative Noetherian 
$k$-algebras with residue field $k$. The morphisms in $\hat{\mathcal{C}}$ are 
continuous $k$-algebra homomorphisms that induce the identity on $k$.
Let $\mathcal{C}$ be the full subcategory of  Artinian rings in $\hat{\mathcal{C}}$.

For $R \in \mathrm{Ob}(\hat{\mathcal{C}})$, define $R\Lambda=R\otimes_k\Lambda$. 
Then $R$ is a commutative pseudocompact ring, and
$R\Lambda$ is a pseudocompact $R$-algebra. Define $\mathrm{PCMod}(R\Lambda)$ 
to be the category of pseudocompact $R\Lambda$-modules.

Pseudocompact rings, algebras and modules have been studied, for example, in 
\cite{ga1,ga2} and \cite{brumer}. For the convenience of the reader, we state some useful facts
from these references.

\begin{rem}
\label{rem:pseudocompact}
Let $R \in \mathrm{Ob}(\hat{\mathcal{C}})$.
\begin{enumerate}
\item[(i)] The ring $R\Lambda$ is the inverse limit 
of Artinian quotient rings. 
An $R\Lambda$-module is pseudocompact if and only if it is the inverse limit of 
$R\Lambda$-modules of finite length. 
Moreover, an $R\Lambda$-module has finite length 
if and only if it has finite length as an $R$-module.
The category $\mathrm{PCMod}(R\Lambda)$ is an abelian
category with exact inverse limits. 

\item[(ii)] A pseudocompact $R\Lambda$-module $M$ is said to be topologically free on a set
$X=\{x_i\}_{i\in I}$ if $M$ is isomorphic to the product of a family $(R\Lambda_i)_{i\in I}$ where
$R\Lambda_i=R\Lambda$ for all $i$.
Every topologically free pseudocompact $R\Lambda$-module is a projective object of 
$\mathrm{PCMod}(R\Lambda)$, and every pseudocompact
$R\Lambda$-module is the quotient of a topologically free $R\Lambda$-module. Hence
$\mathrm{PCMod}(R\Lambda)$ has enough projective objects. 

\item[(iii)] Suppose $M$ and $N$ are pseudocompact $R\Lambda$-modules. Then we define the right 
derived functors $\mathrm{Ext}^n_{R\Lambda}(M,N)$ by using a projective resolution of $M$ in $\mathrm{PCMod}(R\Lambda)$.
\end{enumerate}
\end{rem}

\begin{rem}
\label{rem:extrafree}
Let $R$ be an object of $\hat{\mathcal{C}}$ with maximal ideal $m_R$. Suppose that 
$(R/m_R^i)X_i$ denotes an abstractly free $(R/m_R^i)$-module on the finite topological space $X_i$
for all $i$,
and that $\{X_i\}_i$ forms an inverse system. Define $X=\displaystyle \lim_{\stackrel{
\longleftarrow}{i}} X_i$ and $[[R X]] = \displaystyle \lim_{\stackrel{\longleftarrow}{i}} \,
(R/m_R^i) X_i$. Then
$[[R X]]$ is a topologically free pseudocompact $R$-module on $X$. 
\end{rem}

\begin{rem}
\label{rem:topflat}
Suppose $R \in \mathrm{Ob}(\hat{\mathcal{C}})$, and $\Omega=R$ or $\Omega=R\Lambda$. 
Let $M$ be a right (resp. left) pseudocompact $\Omega$-module.
\begin{enumerate}
\item[(i)]
Let $\hat{\otimes}_{\Omega}$ denote the completed tensor product in the category 
$\mathrm{PCMod}(\Omega)$ (see \cite[Sect. 2]{brumer}). Then $M\hat{\otimes}_\Omega- $ 
(resp. $-\hat{\otimes}_\Omega M$) is a right exact functor. Moreover,  $M$ is said to be 
topologically flat, if the functor $M\hat{\otimes}_{\Omega}-$ (resp. $-\hat{\otimes}_\Omega M$)
is exact.

\item[(ii)]
By \cite[Lemma 2.1]{brumer} and \cite[Prop. 3.1]{brumer}, $M$ is topologically flat if and only 
if $M$ is projective as a pseudocompact $\Omega$-module.

\item[(iii)]
If $M$ is finitely generated as a pseudocompact $\Omega$-module, it follows from
\cite[Lemma 2.1(ii)]{brumer} that the functors
$M \otimes_\Omega -$ and $M\hat{\otimes}_\Omega -$ (resp. $-\otimes_\Omega M$ and 
$-\hat{\otimes}_\Omega M$) are naturally isomorphic.

\item[(iv)]
If $\Omega=R$ and $M$ is a pseudocompact $R$-module,
it follows from \cite[Proof of Prop. 0.3.7]{ga2} and \cite[Cor. 0.3.8]{ga2} that $M$ is 
topologically flat if and only if $M$ is topologically free if and only if $M$ is abstractly flat.
In particular, if $R$ is Artinian, a pseudocompact $R$-module is topologically flat  if and only
if it is abstractly free.
\end{enumerate}
\end{rem}

For $R \in \mathrm{Ob}(\hat{\mathcal{C}})$, let $C^-(R\Lambda)$
be the abelian category of complexes of pseudocompact $R\Lambda$-modules that are bounded above, let $K^-(R\Lambda)$ be the homotopy category of $C^-(R\Lambda)$, and
let $D^-(R\Lambda)$ be the derived category of $K^-(R\Lambda)$. 
Let $[1]$ denote the translation functor on $C^-(R\Lambda)$ (resp. $K^-(R\Lambda)$, 
resp. $D^-(R\Lambda)$), 
i.e. $[1]$ shifts complexes
one place to the left and changes the signs of the differentials (see Convention \ref{con:complexes}).
Recall that a homomorphism in $C^-(R\Lambda)$
is a quasi-isomorphism if and only if the induced homomorphisms on 
all the cohomology groups are bijective.

\begin{rem}
\label{rem:leftderivedtensor}
Let $R \in \mathrm{Ob}(\hat{\mathcal{C}})$, and let $\mathcal{P}_R$ be the additive subcategory of 
$\mathrm{PCMod}(R\Lambda)$ of projective objects.
By Remark \ref{rem:pseudocompact}(ii) and \cite[Thm. 10.4.8]{Weibel}, the natural functor 
$K^-(\mathcal{P}_R)\to D^-(R\Lambda)$ is an equivalence of triangulated categories. 
Let $\sigma_R: D^-(R\Lambda)\to K^-(\mathcal{P}_R)$ be a quasi-inverse.

Suppose $S$ is a pseudocompact $R$-module. In the case when $S\in \mathrm{Ob}(\hat{\mathcal{C}})$ and there
exists a morphism $\alpha:R\to S$ in $\hat{\mathcal{C}}$ defining the $R$-module structure on $S$, let $R_S=S$.
In all other cases, let $R_S=R$. Consider the completed tensor product functor
$$S\hat{\otimes}_R-:K^-(R\Lambda)\to K^-(R_S\Lambda).$$
By \cite[Thm. 10.5.6]{Weibel}, its left derived functor $S\hat{\otimes}^{\LL}_R-:
D^-(R\Lambda)\to D^-(R_S\Lambda)$ is the following
composition of functors of triangulated categories:
\begin{equation}
\label{eq:leftderiveddef}
D^-(R\Lambda)\xrightarrow{\sigma_R}K^-(\mathcal{P}_R)\xrightarrow{S\hat{\otimes}_R-}
K^-(R_S\Lambda)\xrightarrow{q_S}D^-(R_S\Lambda)
\end{equation}
where $q_S:K^-(R_S\Lambda)\to D^-(R_S\Lambda)$ is the localization functor.
In other words, if $X^\bullet$ is in $\mathrm{Ob}(K^-(R\Lambda))$, then 
there exists an isomorphism $\rho_X:X^\bullet\to \sigma_R(X^\bullet)$ in $D^-(R\Lambda)$ and
\begin{equation}
\label{eq:leftderived}
S\hat{\otimes}^{\LL}_R X^\bullet=S\hat{\otimes}_R\;\sigma_R(X^\bullet).
\end{equation}
\end{rem}

The following definitions and remarks are adapted from \cite[Sect. 2]{bcderived} to our situation.
Note that we follow Illusie's definition of finite tor dimension as given in \cite[Def. 5.2]{IllusieSGA6}
(see also \cite[Sect. 8.3.6]{GrothendieckFGAExplained}).

\begin{dfn}
\label{def:fintor}
We will say that a complex $M^\bullet$ in $K^-(R\Lambda)$
has \emph{finite pseudocompact $R$-tor dimension}, 
if there exists an integer $N$ such that for all pseudocompact
$R$-modules $S$, and for all integers $i<N$, ${\HH}^i(S\hat{\otimes}^{\LL}_R M^\bullet)=0$.
If we want to emphasize the integer $N$ in this definition, we say $M^\bullet$ has 
\emph{finite pseudocompact $R$-tor dimension at $N$}.
\end{dfn}

\begin{rem}
\label{rem:dumbdumb}
Suppose $M^\bullet$ is a complex in $K^-(R\Lambda)$ of topologically flat, hence topologically free, 
pseudocompact $R$-modules that has finite pseudocompact $R$-tor dimension
at $N$. Then the bounded complex ${M'}^\bullet$, which is
obtained from $M^\bullet$ by replacing $M^N$ by
${M'}^N=M^N/\delta^{N-1}(M^{N-1})$ and by setting ${M'}^i = 0$ if $i < N$,
is quasi-isomorphic to $M^\bullet$ and
has topologically free pseudocompact terms over $R$.
\end{rem}

\begin{dfn}
\label{def:lifts} 
Let $R$ be an object of $\hat{\mathcal{C}}$.
A \emph{quasi-lift}
of $V^\bullet$ over $R$ is a pair $(M^\bullet,\phi)$ consisting of a complex
$M^\bullet$ in
$D^-(R\Lambda)$ that has finite pseudocompact $R$-tor  dimension
together with an isomorphism
\hbox{$\phi: k \hat{\otimes}^{\LL}_R M^\bullet \to V^\bullet$} in $D^-(\Lambda)$.
Two
quasi-lifts $(M^\bullet, \phi)$ and $({M'}^\bullet,\phi')$ of $V^\bullet$ over $R$ are
\emph{isomorphic} if there is an isomorphism
$f:M^\bullet \to {M'}^\bullet$ in $D^-(R\Lambda)$ with
$\phi'\circ (k\,\hat{\otimes}_R^{\LL} f)=\phi$.
A \emph{deformation}
of $V^\bullet$ over $R$ is an isomorphism class of quasi-lifts of $V^\bullet$ over $R$.
We denote the deformation of $V^\bullet$ over $R$ represented by $(M^\bullet,\phi)$ by 
$[M^\bullet,\phi]$.

A \emph{proflat quasi-lift} of $V^\bullet$ over $R$ 
is a quasi-lift $(M^\bullet,\phi)$ of $V^\bullet$ over $R$ whose cohomology
groups are topologically flat, and hence topologically free, pseudocompact $R$-modules.
A \emph{proflat deformation} of $V^\bullet$ over $R$ is an isomorphism class
of proflat quasi-lifts of $V^\bullet$ over $R$.
\end{dfn}

\begin{rem}
\label{rem:difference}
There exist quasi-lifts that are not isomorphic to proflat quasi-lifts in $D^-(R\Lambda)$.
For example, suppose $\Lambda=k$
and $V^\bullet = k \xrightarrow{0} k$ is the two-term complex concentrated in the degrees
$-1$ and $0$ with trivial differential.  Then the two-term complex
$M^\bullet = k[[t]] \xrightarrow{t} k[[t]]$ concentrated in the degrees $-1$ and $0$
defines a quasi-lift of $V^\bullet$ over $k[[t]]$. However,  since $M^\bullet$ is isomorphic to
the one-term complex $k[[t]]/tk[[t]]\cong k$ concentrated in degree $0$,
this quasi-lift is not isomorphic to a  proflat quasi-lift of $V^\bullet$ over $k[[t]]$.
\end{rem}

\begin{rem}
\label{rem:profree}
The following two statements are proved in the same way as \cite[Lemmas 2.9 and 2.11]{bcderived}
by using Remark \ref{rem:pseudocompact}(ii) and the fact that
a bounded above complex of topologically free pseudocompact modules  whose cohomology groups are all topologically free splits completely.
\begin{enumerate} 
\item[(i)] Suppose $R \in \mathrm{Ob}(\hat{\mathcal{C}})$ and $(M^\bullet,\phi)$ is a quasi-lift of $V^\bullet$ 
over $R$. Then there exists a quasi-lift $(P^\bullet,\psi)$ of $V^\bullet$ over $R$ that is isomorphic to 
$(M^\bullet,\phi)$ such that the terms of $P^\bullet$ are topologically free pseudocompact 
$R\Lambda$-modules.
\item[(ii)] Suppose $R \in \mathrm{Ob}(\hat{\mathcal{C}})$ and $(M^\bullet,\phi)$ is a proflat quasi-lift of 
$V^\bullet$ over $R$. Then ${\HH}^n(M^\bullet)$ is an abstractly free $R$-module of rank $d_n=
\mathrm{dim}_k\,{\HH}^n(V^\bullet)$ for all $n$.
Moreover, for any $R'\in\mathrm{Ob}(\hat{\mathcal{C}})$ and for any morphism 
$\alpha:R\to R'$ in $\hat{\mathcal{C}}$, there is a natural
$R'$-linear isomorphism $R'\hat{\otimes}_R {\HH}^n(M^\bullet) \cong {\HH}^n(R'\hat{\otimes}_R^{\LL} 
M^\bullet)$.
\end{enumerate}
\end{rem}

\begin{dfn}
\label{def:functordef}
Let $\hat{F} = \hat{F}_{V^\bullet}:\hat{\mathcal{C}} \to \mathrm{Sets}$
(resp. $\hat{F}^{\fl} = \hat{F}^{\fl}_{V^\bullet}:\hat{\mathcal{C}} \to \mathrm{Sets}$)
be the map that sends an object $R$ of $\hat{\mathcal{C}}$ to the set
$\hat{F}(R)$ (resp. $\hat{F}^{\fl}(R)$) of all deformations (resp. all proflat deformations)
of $V^\bullet$ over $R$, and that sends
a morphism $\alpha:R\to R'$ in $\hat{\mathcal{C}}$ to the set map
$\hat{F}(R)\to \hat{F}(R')$ (resp. $\hat{F}^{\fl}(R)\to \hat{F}^{\fl}(R')$)
given by $[M^\bullet,\phi] \mapsto [R'\hat{\otimes}_{R,\alpha}^{\LL}
M^\bullet,\phi_\alpha]$. Here $\phi_\alpha$ denotes the composition
$k\hat{\otimes}^{\LL}_{R'} (R'\hat{\otimes}_{R,\alpha}^{\LL} M^\bullet)
\cong k \hat{\otimes}^{\LL}_R M^\bullet \xrightarrow{\phi} V^\bullet$.
Let $F = F_{V^\bullet}$ (resp. $F^{\fl} = F^{\fl}_{V^\bullet}$)
be the restriction of $\hat{F}$ (resp. $\hat{F}^{\fl}$) to the subcategory $\mathcal{C}$
of Artinian rings in $\hat{\mathcal{C}}$.
In the following, we will use the subscript $\mathcal{D}$ to denote the empty
condition if we consider the map $\hat{F}$, and the condition of having
topologically free cohomology groups
if we consider the map $\hat{F}^{\fl}$. In particular, the notation $\hat{F}_{\mathcal{D}}$
will be used to refer to both $\hat{F}$ and $\hat{F}^{\fl}$.

Let $k[\varepsilon]$, where $\varepsilon^2=0$, denote the ring of dual numbers over
$k$. The set $F_{\mathcal{D}}(k[\varepsilon])$ is called the \emph{tangent space} to 
$F_{\mathcal{D}}$, denoted by $t_{F_{\mathcal{D}}}$. 
\end{dfn}

\begin{rem}
\label{rem:functor}
Let ${F}_{\mathcal{D}}$ and
$\hat{F}_{\mathcal{D}}$ be as in Definition $\ref{def:functordef}$.
Note that $\hat{F}_{\mathcal{D}}(R)$ is indeed a set for each object $R$ of $\hat{\mathcal{C}}$,
as can be seen, for example, by using the concepts of \cite[Sect. 3A]{magurn}.
Using similar arguments as in the proof of \cite[Prop. 2.12]{bcderived}, it follows that
the map $\hat{F}_{\mathcal{D}}$ is a functor $\hat{\mathcal{C}}\to\mathrm{Sets}$.
Moreover, $\hat{F}^{\fl}$ is a subfunctor of $\hat{F}$ in the sense that there is
a natural transformation $\hat{F}^{\fl}\to\hat{F}$ that is injective.
If ${V'}^\bullet$ is a complex in $D^-(\Lambda)$ satisfying Hypothesis 
$\ref{hypo:fincoh}$ such
that there is an isomorphism $\nu:V^\bullet\to {V'}^\bullet$ in $D^-(\Lambda)$, 
then the natural transformation
$\hat{F}_{\mathcal{D},V^\bullet}\to \hat{F}_{\mathcal{D},{V'}^\bullet}$ $($resp.
${F}_{\mathcal{D},V^\bullet}\to {F}_{\mathcal{D},{V'}^\bullet}${}$)$ induced by
$[M^\bullet,\phi]\mapsto [M^\bullet,\nu\circ\phi]$ is an isomorphism of functors.
\end{rem}

The following theorem is the main result of Section \ref{s:udr}.

\begin{thm}
\label{thm:bigthm} Assume Hypothesis $\ref{hypo:fincoh}$, and let ${F}_{\mathcal{D}}$ and
$\hat{F}_{\mathcal{D}}$ be as in Definition $\ref{def:functordef}$.
\begin{enumerate}
\item[(i)] 
The functor
$F_{\mathcal{D}}$ has a pro-representable hull 
$R_{\mathcal{D}}(\Lambda,V^\bullet)\in \mathrm{Ob}(\hat{\mathcal{C}})$ 
$($c.f. \cite[Def. 2.7]{Sch} and \cite[Sect. 1.2]{Maz}$)$, and 
the functor $\hat{F}_{\mathcal{D}}$ is continuous
$($c.f. \cite{Maz}$)$. 

\item[(ii)] 
If $F_{\mathcal{D}}=F$, then $t_F$ is a vector space over $k$ and
there is a $k$-vector space isomorphism $h: t_F \to
\mathrm{Ext}^1_{D^-(\Lambda)}(V^\bullet,V^\bullet)$.
If $F_{\mathcal{D}}=F^{\fl}$, then 
the composition of the natural map $t_{F^{\fl}}\to t_F$ and $h$ induces an isomorphism between
$t_{F^{\fl}}$ and
the kernel of the natural map $\mathrm{Ext}^1_{D^-(\Lambda)}(V^\bullet,V^\bullet)
\to \mathrm{Ext}^1_{D^-(k)}(V^\bullet,V^\bullet)$ given by forgetting the
$\Lambda$-action.

\item[(iii)]
If $\mathrm{Hom}_{D^-(\Lambda)}(V^\bullet,V^\bullet)= k$, then $\hat{F}_{\mathcal{D}}$ is represented
by $R_{\mathcal{D}}(\Lambda,V^\bullet)$. 
\end{enumerate}
\end{thm}

\begin{rem}
\label{rem:newrem}
By Theorem \ref{thm:bigthm}(i), there exists 
a quasi-lift $(U_{\mathcal{D}}(\Lambda,V^\bullet),\phi_U)$ of $V^\bullet$ over $R_{\mathcal{D}}(\Lambda,V^\bullet)$ 
with the following property. For each $R\in \mathrm{Ob}(\hat{\mathcal{C}})$, the map
$\mathrm{Hom}_{\hat{\mathcal{C}}}(R_{\mathcal{D}}(\Lambda,V^\bullet),R) \to \hat{F}_{\mathcal{D}}(R)$ 
given by $\alpha \mapsto [R\hat{\otimes}^{\LL}_{R_{\mathcal{D}}(\Lambda,V^\bullet),\alpha} U_{\mathcal{D}}(\Lambda,V^\bullet),\phi_{U,\alpha}]$ is surjective,
and this map is bijective if $R$ is the ring of dual numbers $k[\varepsilon]$ over $k$
where $\varepsilon^2=0$.

In general,
the isomorphism type of the pair consisting of the
pro-representable hull $R_{\mathcal{D}}(\Lambda,V^\bullet)$ 
and the deformation $[U_{\mathcal{D}}(\Lambda,V^\bullet),\phi_U]$ of $V^\bullet$ over $R_{\mathcal{D}}(\Lambda,V^\bullet)$ is
unique up to a non-canonical isomorphism.
If $R_{\mathcal{D}}(\Lambda,V^\bullet)$ represents $\hat{F}_{\mathcal{D}}$,
the pair $(R_{\mathcal{D}}(\Lambda,V^\bullet),[U_{\mathcal{D}}(\Lambda,V^\bullet),\phi_U])$ is uniquely determined up 
to a canonical isomorphism.
\end{rem}

\begin{dfn}
\label{def:newdef}
Using the notation of Theorem \ref{thm:bigthm} and Remark \ref{rem:newrem}, 
if $\hat{F}_{\mathcal{D}}=\hat{F}$ then we call 
$R_{\mathcal{D}}(\Lambda,V^\bullet)=R(\Lambda,V^\bullet)$
the \emph{versal deformation ring} of $V^\bullet$ and the deformation 
$[U(\Lambda,V^\bullet),\phi_U]$ is called the \emph{versal deformation} of $V^\bullet$.
If $\hat{F}_{\mathcal{D}}=\hat{F}^{\fl}$ then
we call $R_{\mathcal{D}}(\Lambda,V^\bullet)=R^{\fl}(\Lambda,V^\bullet)$ the 
\emph{versal proflat deformation ring} of $V^\bullet$ and 
the deformation $[U^{\fl}(\Lambda,V^\bullet),\phi_U]$ is called
the \emph{versal proflat deformation} of $V^\bullet$.

If $R_{\mathcal{D}}(\Lambda,V^\bullet)$ represents $\hat{F}_{\mathcal{D}}$, then
$R(\Lambda,V^\bullet)$ (resp. $R^{\fl}(\Lambda,V^\bullet)$) will be
called the \emph{universal deformation ring} (resp. \emph{the universal proflat deformation ring})
of $V^\bullet$, and the deformation
$[U(\Lambda,V^\bullet),\phi_U]$ (resp. $[U^{\fl}(\Lambda,V^\bullet),\phi_U]$) will be called the
\emph{universal deformation} (resp. the \emph{universal proflat deformation}) of $V^\bullet$.
\end{dfn}

\begin{rem}
\label{rem:bigthm}
\hspace*{1em}
\begin{itemize}
\item[(i)]
By part (ii) of Theorem \ref{thm:bigthm},
the tangent space $t_{F^{\fl}}$ consists of those elements
$$\gamma \in\mathrm{Ext}^1_{D^-(\Lambda)}(V^\bullet,V^\bullet)=\mathrm{Hom}_{D^-(\Lambda)}(
V^\bullet,V^\bullet[1])$$ that induce the trivial map on cohomology. In other words,
the $k$-vector space maps $\gamma^i:{\HH}^i(V^\bullet)\to {\HH}^{i+1}(V^\bullet)$ that are
induced by $\gamma$ have to be zero for all $i$.
\item[(ii)]
By part (i) of Theorem \ref{thm:bigthm}, there exists a non-canonical
continuous $k$-algebra homomorphism
$f_{\fl}: R(\Lambda,V^\bullet) \to R^{\fl}(\Lambda,V^\bullet)$. By part (ii) of  Theorem \ref{thm:bigthm},
it follows that the induced map between the Zariski cotangent spaces of these rings is surjective, implying that $f_{\fl}$
is surjective.
\item[(iii)]
If $V^\bullet$ consists of a single module $V^0$ in degree $0$,
the versal deformation ring $R(\Lambda,V^\bullet)$ is isomorphic to the versal deformation
ring $R(\Lambda,V^0)$ studied in \cite{blehervelez} (see Proposition \ref{prop:modulecase}).
\end{itemize}
\end{rem}

To prove Theorem \ref{thm:bigthm}, we adapt the argumentation in \cite{bcderived} to our situation.
Our first task is to adapt results from \cite[Sects.  3, 4 and 14]{bcderived} which
prove key properties of quasi-lifts of $V^\bullet$.

\subsection{Properties of quasi-lifts of $V^\bullet$}
\label{s:quasilifts}

In this subsection, 
we analyze the structure of quasi-lifts of $V^\bullet$ over Artinian rings $R$ in $\mathcal{C}$.
The following full subcategories of $C^{-}(R\Lambda)$, $K^{-}(R\Lambda)$ and 
$D^{-}(R\Lambda)$ play an important role in this situation.

\begin{dfn}
\label{dfn:fincategory}
Let $R \in \mathrm{Ob}(\mathcal {C})$ be Artinian.  Define $C^{-}_{\fin}(R\Lambda)$ 
(resp. $K^{-}_{\fin}(R\Lambda)$, resp. $D^{-}_{\fin}(R\Lambda)$) to
be the full subcategory of $C^{-}(R\Lambda)$ (resp. $K^{-}(R\Lambda)$, resp. 
$D^{-}(R\Lambda)$) whose objects are those complexes
$M^\bullet$ of finite pseudocompact $R$-tor dimension having finitely many 
non-zero cohomology groups, all of which have  finite $R$-length.  
\end{dfn}

\begin{rem}
\label{rem:whatever}
Suppose $R$ is an Artinian ring in $\mathrm{Ob}(\mathcal {C})$.
\begin{enumerate}
\item[(i)]
By Remark \ref{rem:pseudocompact}(i), an $R\Lambda$-module has finite
length if and only if it has finite length as an $R$-module.
Since $R$ is local Artinian, an $R$-module has finite $R$-length if 
and only if it has finite $k$-length. 
\item[(ii)] Suppose $N^\bullet$ is a complex in $D^-_{\fin}(R\Lambda)$ and
$X^\bullet$ is a complex in $D^-(R\Lambda)$ such that there is an isomorphism
$\xi:X^\bullet \to N^\bullet$ in $D^-(R\Lambda)$. Then $X^\bullet$ is an object of
$D^-_{\fin}(R\Lambda)$ and $\xi$ is an isomorphism in $D^-_{\fin}(R\Lambda)$. This follows
since having finite pseudocompact $R$-tor dimension is an invariant of isomorphisms in
$D^-(R\Lambda)$, and since such isomorphisms induce isomorphisms between
the cohomology groups. In particular, if $A^\bullet\to B^\bullet$ is a quasi-isomorphism
in $K^-(R\Lambda)$ and one of $A^\bullet$ or $B^\bullet$ is an object of $K^{-}_{\fin}(R\Lambda)$,
then so is the other.
\item[(iii)]
Let $N^\bullet,N_1^\bullet,N_2^\bullet$ be complexes in $D^-_{\fin}(R\Lambda)$ 
such that \textit{all their terms have finite $k$-length}, and let $g:N_1^\bullet\to N_2^\bullet$ be a
morphism in $D^{-}_{\fin}(R\Lambda)$. By Remark \ref{rem:pseudocompact}(ii), and 
since  $R\Lambda$ is Artinian,
there exist  bounded above complexes $M^\bullet$, $M_1^\bullet$ and 
$M_2^\bullet$ of abstractly free finitely generated $R\Lambda$-modules such that there
are isomorphisms $\beta:N^\bullet\to M^\bullet$ and
$\beta_i: N_i^\bullet\to M_i^\bullet$ in $D^{-}_{\fin}(R\Lambda)$ ($i=1,2$). 
Then $f=\beta_2\circ g\circ\beta_1^{-1}$ is a 
morphism $f:M_1^\bullet \to M_2^\bullet$ in $D^{-}_{\fin}(R\Lambda)$.
Let $\mathcal{P}_R$ be the additive subcategory of $\mathrm{PCMod}(R\Lambda)$ of 
projective objects.
By  \cite[Thm. 10.4.8]{Weibel}, the natural functor $K^-(\mathcal{P}_R)
\to D^-(R\Lambda)$ is an equivalence of categories. Hence $f$
can be taken to be a morphism in $K^{-}_{\fin}(R\Lambda)$. 
\end{enumerate}
\end{rem}

The following two results, Lemmas \ref{lem:corollary3.6} and \ref{lem:lemma3.8},
establish key properties of objects and morphisms in $D^{-}_{\fin}(R\Lambda)$.
They replace \cite[Cor. 3.6 and Lemma 3.8]{bcderived} in our situation. Since
$\Lambda$ is Artinian, some of the statements can be simplified.

\begin{lemma}
\label{lem:corollary3.6}
Suppose $R \in \mathrm{Ob}(\mathcal {C})$ is Artinian, and $N^\bullet$, 
$N_1^\bullet$ and $N_2^\bullet$ are objects in $D^{-}_{\fin}(R\Lambda)$. 
Let $g:N_1^\bullet \to 
N_2^\bullet$ be a morphism in $D^{-}_{\fin}(R\Lambda)$. 
\begin{enumerate}
\item[(i)]  There exists a bounded above complex $M^\bullet$ of abstractly free finitely generated 
$R\Lambda$-modules, and an isomorphism $\beta:N^\bullet \to M^\bullet$ in 
$D^{-}_{\fin}(R\Lambda)$.

\item[(ii)]  There exist bounded above complexes $M_1^\bullet$ and $M_2^\bullet$ of 
abstractly free finitely generated $R\Lambda$-modules, a morphism
$f:M_1^\bullet \to M_2^\bullet$ in $K^-_{\fin}(R\Lambda)$, and
isomorphisms $\beta_i:N_i^\bullet \to M_i^\bullet$ in
$D^{-}_{\fin}(R\Lambda)$ $(i=1,2)$ such that $f = \beta_{2}\circ g \circ
\beta_1^{-1}$ as morphisms in $D^{-}_{\fin}(R\Lambda)$. 
\end{enumerate}
\end{lemma}

\begin{proof}
In view of Remark \ref{rem:whatever}(iii), the main ingredient in the proof is the following claim,
which is proved using similar arguments  as
in the proof of \cite[Lemma 3.4(i)]{bcderived}.

\medskip

\noindent
\textit{Claim $1$.} 
Suppose  $N^\bullet$ is an object of  $D^{-}_{\fin}(R\Lambda)$ satisfying
${\HH}^j(N^\bullet)=0$ for $j<n$. Then there exists an exact sequence of complexes
\begin{equation}
\label{eq:cutting}
0 \to U^\bullet \xrightarrow{\iota} N^\bullet \to  {N'}^\bullet \to 0
\end{equation}
in $C^{-}_{\fin}(R\Lambda)$ such that $U^\bullet$ is acyclic, and such that the terms of ${N'}^\bullet$
have finite $k$-length and satisfy ${N'}^j=0$ for $j<n$. 

\medskip

Part (i) of Lemma \ref{lem:corollary3.6} now follows from Claim 1 and Remark \ref{rem:whatever}(iii).
To prove part(ii) of Lemma \ref{lem:corollary3.6}, we use Claim 1 to see that
there exist bounded complexes
${N_1'}^\bullet$ and ${N_2'}^\bullet$ such that all their terms have finite $k$-length, together
with quasi-isomorphisms $\gamma_i:N_i^\bullet \to {N'_i}^\bullet$ in $C^{-}_{\fin}(R\Lambda)$
($i=1,2$) that are surjective on terms. Let $g'=\gamma_2 \circ g \circ \gamma_1^{-1}:{N'_1}^\bullet
\to {N'_2}^\bullet$, so $g'$ is a morphism in $D^{-}_{\fin}(R\Lambda)$.
Using Remark \ref{rem:whatever}(iii),
there exist bounded above complexes $M_1^\bullet$ and $M_2^\bullet$ of 
abstractly free finitely generated $R\Lambda$-modules, a morphism
$f:M_1^\bullet \to M_2^\bullet$ in $K^-_{\fin}(R\Lambda)$, and
isomorphisms $\beta'_i:{N_i'}^\bullet \to M_i^\bullet$ in
$D^{-}_{\fin}(R\Lambda)$ $(i=1,2)$ such that $f = \beta'_{2}\circ g'\circ
{\beta'_1}^{-1}$ as morphisms in $D^{-}_{\fin}(R\Lambda)$. Letting $\beta_i=\beta_i'\circ\gamma_i$
($i=1,2$), part (ii) follows.
\end{proof}

\begin{dfn}
\label{def:Definition3.7}
In the situation of Lemma \ref{lem:corollary3.6}(i), we say \emph{we can replace $N^\bullet$
by  $M^\bullet$}. In the situation of Lemma \ref{lem:corollary3.6}(ii), we say
\emph{we can replace $N_i^\bullet$ by $M_i^\bullet$ $(i=1,2)$, and 
$g$ by $f$}.
\end{dfn}

\begin{lemma}
\label{lem:lemma3.8}
Suppose $R \in \mathrm{Ob}(\mathcal {C})$ is Artinian, and 
$M^\bullet$ is an object of $D^{-}_{\fin}(R\Lambda)$ such that $\HH^j(M^\bullet)=0$
for $j<n$. Then $M^\bullet$ has finite pseudocompact $R$-tor dimension at $n$.
\end{lemma}

\begin{proof}
By Lemma \ref{lem:corollary3.6}(i), we may assume that $M^\bullet$ is a bounded above
complex of abstractly free finitely generated $R\Lambda$-modules. Hence all terms
of $M^\bullet$ are abstractly free finitely generated $R$-modules.
By Remark \ref{rem:dumbdumb}, there exists an integer $n_1\leq n$ such that $M^{n_1}/
\delta^{n_1-1}(M^{n_1-1})$ is a topologically free pseudocompact  $R$-module. Since $M^{n_1}$
is a finitely generated $R$-module, it follows that this is an abstractly free finitely generated $R$-module. 
To prove Lemma \ref{lem:lemma3.8}, it is enough to show that  $M^{n}/\delta^{n-1}(M^{n-1})$ is an 
abstractly free finitely generated $R$-module. 
This is proved exactly in the same way as in the proof of 
\cite[Lemma 3.8]{bcderived}.
\end{proof}

\begin{rem}
\label{rem:stupid}
Suppose $R \in \mathrm{Ob}(\mathcal {C})$ is Artinian.
Using Lemma \ref{lem:corollary3.6}(i) together with Remark \ref{rem:dumbdumb} and Lemma \ref{lem:lemma3.8},
it follows that every object $M^\bullet$ of $D^{-}_{\fin}(R\Lambda)$ is isomorphic in $D^{-}_{\fin}(R\Lambda)$
to a bounded complex ${M'}^\bullet$ such that all terms of ${M'}^\bullet$ are abstractly free finitely generated 
$R$-modules and such that all its terms, except possibly its leftmost non-zero term, 
are actually abstractly free finitely generated $R\Lambda$-modules.
\end{rem}

The following remark replaces  \cite[Cors. 3.10 and 3.11]{bcderived} in our situation. Note that since
$\Lambda$ is already Artinian, we do not need to deal with quotient algebras of $\Lambda$.

\begin{rem}
\label{rem:corollary3.1011}
Suppose $R,S,T\in\mathrm{Ob}(\mathcal{C})$ are Artinian rings with morphisms 
$R\xrightarrow{\alpha} T \xleftarrow{\beta} S$ in $\mathcal{C}$. 
Let $X^\bullet$ be an object of $D^{-}_{\fin}(R\Lambda)$, and let $Z^\bullet$ be an object of
$D^{-}_{\fin}(S\Lambda)$. 
Suppose $\tau:T\hat{\otimes}^{\LL}_{S} Z^\bullet \to T\hat{\otimes}^{\LL}_{R} 
X^\bullet$ is a morphism in $D^-_{\fin}(T\Lambda)$. By Remark \ref{rem:leftderivedtensor},
$$ T\hat{\otimes}^{\LL}_{R} X^\bullet= T\hat{\otimes}_R\,\sigma_R(X^\bullet)
\quad\mbox{and}\quad
T\hat{\otimes}^{\LL}_{S} Z^\bullet =T\hat{\otimes}_S\,\sigma_S(Z^\bullet)$$
where $\sigma_R(X^\bullet)$ (resp. $\sigma_S(Z^\bullet)$) is an object of 
$K^-(\mathcal{P}_R)$ (resp. $K^-(\mathcal{P}_S)$) and there exists an isomorphism
$\rho_X:X^\bullet\to \sigma_R(X^\bullet)$ in $D^-(R\Lambda)$
(resp. $\rho_Z:Z^\bullet\to\sigma_S(Z^\bullet)$ in $D^-(S\Lambda)$).
By Remark \ref{rem:whatever}(ii), $\sigma_R(X^\bullet)$ and $\rho_X$ (resp. $\sigma_S(Z^\bullet)$
and $\rho_Z$) are in $D^-_{\fin}(R\Lambda)$ (resp. $D^-_{\fin}(S\Lambda)$).

By Lemma \ref{lem:corollary3.6}(i), there exists a bounded above complex $\tilde{X}^\bullet$
(resp. $\tilde{Z}^\bullet$) of abstractly free finitely generated $R\Lambda$-modules
(resp. $S\Lambda$-modules) and an isomorphism 
$\tilde{\beta}: \tilde{X}^\bullet\to X^\bullet$ in $D^{-}_{\fin}(R\Lambda)$
(resp. $\tilde{\gamma}: \tilde{Z}^\bullet\to Z^\bullet$ in $D^{-}_{\fin}(S\Lambda)$). 
Let $\beta=\rho_X\circ\tilde{\beta}$ and $\gamma=\rho_Z\circ\tilde{\gamma}$.
Then $\beta:\tilde{X}^\bullet\to \sigma_R(X^\bullet)$ (resp. 
$\gamma:\tilde{Z}^\bullet\to\sigma_S(Z^\bullet)$) is an isomorphism in $D^{-}_{\fin}(R\Lambda)$
(resp. $D^{-}_{\fin}(S\Lambda)$). Since $\tilde{X}^\bullet,\sigma_R(X^\bullet)$ (resp.
$\tilde{Z}^\bullet,\sigma_S(Z^\bullet)$) are objects in $K^-(\mathcal{P}_R)$ (resp. $K^-(\mathcal{P}_S)$),
$\beta$ (resp. $\gamma$) can be taken to be a morphism in $K^{-}_{\fin}(R\Lambda)$
(resp. $K^{-}_{\fin}(S\Lambda)$). Moreover, $T\hat{\otimes}_R\beta=T\hat{\otimes}^{\LL}_{R} \beta$ 
(resp. $T\hat{\otimes}_S\gamma=T\hat{\otimes}^{\LL}_{S} \gamma$)
is an isomorphism in $K^{-}_{\fin}(T\Lambda)$, and $\tau:T\hat{\otimes}_S\,\sigma_S(Z^\bullet)
\to T\hat{\otimes}_R\,\sigma_R(X^\bullet)$ can be taken to be a morphism in $K^{-}_{\fin}(T\Lambda)$.
Define
\begin{equation}
\label{eq:replace}
\tilde{\tau}=(T\hat{\otimes}_R\beta)^{-1}\circ\tau\circ(T\hat{\otimes}_S\gamma):\quad
T\hat{\otimes}_S\tilde{Z}^\bullet \to T\hat{\otimes}_R\tilde{X}^\bullet .
\end{equation}
Then we can replace $X^\bullet$ (resp. $Z^\bullet$) by $\tilde{X}^\bullet$ (resp. $\tilde{Z}^\bullet$),
in the sense of Definition \ref{def:Definition3.7}, and we can replace $\tau$ by $\tilde{\tau}$.
Note that $\beta$ (resp. $\gamma$) only depends on $\tilde{X}^\bullet$ and 
$\tilde{\beta}$ (resp. $\tilde{Z}^\bullet$ and $\tilde{\gamma}$).
\end{rem}

Using Remark \ref{rem:profree}(i), the following result is proved in a similar way to \cite[Lemma 3.1]{bcderived}.

\begin{lemma}
\label{lem:lemma3.1}
Suppose $(M^\bullet,\phi)$ is a quasi-lift of $V^\bullet$ over some Artinian ring
$R \in \mathrm{Ob}(\mathcal{C})$. Then $M^\bullet$ is an object of $D^{-}_{\fin}(R\Lambda)$. 
More precisely, ${\HH}^i(M^\bullet)$ is a subquotient of  an abstractly
free $R$-module of rank $d_i=\mathrm{dim}_k\,{\HH}^i(V^\bullet)$ for all $i$.
\end{lemma}

The following result  summarizes the main properties of quasi-lifts and proflat quasi-lifts
of $V^\bullet$ over arbitrary objects $R$ in $\hat{\mathcal{C}}$. 
The proof is very similar to the proof of \cite[Thm. 2.10]{obstructions},
once we replace the results from \cite{bcderived} by the corresponding results stated above.
For the convenience of the reader, we provide some of the details.

\begin{thm} 
\label{thm:derivedresult}  
Suppose that $\HH^i(V^\bullet) = 0$ unless $n_1 \le i \le n_2$.  Every quasi-lift of $V^\bullet$ over 
an object $R$ of $\hat{\mathcal{C}}$ is isomorphic to a quasi-lift $(P^\bullet, \psi)$ for
a complex $P^\bullet$ with the following properties:
\begin{enumerate}
\item[(i)] The terms of $P^\bullet$ are topologically free $R\Lambda$-modules.
\item[(ii)] The cohomology group $\HH^i(P^\bullet)$ is finitely generated 
as an abstract $R$-module for all $i$, and $\HH^i(P^\bullet) = 0$ unless $n_1 \le i \le n_2$. 
\item[(iii)]   One has $\HH^i(S\hat{\otimes}^{\LL}_RP^\bullet)=0$  for all pseudocompact $R$-modules 
$S$ unless $n_1 \le i \le n_2$.
\end{enumerate}
\end{thm}

\begin{proof}
Let $R$ be an object of $\hat{\mathcal{C}}$, and let $(M^\bullet,\phi)$ be a quasi-lift of $V^\bullet$ over $R$.
By Remark \ref{rem:profree}(i), it follows that there exists a quasi-lift $(P^\bullet,\psi)$ of $V^\bullet$
over $R$ that is isomorphic to the quasi-lift $(M^\bullet,\phi)$ such that $P^\bullet$ satisfies property (i).
It remains to verify properties (ii) and (iii). By (i), we can assume that the terms of $P^\bullet$ 
are topologically free pseudocompact $R\Lambda$-modules. In particular, the functors 
$-\hat{\otimes}^{\LL}_RP^\bullet$ and $-\hat{\otimes}_RP^\bullet$ are
naturally isomorphic. Let $m_R$ denote the maximal ideal of $R$, and let $n$ be an
arbitrary positive integer.
By Lemmas \ref{lem:lemma3.8} and \ref{lem:lemma3.1}, $\HH^i((R/m_R^n)\hat{\otimes}_RP^\bullet)=0$ for $i>n_2$ 
and $i<n_1$. Moreover, for $n_1\le i\le n_2$,
$\HH^i((R/m_R^n)\hat{\otimes}_RP^\bullet)$ is a subquotient of  an abstractly free $(R/m_R^n)$-module
of rank $d_i=\mathrm{dim}_k\,\HH^i(V^\bullet)$,
and $(R/m_R^n)\hat{\otimes}_RP^\bullet$ has finite pseudocompact $(R/m_R^n)$-tor dimension at
$N=n_1$. Since 
$P^\bullet\cong \displaystyle \lim_{\stackrel{\longleftarrow}{n}}\, (R/m_R^n)\hat{\otimes}_RP^\bullet$ 
and since by Remark \ref{rem:pseudocompact}(i), the category $\mathrm{PCMod}(R)$ has 
exact inverse limits, it  follows that for all pseudocompact $R$-modules $S$
$$\HH^i(S\hat{\otimes}_RP^\bullet)= \lim_{\stackrel{\longleftarrow}{n}}\,\HH^i\left(
(S/m_R^nS)\hat{\otimes}_{R/m_R^n}\left((R/m_R^n)\hat{\otimes}_R P^\bullet\right)\right)$$ 
for all $i$. Hence Theorem \ref{thm:derivedresult} follows.
\end{proof}

\subsection{More results on complexes and quasi-lifts}
\label{s:prelims}

In this subsection, we first provide some results from Milne \cite{milne}, which
are adapted from \cite[Sect. 14]{bcderived} to our situation.
Note that in  \cite[Lemma VI.8.17]{milne} (resp.
\cite[Lemma VI.8.18]{milne}), the condition ``$\pi$ is surjective on
terms'' (resp. ``$\psi$ is surjective on terms'') is necessary in the statement.

\begin{lemma}
\label{lem:Lemma14.1}
{\rm (\cite[Lemma VI.8.17]{milne})}
Let $R\in\mathrm{Ob}(\hat{\mathcal{C}})$, and
let $M^\bullet \xrightarrow{\phi} L^\bullet \xleftarrow{\pi} N^\bullet$
be morphisms in $C^-(R\Lambda)$ such
that $\pi$ is a quasi-isomorphism that is surjective on terms.
If $M^\bullet$ is a complex of topologically free pseudocompact $R\Lambda$-modules,
there exists a morphism $\psi:M^\bullet\to N^\bullet$ in $C^-(R\Lambda)$ such that $\pi\circ\psi=\phi$.
\end{lemma}

\begin{rem}
\label{rem:Remark14.2}
Suppose $R,R_0\in\mathrm{Ob}(\hat{\mathcal{C}})$
such that $R_0$ is a quotient ring
of $R$.  We write $X\to X_0$, $\phi\to \phi_0$ for the functor
$R_0\hat{\otimes}_R -$.
Suppose $M$, $N$ are
topologically free pseudocompact $R\Lambda$-modules.
Then every continuous $R_0\Lambda$-module homomorphism $\pi:M_0\to N_0$
can be lifted to a continuous $R\Lambda$-module homomorphism $\phi:M\to N$
such that $\pi=\phi_0$.
\end{rem}

\begin{lemma}
\label{lem:Lemma14.3}
{\rm (\cite[Sublemma VI.8.20]{milne})}
Let $R,R_0\in\mathrm{Ob}(\hat{\mathcal{C}})$ such that
$R_0$ is a quotient ring of $R$. As in Remark $\ref{rem:Remark14.2}$,
we write $X\to X_0$, $\phi\to
\phi_0$ for the functor $R_0\hat{\otimes}_R -$.
Let $\phi:L^\bullet\to M^\bullet$ be a
morphism in $C^-(R\Lambda)$ of complexes of topologically free pseudocompact
$R\Lambda$-modules. Then any morphism $L_0^\bullet\to M_0^\bullet$ 
in $C^-(R_0\Lambda)$ that is homotopic to
$\phi_0$ is of the form $\psi_0$, where $\psi:L^\bullet\to M^\bullet$ is a morphism
in $C^-(R\Lambda)$ that is homotopic to $\phi$.
\end{lemma}

\begin{lemma}
\label{lem:Lemma14.4}
{\rm (\cite[Lemma VI.8.18]{milne})}
Let $R,R_0\in\mathrm{Ob}(\mathcal{C})$ be Artinian such that $R_0$ is a quotient ring of $R$.
As in Remark $\ref{rem:Remark14.2}$, we write $X\to X_0$, $\phi\to
\phi_0$ for the functor $R_0\hat{\otimes}_R -$. Let $M^\bullet$ $($resp. $N^\bullet${}$)$ be
a {bounded above} complex of abstractly free finitely generated  $R\Lambda$-modules $($resp.
$R_0\Lambda$-modules$)$, and let $\psi$ be
a quasi-isomorphism $\psi:M_0^\bullet {\rightarrow}N^\bullet$ in $C^-(R_0\Lambda)$
that is surjective on terms. Then there
exist a {bounded above} complex $L^\bullet$ of {abstractly free finitely generated}
$R\Lambda$-modules,
a quasi-isomorphism
$\phi:M^\bullet\to L^\bullet$ in $C^-(R\Lambda)$, and an isomorphism 
$\rho:L^\bullet_0 {\rightarrow}
N^\bullet$ in $C^-(R_0\Lambda)$, such that $\rho\circ \phi_0=\psi$.
\end{lemma}

The following remark (which replaces \cite[Remark 5.2]{bcderived} in our situation)
shows how one can relate a morphism $f$ in $C^-(R\Lambda)$ to 
a morphism $g$ in $C^-(R\Lambda)$ that is surjective on terms.

\begin{rem}
\label{rem:Remark5.2}
Suppose $R\in\mathrm{Ob}(\hat{\mathcal{C}})$. Let $M^\bullet$ and $N^\bullet$ be
two bounded above complexes of pseudocompact $R\Lambda$-modules, and
let $f:M^\bullet\to N^\bullet$ be a morphism in $C^-(R\Lambda)$. 
Let $P^\bullet$ be a bounded above complex  of topologically free pseudocompact
$R\Lambda$-modules such that there is a quasi-isomorphism
$P^\bullet \to N^\bullet$ in $C^-(R\Lambda)$
that is surjective on terms. Then
the mapping cone $C^\bullet$ of  $P^\bullet[-1] \xrightarrow{\mathrm{id}} P^\bullet[-1]$
is an acyclic complex, and there is a morphism $\pi: C^\bullet\to N^\bullet$ 
in $C^-(R\Lambda)$ that is
surjective on terms. Define $g:M^\bullet\oplus C^\bullet\to N^\bullet$ by $g=(f,\pi)$, and 
define $s:M^\bullet\to M^\bullet\oplus C^\bullet$ by $s=\left(\begin{array}{c}\mathrm{id}_{
M^\bullet}\\0
\end{array}\right)$. Then $g$ is surjective on terms, $s$ is a quasi-isomorphism and $g\circ s=f$.

Suppose there is a surjective morphism $R_1\to R$ in $\hat{\mathcal{C}}$, 
and there is a bounded above complex
$X^\bullet$ of topologically free pseudocompact $R_1\Lambda$-modules such that $M^\bullet=
R\hat{\otimes}_{R_1}X^\bullet$. Since $R\Lambda=R\hat{\otimes}_{R_1}R_1\Lambda$, there exists
a bounded above complex $Q^\bullet$ of topologically free pseudocompact
$R_1\Lambda$-modules with
$P^\bullet=R\hat{\otimes}_{R_1}Q^\bullet$. Hence $C^\bullet=R\hat{\otimes}_{R_1}D^\bullet$,
where $D^\bullet$ is the mapping cone of $Q^\bullet[-1] \xrightarrow{\mathrm{id}} Q^\bullet[-1]$, and $M^\bullet\oplus C^\bullet=R\hat{\otimes}_{R_1}(X^\bullet\oplus Q^\bullet)$.
\end{rem}

Next, we look at quasi-lifts of $V^\bullet$ in the case when  the endomorphism ring of $V^\bullet$ 
in $D^-(\Lambda)$ is isomorphic to $k$. We  need the following remark.

\begin{rem}
\label{rem:Lemma4.2}
Define $\hat{F}_1$ (resp. $F_1$) to be the functor from
$\hat{\mathcal{C}}$ (resp. $\mathcal{C}$) to the category $\mathrm{Sets}$
that sends $R$ to the set $\hat{F}_1(R)$ (resp.
$F_1(R)$) of isomorphism classes of quasi-lifts of $V^\bullet$
over $R$ that are represented by bounded  above complexes of topologically free pseudocompact $R\Lambda$-modules.

Using Remark \ref{rem:profree}(i), it follows as in the proof of \cite[Lemma 4.2]{bcderived} that the 
natural transformation $\hat{F}_1\to \hat{F}$ $($resp.
$F_1 \to F${}$)$ is an isomorphism of functors.
\end{rem}

Using Remark \ref{rem:Lemma4.2} and Lemma \ref{lem:Lemma14.3},
the following result is proved in a similar way to \cite[Prop. 4.3]{bcderived}.

\begin{prop}
\label{prop:liftendos}
Suppose $\mathrm{Hom}_{D^-(\Lambda)}(V^\bullet,V^\bullet)= k$.
Then $\mathrm{Hom}_{D^-(R\Lambda)}(M^\bullet,M^\bullet)= R$
for every quasi-lift $(M^\bullet,\phi)$ of $V^\bullet$ over an Artinian ring
$R\in\mathrm{Ob}(\mathcal{C})$.
\end{prop}

\subsection{Proof of Theorem \ref{thm:bigthm}}
\label{s:proofbigthm}

In this subsection, we prove Theorem \ref{thm:bigthm}. We follow the argumentation in
Sections 5 through 7 of \cite{bcderived} and explain how the key steps are adapted to our situation.
As in \cite{bcderived}, we use Schlessinger's criteria $\mathrm{(H1)}$ - $\mathrm{(H4)}$ for the pro-representability,
respectively for the existence of a pro-representable hull, of a functor of Artinian rings.
We refer to \cite[Thm. 2.11]{Sch} for a precise description of these criteria $\mathrm{(H1)}$ - $\mathrm{(H4)}$.

\begin{prop}
\label{prop:step1}
Schlessinger's criteria $\mathrm{(H1)}$ and $\mathrm{(H2)}$
are always satisfied for $F_{\mathcal{D}}$. In the case when 
$\mathrm{Hom}_{D^-(\Lambda)}(V^\bullet,V^\bullet)= k$,
$\mathrm{(H4)}$ is also satisfied.
\end{prop}

\begin{proof}
The proof of Proposition \ref{prop:step1} closely follows the proof of \cite[Prop. 5.1]{bcderived}.
By Remark \ref{rem:functor} and Lemma \ref{lem:corollary3.6}(i),
we may assume, without loss of generality,
that $V^\bullet$ is a bounded above complex of abstractly free finitely generated 
$\Lambda$-modules.
Suppose $A,B,C$ are Artinian rings in $\mathrm{Ob}(\mathcal{C})$ and that we have a diagram in $\mathcal{C}$
$$
\xymatrix @-1pc {
A\ar[dr]_{\alpha}&&B\ar[dl]^{\beta}\\
&C&
}$$
Let $D$ be the pullback $\displaystyle D=A\times_C B = \{(a,b)\in A\times B\; \big| \;
\alpha(a)=\beta(b)\}$.
Consider the natural map
$$ \chi_{\mathcal{D}}: F_{\mathcal{D}}(D) \to F_{\mathcal{D}}(A)\times_{F_{\mathcal{D}}(C)} F_{\mathcal{D}}(B) .$$

\medskip

\noindent
\textit{Claim $1$.} If $\beta$ is surjective, then $\chi_{\mathcal{D}}$ is surjective. In particular, this implies that $\mathrm{(H1)}$ 
is always satisfied for $F_{\mathcal{D}}$.

\medskip

To prove Claim 1, suppose $[X_A^\bullet,\xi_A]\in F_{\mathcal{D}}(A)$ and $[X_B^\bullet,\xi_B]\in F_{\mathcal{D}}(B)$
such that there exists an
isomorphism $$\tau:C\hat{\otimes}^{\LL}_B X_B^\bullet
\to C\hat{\otimes}^{\LL}_A X_A^\bullet$$ in $D^-(C\Lambda)$ with $\xi_A\circ(k\hat{\otimes}^{\LL}_C\tau)
=\xi_B$. 
By Remark \ref{rem:corollary3.1011}, we can assume
the following.
The complex $X_A^\bullet$ (resp. $X_B^\bullet$) is a bounded above complex of 
abstractly free finitely generated $A\Lambda$-modules (resp. $B\Lambda$-modules),
$\tau$ is given by a quasi-isomorphism in $C^-(C\Lambda)$, and $\xi_A$ 
(resp. $\xi_B$) is given by a quasi-isomorphism in $C^-(\Lambda)$. By Remark \ref{rem:Remark5.2}, 
we can add to $X_B^\bullet$ an acyclic complex of abstractly free 
finitely generated $B\Lambda$-modules to be able to assume that $\tau$ is surjective 
on terms.
We can now complete the proof of Claim 1 
in a similar way to the proof of \cite[Lemma 5.3]{bcderived}.

\medskip

\noindent
\textit{Claim $2$.} If $\beta$ is surjective, and either
$\mathrm{Hom}_{D^-(\Lambda)}(V^\bullet,V^\bullet)= k$
or $C=k$, then $\chi_{\mathcal{D}}$ is injective. In particular, this implies that $\mathrm{(H2)}$ is always satisfied  for $F_{\mathcal{D}}$,
and $\mathrm{(H4)}$ is satisfied if $\mathrm{Hom}_{D^-(\Lambda)}(V^\bullet,V^\bullet)= k$.

\medskip

Since $F^{\fl}$ is a subfunctor of $F$ by Remark \ref{rem:functor}, it
is enough to prove Claim 2 in the case when $F_{\mathcal{D}}=F$.
Suppose $[X_D^\bullet,\xi]$ and $[Z_D^\bullet,\zeta]$ are two elements in $F(D)$
such that there is an isomorphism
$\tau_A:A\hat{\otimes}^{\LL}_D Z_D^\bullet\to A\hat{\otimes}^{\LL}_D X_D^\bullet$
(resp. $\tau_B:B\hat{\otimes}^{\LL}_D Z_D^\bullet\to B\hat{\otimes}^{\LL}_D X_D^\bullet$) in
$D^-(A\Lambda)$
(resp. $D^-(B\Lambda)$) with $\xi\circ (k\hat{\otimes}^{\LL}_A \tau_A) = \zeta$
(resp. $\xi\circ (k\hat{\otimes}^{\LL}_B \tau_B) = \zeta$)
in $D^-(\Lambda)$. In other words
$(A\hat{\otimes}^{\LL}_D Z_D^\bullet,\zeta)$ and $(A\hat{\otimes}^{\LL}_D X_D^\bullet,\xi)$
(resp. $(B\hat{\otimes}^{\LL}_D Z_D^\bullet,\zeta)$ and $(B\hat{\otimes}^{\LL}_D X_D^\bullet,\xi)$)
are isomorphic as quasi-lifts of $V^\bullet$ over $A$ (resp. $B$).
Consider $\varphi_C:C\hat{\otimes}^{\LL}_D Z_D^\bullet \to C\hat{\otimes}^{\LL}_D Z_D^\bullet$
in $D^-(C\Lambda)$, defined by $\varphi_C=(C\hat{\otimes}^{\LL}_A\tau_A)^{-1}\circ
(C\hat{\otimes}^{\LL}_B\tau_B)$. If $C=k$, then 
$$\varphi_k=(k\hat{\otimes}^{\LL}_A\tau_A)^{-1}\circ(k\hat{\otimes}^{\LL}_B\tau_B)=
(\zeta^{-1}\circ\xi)\circ(\xi^{-1}\circ\zeta)=\mathrm{id}_{k\hat{\otimes}^{\LL}_D Z_D^\bullet}$$
in $D^-(\Lambda)$. If $\mathrm{Hom}_{D^-(\Lambda)}(V^\bullet,V^\bullet)= k$, then, by Proposition
\ref{prop:liftendos}, 
$$\mathrm{Hom}_{D^-(C\Lambda)}(C\hat{\otimes}^{\LL}_D Z_D^\bullet,
C\hat{\otimes}^{\LL}_D Z_D^\bullet)=C.$$ 
Hence, in either case there exists a unit $\alpha_C\in C$, with image $1$ in $k$,
such that $\varphi_C$ is multiplication by $\alpha_C$ in $D^-(C\Lambda)$.
By Remark \ref{rem:corollary3.1011}, we can assume the following.
\begin{enumerate}
\item[(i)] The complexes $X_D^\bullet$ and $Z_D^\bullet$ are bounded above complexes of 
abstractly free finitely generated $D\Lambda$-modules,
\item[(ii)] The morphisms $\xi$ and $\zeta$ are given by quasi-isomorphisms 
$\xi:k\otimes_D X_D^\bullet\to V^\bullet$ and $\zeta:k\otimes_D
Z_D^\bullet\to V^\bullet$ in $C^-(\Lambda)$.
\item[(iii)] The morphism $\tau_A$ (resp. $\tau_B$) is given by a
quasi-isomorphism $\tau_A:A\otimes_D Z_D^\bullet \to A\otimes_D X_D^\bullet$
(resp. $\tau_B:B\otimes_D Z_D^\bullet \to B \otimes_D X_D^\bullet$)
in $C^-(A\Lambda)$ (resp. in $C^-(B\Lambda)$) such that
$\xi\circ(k\otimes_A \tau_A) =\zeta$ (resp. $\xi\circ(k\otimes_B \tau_B)=\zeta$)
in $K^-(\Lambda)$.
\end{enumerate}
Since $\beta$ is surjective, $D\to A$ is also surjective.
By Remark \ref{rem:Remark5.2}, we can add to $Z_D^\bullet$ an acyclic complex of abstractly free 
finitely generated $D\Lambda$-modules to be able to assume that $\tau_A$ is surjective 
on terms. 
We can now complete the proof of Claim 2 in a similar way to the proof of \cite[Lemma 5.4]{bcderived}.
\end{proof}

\begin{prop}
\label{prop:step2}
Let $k[\varepsilon]$ be the ring of dual numbers over $k$
where $\varepsilon^2=0$.
\begin{enumerate}
\item[(i)] The tangent space $t_F$ is a vector space over $k$, and
there is a $k$-vector space isomorphism 
$$h: \;t_F = F(k[\varepsilon]) \longrightarrow \mathrm{Hom}_{D^-(\Lambda)}(V^\bullet,V^\bullet[1])=
\mathrm{Ext}^1_{D^-(\Lambda)}(V^\bullet, V^\bullet).$$
\item[(ii)]
Composing the natural injection $t_{F^{\fl}}\to t_F$ with $h$, we obtain an isomorphism between 
$t_{F^{\fl}}$ and the kernel of the natural forgetful map
\begin{equation}
\label{eq:forget}
\mathrm{Ext}^1_{D^-(\Lambda)}(V^\bullet,V^\bullet)
\to \mathrm{Ext}^1_{D^-(k)}(V^\bullet,V^\bullet).
\end{equation}
\end{enumerate}
\end{prop}

\begin{proof}
The proof of part (i) closely follows the proof of  \cite[Lemma 6.1]{bcderived}.
By Remark \ref{rem:functor}, we may assume that $V^\bullet$ is a bounded above complex of 
topologically free pseudocompact $\Lambda$-modules.
Suppose $(M^\bullet,\phi)$ is a quasi-lift of $V^\bullet$ over $k[\varepsilon]$.
By Remark \ref{rem:Lemma4.2}, we can assume that $M^\bullet$
is a bounded above complex of topologically free pseudocompact $k[\varepsilon]\Lambda$-modules.
We have a short exact sequence 
$$0\to \varepsilon M^\bullet \xrightarrow{\iota}{} M^\bullet \xrightarrow{\pi}{}
M^\bullet/\varepsilon M^\bullet \to 0 $$
in $C^-(k[\varepsilon]\Lambda)$.
The mapping cone of $\iota$ is $C(\iota)^\bullet=\varepsilon M^\bullet[1]\oplus
M^\bullet$ with $i$-th differential
$$\delta_{C(\iota)}^i=\left(\begin{array}{cc} -\delta_M^{i+1}&0\\
\iota^{i+1}&\delta_M^i\end{array}\right)  .$$
We obtain a distinguished triangle in $K^-(k[\varepsilon]\Lambda)$
\begin{equation}
\label{eq:hyp2}
\varepsilon M^\bullet \xrightarrow{\iota} M^\bullet \xrightarrow{g} C(\iota)^\bullet \xrightarrow{f}
\varepsilon M^\bullet[1]
\end{equation}
where $g^i(b)=(0,b)$ and $f^i(a,b)=-a$. 
We define two morphisms in $C^-(k[\varepsilon]\Lambda)$
\begin{eqnarray}
\label{eq:hyp3}
(0,\pi)&:& C(\iota)^\bullet=\varepsilon M^\bullet[1]\oplus M^\bullet
\to M^\bullet/\varepsilon M^\bullet \\
\psi&:& \varepsilon M^\bullet \to M^\bullet/\varepsilon M^\bullet \nonumber
\end{eqnarray}
by $(0, \pi)^i(a,b)=\pi^i(b)$ and $\psi^i(\varepsilon x)=\pi^i(x)$.
The kernel of $(0,\pi)$ is the mapping cone of
$\varepsilon M^\bullet \xrightarrow{\mathrm{id}} \varepsilon M^\bullet$ which is
acyclic; hence $(0,\pi)$ is a quasi-isomorphism. 
The morphism $\psi$ is an isomorphism of complexes
with inverse $\psi^{-1}$ given by $(\psi^{-1})^i(\pi^i(x))=\varepsilon x$.
Let $C(f)^\bullet$ be the mapping cone of $f$ from (\ref{eq:hyp2}). By the triangle axioms
(TR2) and (TR3) (see, for example, \cite[Def. 10.2.1]{Weibel}) and by the 5-lemma for
distinguished triangles (see, for example, \cite[Ex. 10.2.2]{Weibel}), 
there exists an isomorphism $\rho: C(f)^\bullet\to M^\bullet[1]$ in
$K^-(k[\varepsilon]\Lambda)$, which is represented by 
a quasi-isomorphism in $C^-(k[\varepsilon]\Lambda)$.
Therefore, we get a distinguished triangle in $K^-(k[\varepsilon]\Lambda)$
\begin{equation}
\label{eq:hyp6}
\xymatrix @-1pc {
C(\iota)^\bullet \ar[r]^{f} \ar[d]_{(0,\pi)}& \varepsilon M^\bullet[1]\ar@{=}[d] \ar[r]&
C(f)^\bullet\ar[r]\ar[d]^{\rho}&C(\iota)^\bullet[1]\ar[d]^{(0,\pi)[1]}\\
M^\bullet/\varepsilon M^\bullet&\varepsilon M^\bullet[1]&
M^\bullet[1]&(M^\bullet/\varepsilon M^\bullet)[1]
}
\end{equation}
where the downward arrows are quasi-isomorphisms in $C^-(k[\varepsilon]\Lambda)$.
Hence the diagram
\begin{equation}
\label{eq:hyp7}
\xymatrix @-1pc {
&C(\iota)^\bullet \ar[dl]_{(0,\pi)} \ar[dr]^{f}&\\
M^\bullet/\varepsilon M^\bullet&&\varepsilon M^\bullet[1]
}
\end{equation}
defines a morphism $\hat{f}:
M^\bullet/\varepsilon M^\bullet\to \varepsilon M^\bullet[1]$ in
$D^-(k[\varepsilon]\Lambda)$.
Because of (\ref{eq:hyp6}), we obtain a distinguished triangle 
$$M^\bullet/\varepsilon M^\bullet\xrightarrow{\hat{f}}\varepsilon M^\bullet[1]
\to M^\bullet[1]\to (M^\bullet/\varepsilon M^\bullet)[1] $$
in $D^-(k[\varepsilon]\Lambda)$.
Using the isomorphism $\phi:M^\bullet/\varepsilon M^\bullet
\to V^\bullet$ in $D^-(\Lambda)$, we obtain a morphism
$$\hat{f}_1\in\mathrm{Hom}_{D^-(k[\varepsilon]\Lambda)}(V^\bullet, V^\bullet[1])
$$
associated to $\hat{f}$, namely $\hat{f}_1=\phi'\circ\hat{f}\circ
\phi^{-1}$, where $\phi'=\phi[1]\circ\psi[1]$ and $\psi$ is as
in (\ref{eq:hyp3}).
We define a map 
\begin{eqnarray}
\label{eq:hyp17}
\hat{h}:\quad F(k[\varepsilon]) &\to& \mathrm{Hom}_{D^-(k[\varepsilon]\Lambda)}
(V^\bullet, V^\bullet[1])\, ,\\
 {[M^\bullet,\phi]} &\mapsto& \hat{f}_1 \; .\nonumber
\end{eqnarray}
As in the proof of \cite[Lemma 6.1]{bcderived}, it follows that $\hat{h}$ is a well-defined injective set map.
Moreover, it follows that if $(X^\bullet,\xi)$ is given by $X^\bullet=k[\varepsilon]\hat{\otimes}_k V^\bullet$ and
$\xi:k\hat{\otimes}_{k[\varepsilon]}X^\bullet\cong V^\bullet \xrightarrow{\mathrm{id}} V^\bullet$, i.e.
$[X^\bullet,\xi]$ is the trivial deformation of $V^\bullet$ over $k[\varepsilon]$, then
$$\hat{f}_1=\hat{h}([X^\bullet,\xi]) +\mathrm{Inf}_{\Lambda}^{k[\varepsilon]\Lambda}
\mathrm{Res}_{k[\varepsilon]\Lambda}^{\Lambda}(\hat{f}_1)$$
in $\mathrm{Hom}_{D^-(k[\varepsilon]\Lambda)}(V^\bullet,V^\bullet[1])$.
Because the map $\hat{h}$ in (\ref{eq:hyp17}) is injective
and the inflation map
$$\mathrm{Inf}_\Lambda^{k[\varepsilon]\Lambda}:\quad
\mathrm{Hom}_{D^-(\Lambda)}(V^\bullet,V^\bullet[1])
\to \mathrm{Hom}_{D^-(k[\varepsilon]\Lambda)}(V^\bullet,V^\bullet[1])$$
is injective, 
we obtain an injective map
\begin{eqnarray*}
h:\quad F(k[\varepsilon]) &\to& \mathrm{Hom}_{D^-(\Lambda)}(V^\bullet,V^\bullet[1]) \, ,\\
{[M^\bullet,\phi]}&\mapsto& \mathrm{Res}_{k[\varepsilon]\Lambda}^{\Lambda}(\hat{f}_1)
\; . \nonumber
\end{eqnarray*}
Since Schlessinger's criterion $\mathrm{(H2)}$ is valid by Proposition \ref{prop:step1}, it follows from 
\cite[Lemma 2.10]{Sch} that $t_F=F(k[\varepsilon])$ has a vector space structure. Hence we obtain as in the
proof of \cite[Lemma 6.1]{bcderived} that the map $h$ is $k$-linear and surjective.
More precisely, an element $\alpha\in \mathrm{Hom}_{D^-(\Lambda)}(V^\bullet, V^\bullet[1])$
defines a quasi-lift, and hence a deformation, of $V^\bullet$ over $k[\varepsilon]$ as follows.
Since $V^\bullet$ is assumed to be a bounded above complex of topologically free
pseudocompact  $\Lambda$-modules, $\alpha:V^\bullet \to V^\bullet[1]$
can be represented by a morphism in  $C^-(\Lambda)$.
Define an $\varepsilon$-action on the mapping cone $C(\alpha[-1])^\bullet$
by $\varepsilon(a,b)=(0,a)$ for all $(a,b)\in C(\alpha[-1])^i=V^i\oplus V^i$. Then the complex 
$M^\bullet=C(\alpha[-1])^\bullet$ defines a quasi-lift $(M^\bullet,\phi)$ of $V^\bullet$ over $k[\varepsilon]$
whose deformation is sent to $\alpha$ under $h$.
This proves part (i) of Proposition \ref{prop:step2}.

Part (ii) is proved similarly to \cite[Lemma 6.3]{bcderived}.
Namely, by part (i), $t_{F^{\fl}}$ is isomorphic
to the subspace of $\mathrm{Ext}^1_{D^-(\Lambda)}(V^\bullet,V^\bullet)$
consisting of those elements that define proflat deformations of $V^\bullet$
over $k[\varepsilon]$. Let $(M^\bullet,\phi)$ be a quasi-lift of $V^\bullet$ over
$k[\varepsilon]$. Then $(M^\bullet,\phi)$ is a proflat quasi-lift if the
cohomology groups of $M^\bullet$ are topologically free pseudocompact, and hence 
abstractly free, $k[\varepsilon]$-modules.
In this case $M^\bullet$ is isomorphic in $D^-(k[\varepsilon])$ to a bounded
complex with trivial differentials whose term in degree $n$ is a free $k[\varepsilon]$-module
of rank $\mathrm{dim}_k\, {\HH}^n(V^\bullet)$ for all integers $n$.
Therefore, all proflat quasi-lifts $(M^\bullet,\phi)$ of $V^\bullet$ over $k[\varepsilon]$
are isomorphic in $D^-(k[\varepsilon])$, when we forget the $\Lambda$-action.
This means that the tangent space $t_{F^{\fl}}$ is mapped to $\{0\}$
under the forgetful map (\ref{eq:forget}). Suppose now that $\alpha\in
\mathrm{Ext}^1_{D^-(\Lambda)}(V^\bullet,V^\bullet)$ is mapped to the zero morphism
under (\ref{eq:forget}). Then it follows from the proof of part (i) that
the deformation $[M^\bullet,\phi]$ of $V^\bullet$ over $k[\varepsilon]$ corresponding
to $\alpha$ is equal to the trivial deformation $[X^\bullet,\xi]$ of $V^\bullet$ over $k[\varepsilon]$
where $X^\bullet=k[\varepsilon]\hat{\otimes}_k V^\bullet$ and
$\xi:k\hat{\otimes}_{k[\varepsilon]}X^\bullet\cong V^\bullet \xrightarrow{\mathrm{id}} V^\bullet$. 
Since $V^\bullet$ is completely split in
$D^-(k)$, it follows that the cohomology groups of $X^\bullet$ are abstractly free, and hence 
topologically free, over $k[\varepsilon]$. Thus $[M^\bullet,\phi]$ is a proflat deformation of 
$V^\bullet$ over $k[\varepsilon]$.
This completes the proof of part (ii), and hence of Proposition \ref{prop:step2}.
\end{proof}

\begin{prop}
\label{prop:step3}
Schlessinger's criterion $\mathrm{(H3)}$ is satisfied, i.e.
the $k$-dimension of the tangent space $t_{F_{\mathcal{D}}}$ is finite.
\end{prop}

\begin{proof}
The proof of Proposition \ref{prop:step3} closely follows the proof of  \cite[Prop. 6.4]{bcderived}. 
By Proposition \ref{prop:step2}
it is enough to find an upper bound for the $k$-dimension of
$\mathrm{Ext}^1_{D^-(\Lambda)}(V^\bullet,V^\bullet)$.
By truncating and shifting, we can assume that $V^\bullet$
has the form
$$V^\bullet: \quad  \cdots 0 \to V^{-n} \to V^{-n+1} \to \cdots \to V^0
\to 0 \cdots\, .$$
We first prove the following claim, which replaces the assumption in \cite{bcderived}
that $G$ has finite pseudocompact cohomology in our situation.

\medskip

\noindent
\textit{Claim $1$.} If $M_1,M_2$ are 
pseudocompact $\Lambda$-modules that are finite dimensional over $k$, then
$\mathrm{Ext}^j_\Lambda(M_1,M_2)$ is finite dimensional over $k$ for all integers $j$.

\medskip

To prove Claim 1, we note that
by Remark \ref{rem:pseudocompact}(iii), $\mathrm{Ext}^j_\Lambda(M_1,M_2)$ is computed by using
a projective resolution of $M_1$ in $\mathrm{PCMod}(\Lambda)$. 
Since $\mathrm{dim}_k\,M_1$ is finite, there exists a resolution of
$M_1$ in $\mathrm{PCMod}(\Lambda)$ consisting of abstractly free finitely generated
$\Lambda$-modules. In other words, $\mathrm{Ext}^j_\Lambda(M_1,M_2)$ can be identified
with the corresponding $j$-th Ext group in the category of finitely generated $\Lambda$-modules.
The latter group is known to be a finite dimensional $k$-vector space, which proves Claim 1.

\medskip

\noindent
\textit{Claim $2$.} Suppose $n$ is a non-negative integer, and
$L_1^\bullet$ (resp. $L_2^\bullet$) is a complex of pseudocompact $\Lambda$-modules
whose terms are concentrated between the degrees $-n_1$ and $-n_1+n$
(resp. between $-n_2$ and $-n_2+n$), for integers $n_1$ and $n_2$.
Then for all integers $j$, $\mathrm{Ext}^j_{D^-(\Lambda)}(L_1^\bullet,
L_2^\bullet)$ has finite $k$-dimension, if all cohomology groups  of $L_1^\bullet$
and of $L_2^\bullet$ have finite $k$-dimension.

\medskip

Claim 2 is proved by induction on $n$.
If $n=0$, then $L_1^\bullet$ (resp. $L_2^\bullet$) is a module in degree $-n_1$
(resp. $-n_2$). Hence
$$\mathrm{Ext}^j_{D^-(\Lambda)}(L_1^\bullet,L_2^\bullet) \cong
\mathrm{Ext}^j_{\Lambda}({\HH}^{-n_1}(L_1^\bullet)[n_1],{\HH}^{-n_2}(L_2^\bullet)[n_2])\cong
\mathrm{Ext}^{j+n_2-n_1}_{\Lambda}(L_1^{-n_1},L_2^{-n_2})$$
which has finite $k$-dimension by Claim 1.
We can now complete the proof of Claim 2 in a similar way to the proof of 
\cite[Prop. 6.4]{bcderived}.

By setting $L_1^\bullet=V^\bullet=L_2^\bullet$ and $j=1$, Proposition \ref{prop:step3}
follows from Claim 2.
\end{proof}

\begin{prop}
\label{prop:step4}
The functor $\hat{F}_{\mathcal{D}}:\hat{\mathcal{C}}\to \mathrm{Sets}$
is continuous. In other words, for all objects $R$ in $\hat{\mathcal{C}}$
with maximal ideal $m_R$ we have
$$\hat{F}_{\mathcal{D}}(R)=\lim_{\stackrel{\longleftarrow}{i}} \hat{F}_{\mathcal{D}}(R/m_R^i)\, .$$
\end{prop}

\begin{proof}
The proof of Proposition \ref{prop:step4} closely follows the proof of \cite[Prop. 7.2]{bcderived}.
By Remark \ref{rem:functor} and Lemma \ref{lem:corollary3.6}(i),
we may assume, without loss of generality,
that $V^\bullet$ is a bounded above complex of abstractly free finitely generated 
$\Lambda$-modules.
Let $R$ be an object of $\hat{\mathcal{C}}$ with maximal ideal $m_R$.
Consider the natural map
\begin{equation}
\label{eq:themap}
\Xi_{\mathcal{D}}: \hat{F}_{\mathcal{D}}(R) \to \lim_{\stackrel{\longleftarrow}{i}}
\hat{F}_{\mathcal{D}}(R/m_R^i)
\end{equation}
defined by $\Xi_{\mathcal{D}}([M^\bullet,\phi])= 
\{\hat{F}_{\mathcal{D}}(\pi_i)([M^\bullet,\phi])\}_{i=1}^\infty=
\{[(R/m_R^i) \hat{\otimes}^{\LL}_R M^\bullet,\phi_{\pi_i} ]\}_{i=1}^\infty$ when $\pi_i:R\to R/m_R^i$ is the natural surjection for all $i$.

We first show that $\Xi_{\mathcal{D}}$ is surjective. Suppose we have a sequence
of deformations
$\{[M_i^\bullet,\phi_i]\}_{i=1}^\infty$ with $[M_i^\bullet,\phi_i] \in
\hat{F}_{\mathcal{D}}(R/m_R^i)=F_{\mathcal{D}}(R/m_R^i)$
for all $i$ such that there is an isomorphism
$$\alpha_i: (R/m_R^{i})\hat{\otimes}^{\LL}_{R/m_R^{i+1}} M_{i+1}^\bullet \to
M_{i}^\bullet$$
in $D^-((R/m_R^{i})\Lambda)$ with 
$\phi_{i}\circ (k\hat{\otimes}^{\LL}_{R/m_R^{i}}\alpha_i) = \phi_{i+1}$.
We need to construct a quasi-lift $(M^\bullet,\phi)$ with
$[M^\bullet,\phi]\in \hat{F}_{\mathcal{D}}(R)$ such that, for all $i$, there is an isomorphism
$\beta_i:(R/m_R^{i})\hat{\otimes}^{\LL}_{R}M^\bullet
\to M_i^\bullet$ in $D^-((R/m_R^{i})\Lambda)$ with $\alpha_{i}\circ ((R/m_R^{i})\hat{\otimes}^{\LL}_{R/
m_R^{i+1}}\beta_{i+1})=\beta_i$ and $\phi_i\circ(k\hat{\otimes}^{\LL}_{R/m_R^{i}}\beta_i) 
=\phi$. We now construct $(M^\bullet,\phi)$ inductively.

By Remark \ref{rem:corollary3.1011}, 
we can replace $M_i^\bullet$ by a bounded above  complex $N_i^\bullet$
of abstractly free finitely generated $(R/m_R^i)\Lambda$-modules,  
and we can replace $\phi_i$ by a quasi-isomorphism
$\psi_i:k\hat{\otimes}_{R/m_R^i}N_i^\bullet\to V^\bullet$ in $C^-(\Lambda)$.
Suppose $\gamma_i:  N_i^\bullet\to M_i^\bullet$ is an
isomorphism in $D^-((R/m_R^i)\Lambda)$ associated to these replacements.
Then the diagram
\begin{equation}
\label{eq:musthaveit2}
\xymatrix @-1pc @C6pc{
(R/m_R^{i})\hat{\otimes}_{R/m_R^{i+1}}{N}_{i+1}^\bullet \ar[r]^{(R/m_R^i)\hat{\otimes}^{\LL}\gamma_{i+1}} 
& (R/m_R^{i})\hat{\otimes}^{\LL}_{R/m_R^{i+1}}
M_{i+1}^\bullet\ar[d]^{\alpha_i}\\
&M_i^\bullet\ar[d]^{\gamma_i^{-1}}\\ &N_i^\bullet
}
\end{equation}
defines an isomorphism $\tilde{\alpha}_i:(R/m_R^{i})\hat{\otimes}_{R/m_R^{i+1}}{N}_{i+1}^\bullet 
\to  N_i^\bullet$ in $D^-((R/m_R^{i})\Lambda)$, and $\tilde{\alpha}_i$ replaces $\alpha_i$.
Moreover, $\psi_i=\phi_i\circ (k\hat{\otimes}^{\LL}_{R/m_R^i}\gamma_i)$ satisfies
$$\psi_i\circ(k\hat{\otimes}_{R/m_R^i}\tilde{\alpha}_i)=\psi_{i+1}\,.$$
Since $N_{i+1}^\bullet$ is a bounded above  complex
of abstractly free finitely generated $(R/m_R^{i+1})\Lambda$-modules,  
we can assume that $\tilde{\alpha}_i$ is represented by a quasi-isomorphism 
in $C^-((R/m_R^i)\Lambda)$. By Remark \ref{rem:Remark5.2}, we can add
to $N_{i+1}^\bullet$ a suitable acyclic bounded above complex of 
abstractly free finitely generated $(R/m_R^{i+1})\Lambda$-modules 
to be able to assume that  $\tilde{\alpha}_i$ is surjective on terms. 
Continuing this process inductively, we obtain an inverse system of quasi-lifts
$$\{( N_i^\bullet,\psi_i)\}_{i=1}^\infty \quad\mbox{  where } \quad [N_i^\bullet, \psi_i]\in F_{\mathcal{D}}(R/m_R^i) .$$
Further, $N_i^\bullet$
is a complex of abstractly free finitely generated $(R/m_R^i)\Lambda$-modules
such that in the diagram
\begin{equation}
\label{eq:cont2}
\xymatrix @-1pc @C1.5pc{
{N_{i+1}^\bullet} \ar[d]&\\
{(R/m_R^{i})\hat{\otimes}_{R/m_R^{i+1}}{N}_{i+1}^\bullet} \ar[r]^>>>>{\tilde{\alpha}_i}&
{N_i^\bullet}
}
\end{equation}
all arrows are morphisms in $C^-((R/m_R^{i+1})\Lambda)$
that are surjective on terms.
Define 
$$M^\bullet=\lim_{\stackrel{\longleftarrow}{i}}\, N_i^\bullet\quad\mbox{ and }\quad
\phi=\lim_{\stackrel{\longleftarrow}{i}}\,\psi_i\,.$$
Then we obtain in a similar way as in the proof of \cite[Prop. 7.2]{bcderived} that
all terms of $M^\bullet$ are topologically free pseudocompact $R$-modules
and that $[M^\bullet,\phi]\in \hat{F}_{\mathcal{D}}(R)$. 
Letting for each $i$
$$\tilde{\beta}_i:\;(R/m_R^{i})\hat{\otimes}_{R}M^\bullet \;\longrightarrow\;  N_i^\bullet$$
be the natural isomorphism in  $C^-((R/m_R^{i})\Lambda)$, 
it follows that  $\tilde{\alpha}_{i}
\circ ((R/m_R^{i})\hat{\otimes}_{R/m_R^{i+1}}\tilde{\beta}_{i+1})=\tilde{\beta}_i$,
where $\tilde{\alpha}_i$ is as defined in (\ref{eq:musthaveit2}), and 
$\psi_i\circ (k\hat{\otimes}_{R/m_R^{i}}\tilde{\beta}_i) =\phi$. 
Hence, defining $\beta_i=\gamma_i\circ\tilde{\beta}_i$, we have
$\alpha_{i}\circ ((R/m_R^{i})\hat{\otimes}^{\LL}_{R/
m_R^{i+1}}\beta_{i+1})=\beta_i$ and $\phi_i\circ(k\hat{\otimes}^{\LL}_{R/m_R^{i}}\beta_i) 
=\phi$. 
Therefore, the map $\Xi_{\mathcal{D}}$ in (\ref{eq:themap}) is surjective.

We now show that $\Xi_{\mathcal{D}}$ is injective. Since $\hat{F}^{\fl}$
is a subfunctor of $\hat{F}$ by Remark \ref{rem:functor}, it is enough
to show that $\Xi_{\mathcal{D}}$ is injective in the case when $\hat{F}_{\mathcal{D}}=\hat{F}$. 
We notice that in the proof of \cite[Prop. 7.2]{bcderived} an assumption was made to arrive at a morphism 
$f_i$ as in \cite[Eq. (7.5)]{bcderived}. This needs more explanation, which we will now provide in our
situation. We first prove three claims.

\medskip

\noindent
\textit{Claim $1$.} Let $[M^\bullet, \phi]\in\hat{F}(R)$. For all positive integers $i$, define
$$Z_i=\{\zeta_i\in\mathrm{End}_{D^-((R/m_R^i)\Lambda)}((R/m_R^i)\hat{\otimes}^{\LL}_R {M}^\bullet)\;
|\;k\hat{\otimes}^{\LL}_{R/m_R^{i}}\zeta_i=0\mbox{ in } D^-(\Lambda)\}.$$
Then $Z_i$ is a finitely generated nilpotent $(R/m_R^i)$-module.

\medskip

To prove Claim 1, we first use Remarks \ref{rem:leftderivedtensor} and  \ref{rem:Lemma4.2} to
assume without loss of generality that $M^\bullet$ is a bounded
above complex of topologically free pseudocompact $R\Lambda$-modules. Hence
\begin{equation}
\label{eq:nicerzi}
Z_i=\{\zeta_i\in\mathrm{End}_{K^-((R/m_R^i)\Lambda)}((R/m_R^i)\hat{\otimes}_R {M}^\bullet)\;
|\;k\hat{\otimes}_{R/m_R^{i}}\zeta_i=0\mbox{ in } K^-(\Lambda)\}.
\end{equation}
It is obvious that $Z_i$ is an $(R/m_R^i)$-module. Suppose $\HH^j(V^\bullet)=0$ for $j<n_1$ and
$j>n_2$. By Theorem \ref{thm:derivedresult}, it follows that also 
$\HH^j((R/m_R^i)\hat{\otimes}_R {M}^\bullet)=0$ for $j<n_1$ and $j>n_2$.
By Lemma \ref{lem:corollary3.6}(i),  there exists 
a bounded above complex $N_i^\bullet$ of abstractly free finitely generated 
$(R/m_R^i)\Lambda$-modules such that there is an isomorphism
$(R/m_R^i)\hat{\otimes}_R {M}^\bullet\to N_i^\bullet$ in 
$K^{-}_{\fin}((R/m_R^i)\Lambda)$. Hence we can truncate $N_i^\bullet$ at $n_1$ and $n_2$ to be able to 
assume that $N_i^j=0$ for $j<n_1$ and $j>n_2$ and that $N_i^j$ is an abstractly free finitely generated 
$(R/m_R^i)$-module for $n_1\le j\le n_2$. We obtain that $Z_i$ is an $(R/m_R^i)$-submodule of
$$\mathrm{End}_{K^-(R/m_R^i)}((R/m_R^i)\hat{\otimes}_R {M}^\bullet)\cong
\mathrm{End}_{K^-(R/m_R^i)}(N_i^\bullet)$$
which is isomorphic to a quotient module of $\mathrm{End}_{C^-(R/m_R^i)}(N_i^\bullet)$. This, in
turn, is a submodule of $\prod_{j=n_1}^{n_2}\mathrm{End}_{R/m_R^i}(N_i^j)$, which is a free
$(R/m_R^i)$-module of finite rank. Since $R/m_R^i$ is Noetherian, it follows that
$Z_i$ is a finitely generated $(R/m_R^i)$-module. 
Since $(\zeta_i)^i=0$ in $K^-((R/m_R^i)\Lambda)$ for all $\zeta_i\in Z_i$, Claim 1 follows.

\medskip

\noindent 
\textit{Claim $2$.} Let $[M^\bullet,\phi]$ and $Z_i$ be as in Claim 1. 
For all positive integers $i$ and for all $j\ge i$, define $\tau^j_i: Z_j\to Z_i$ to be the map 
that sends $\zeta_j\in Z_j$ to $(R/m_R^i)\hat{\otimes}^{\LL}_{R/m_R^j}\zeta_j$. Then
there exists $N\ge i$ such that for all $j\ge N$, $\tau^j_i(Z_j)=\tau^N_i(Z_N)$. In other words,
the inverse system $\{Z_i, \tau^j_i\}$ satisfies the Mittag-Leffler condition.

\medskip

Claim 2 follows almost immediately from Claim 1. Namely,
by Claim 1, $Z_i$ is a finitely generated module over the Artinian ring 
$R/m_R^i$. In particular,
$Z_i$ is Artinian. Consider the descending chain of submodules
$$Z_i= \tau^i_i(Z_i)\supseteq \tau^{i+1}_i(Z_{i+1})\supseteq \tau^{i+2}_i(Z_{i+2})\supseteq \ldots$$
Since $Z_i$ is Artinian, this must stabilize after finitely many steps, proving Claim 2.

\medskip

\noindent
\textit{Claim $3$.} Suppose $\hat{F}_{\mathcal{D}}=\hat{F}$ and 
$[M^\bullet,\phi], [\widetilde{M}^\bullet,\widetilde{\phi}]
\in \hat{F}(R)$ are such that $\Xi_{\mathcal{D}}([M^\bullet,\phi])=
\Xi_{\mathcal{D}}([\widetilde{M}^\bullet,\widetilde{\phi}])$. Then for all $i$, there exist isomorphisms
$f_i: {(R/m_R^i)\hat{\otimes}^{\LL}_R {M}^\bullet} \to 
{(R/m_R^i)\hat{\otimes}^{\LL}_R \widetilde{M}^\bullet}$ in $D^-((R/m_R^i)\Lambda)$ such that
$(R/m_R^i)\hat{\otimes}^{\LL}_{R/m_R^{i+1}}f_{i+1}= f_i$ in $D^-((R/m_R^i)\Lambda)$ and 
$\widetilde{\phi}\circ (k\hat{\otimes}^{\LL}_{R/m_R^{i}}f_{i})=\phi$ in $D^-(\Lambda)$.

\medskip

To prove Claim 3, we can assume, as in the proof of Claim 1, 
that $M^\bullet$ and 
$\widetilde{M}^\bullet$ are bounded
above complexes of topologically free pseudocompact $R\Lambda$-modules. In particular, $Z_i$ is
given as in $(\ref{eq:nicerzi})$. For each positive integer $i$, define 
$$S_i=\bigcap_{j\ge i} \tau_i^j(Z_j)$$
which is a subset of $Z_i$, and define $\sigma_i:S_{i+1}\to S_i$ by the restriction of
$\tau_i^{i+1}$ to $S_i$. Then $\sigma_i$ is a well-defined surjecive map for all $i\ge 1$. 
By Claim 2, for each $i\ge 1$ there exists a positive integer $N_i\ge i$ such that
$S_i=\tau_i^{N_i}(Z_{N_i})$. Assume $N_i\ge i$ is chosen to be smallest with this property.
Note that by construction, $N_i\le N_{i+1}$ for all $i$.

By our assumptions for Claim 3, $\Xi_{\mathcal{D}}([M^\bullet,\phi])=\Xi_{\mathcal{D}}([\widetilde{M}^\bullet,\widetilde{\phi}])$, which means
that for each positive integer $i$, there exists an isomorphism
$$\gamma_i:{(R/m_R^i)\hat{\otimes}_R {M}^\bullet} \to {(R/m_R^i)\hat{\otimes}_R \widetilde{M}^\bullet}$$
in $K^-((R/m_R^i)\Lambda)$ such that
$\widetilde{\phi}\circ (k\hat{\otimes}_{R/m_R^{i}}\gamma_{i})=\phi$ in $K^-(\Lambda)$. 

For each positive integer $i$, let $\mathrm{Id}_i$ be the identity cochain map
of $(R/m_R^i)\hat{\otimes}_RM^\bullet$ and define 
\begin{equation}
\label{eq:fi}
\tilde{f}_i=(R/m_R^i)\hat{\otimes}_{R/m_R^{N_i}}\,\gamma_{N_i}
\end{equation}
and
\begin{equation}
\label{eq:xi}
\zeta_i=\tilde{f}_i^{-1}\circ ((R/m_R^i)\hat{\otimes}_{R/m_R^{i+1}}\,\tilde{f}_{i+1})-\mathrm{Id}_i\,.
\end{equation}
Then $$\zeta_i=(R/m_R^i)\hat{\otimes}_{R/m_R^{N_i}}\left(
\gamma_{N_i}^{-1}\circ \left((R/m_R^{N_i})\hat{\otimes}_{R/m_R^{N_{i+1}}} \gamma_{N_{i+1}}\right)-
\mathrm{Id}_{N_i}\right)$$ 
is an element of $\tau_i^{N_i}(Z_{N_i})=S_i$, since
$(k\hat{\otimes}_{R/m_R^{i}}\gamma_{i})=\widetilde{\phi}^{-1}\circ\phi$
in $K^-(\Lambda)$ for all $i\ge 1$. 

We define $f_1=\tilde{f}_1$ and $f_2=\tilde{f}_2$. Then
$$(R/m_R)\hat{\otimes}_{R/m_R^{2}}f_2=k\hat{\otimes}_{R/m_R^{N_2}}\,\gamma_{N_2}
=\widetilde{\phi}^{-1}\circ\phi=k\hat{\otimes}_{R/m_R^{N_1}}\,\gamma_{N_1}=f_1.$$
For all $j\ge 2$, define $\zeta_j^{(j)}=\zeta_j$ to be the element in $S_j$ defined in (\ref{eq:xi}). 
Inductively, for all $\ell\ge 1$ and all $j\ge \ell+2$, use that $\sigma_{j-1}:S_j\to S_{j-1}$ is surjective to choose 
an element $\zeta_{j-\ell}^{(j)}\in S_j$ with $\sigma_{j-1}(\zeta_{j-\ell}^{(j)}) =\zeta_{j-\ell}^{(j-1)}$.
In particular, this means that for all $j\ge 3$ and all $2\le i< j$, the element $\zeta_i^{(j)}\in S_j$ is such that
for all $t$ with $i\le  t< j$, $\left(\sigma_t\circ\sigma_{t+1}\circ\cdots\circ\sigma_{j-1}\right)(\zeta_i^{(j)})=\zeta_i^{(t)}$. 
For all $j\ge 3$, we define
$$f_j=\tilde{f}_j\circ\left(\mathrm{Id}_j+\zeta_{j-1}^{(j)}\right)^{-1}\circ \left(\mathrm{Id}_j+\zeta_{j-2}^{(j)}\right)^{-1}\circ\cdots\circ
\left(\mathrm{Id}_j+\zeta_2^{(j)}\right)^{-1}$$
where we use Claim 1 to see that,  for all $2\le i\le j$, $\zeta_i^{(j)}\in S_j\subseteq Z_j$ is nilpotent, 
and hence $(\mathrm{Id}_j+\zeta_{i}^{(j)})$ is invertible in $K^-((R/m_R^j)\Lambda)$.
We obtain
\begin{eqnarray*}
(R/m_R^{j-1})\hat{\otimes}_{R/m_R^{j}}f_j&=&
\left((R/m_R^{j-1})\hat{\otimes}_{R/m_R^{j}}\tilde{f}_j\right)\circ
\left(\mathrm{Id}_{j-1}+\zeta_{j-1}^{(j-1)}\right)^{-1}\circ \\
&&\qquad\left(\mathrm{Id}_{j-1}+\zeta_{j-2}^{(j-1)}\right)^{-1}\circ\cdots\circ\left(\mathrm{Id}_{j-1}+\zeta_2^{(j-1)}\right)^{-1}\\
&=&\tilde{f}_{j-1}\circ \left(\mathrm{Id}_{j-1}+\zeta_{j-2}^{(j-1)}\right)^{-1}\circ \cdots\circ \left(\mathrm{Id}_{j-1}+\zeta_2^{(j-1)}\right)^{-1}\\
&=&f_{j-1}
\end{eqnarray*}
in $K^-((R/m_R^{j-1})\Lambda)$, where the second to last equality follows from $(\ref{eq:xi})$. Moreover, we have 
\begin{eqnarray*}
\widetilde{\phi}\circ(k\hat{\otimes}_{R/m_R^{j}}f_{j})&=&
\widetilde{\phi}\circ(k\hat{\otimes}_{R/m_R^{j}}\tilde{f}_j)\circ
\left(\mathrm{Id}_{1}+k\hat{\otimes}_{R/m_R^{j-1}}\zeta_{j-1}\right)^{-1}\circ \cdots\circ\left(\mathrm{Id}_{1}+k\hat{\otimes}_{R/m_R^{2}}\zeta_2\right)^{-1}\\
&=&\widetilde{\phi}\circ(k\hat{\otimes}_{R/m_R^{j}}\tilde{f}_j)\;=\;
\widetilde{\phi}\circ(k\hat{\otimes}_{R/m_R^{N_j}}\,\gamma_{N_j})\;=\;\phi
\end{eqnarray*}
in $K^-(\Lambda)$, where the second equation follows since all $\zeta_i\in S_i\subseteq Z_i$. 
This proves Claim 3.

\bigskip

We are now ready to prove that $\Xi_{\mathcal{D}}$ in (\ref{eq:themap}) is injective
in the case when $\hat{F}_{\mathcal{D}}=\hat{F}$.
Suppose $[M^\bullet,\phi], [\widetilde{M}^\bullet,\widetilde{\phi}]$
in $\hat{F}(R)$ satisfy $\Xi_{\mathcal{D}}([M^\bullet,\phi])=\Xi_{\mathcal{D}}([\widetilde{M}^\bullet,\widetilde{\phi}])$.
As in the proof of Claim 1, we can assume without loss of generality that $M^\bullet$ and 
$\widetilde{M}^\bullet$ are bounded above complexes of topologically free pseudocompact $R\Lambda$-modules.
Hence it follows from Claim 3 that for all $i\ge 1$,
there exist isomorphisms 
$$f_i: {(R/m_R^i)\hat{\otimes}_R {M}^\bullet} \to {(R/m_R^i)\hat{\otimes}_R \widetilde{M}^\bullet}$$ 
in $K^-((R/m_R^i)\Lambda)$ such that
$(R/m_R^i)\hat{\otimes}_{R/m_R^{i+1}}f_{i+1}= f_i$ in $K^-((R/m_R^i)\Lambda)$ and 
$\widetilde{\phi}\circ(k\hat{\otimes}_{R/m_R^{i}}f_{i})=\phi$ in $K^-(\Lambda)$.
Suppose that for all $i$, $f_i$ is represented by a quasi-isomorphism in $C^-((R/m_R^i)\Lambda)$.
By Remark \ref{rem:Remark5.2}, we can add to $M^\bullet$ a suitable acyclic
bounded above complex of topologically free pseudocompact $R\Lambda$-modules
to be able to assume that all $f_i$ are surjective on terms. Define $h_1=f_1$.
As in the proof of \cite[Prop. 7.2]{bcderived}, we can construct inductively, for all $i$, a quasi-isomorphism
$$h_i: (R/m_R^i)\hat{\otimes}_R
{M}^\bullet\to (R/m_R^i)\hat{\otimes}_R \widetilde{M}^\bullet$$
in $C^-((R/m_R^i)\Lambda)$ that is surjective on terms
such that $h_i$ is homotopic to $f_i$ and such that
$$(R/m_R^j)\hat{\otimes}_{R/m_R^i} h_i=h_j$$ for all $j<i$.
It follows that
$$h=\lim_{\stackrel{\longleftarrow}{i}}\, h_i:\quad
M^\bullet = \lim_{\stackrel{\longleftarrow}{i}}\, ((R/m_R^i)\hat{\otimes}_R M^\bullet)
\;\longrightarrow\; \lim_{\stackrel{\longleftarrow}{i}}\, ((R/m_R^i)\hat{\otimes}_R
\widetilde{M}^\bullet) = \widetilde{M}^\bullet$$
is an isomorphism in $K^-(R\Lambda)$ with
$\displaystyle \widetilde{\phi}\circ(k\hat{\otimes}_R h)=\phi$ in $K^-(\Lambda)$.
This completes the proof of Proposition \ref{prop:step4}.
\end{proof}

By \cite[Sect. 2]{Sch}, Theorem \ref{thm:bigthm} now follows from Propositions \ref{prop:step1} through \ref{prop:step4}.

\begin{rem}
\label{rem:therefereeaskedforit}
Note that we can prove claims that are similar to Claims 1 - 3 in the proof of Proposition \ref{prop:step4} to also provide 
more details in the proof of the continuity of the deformation functor defined in \cite{bcderived}.
As mentioned above, the main point is that in the proof of the injectivity of the map $\Gamma_{\mathcal{D}}$ defined in
the proof of \cite[Prop. 7.2]{bcderived} an assumption was made to arrive at a morphism 
$f_i$ as in \cite[Eq. (7.5)]{bcderived} that needs more explanation. We now briefly sketch the necessary arguments.

Suppose $G$ is a profinite group with finite pseudocompact cohomology, as defined in \cite[Def. 2.13]{bcderived}. 
Make the assumptions on $k$ and $\hat{\mathcal{C}}$ as in \cite{bcderived}. Assume $V^\bullet $ is a
complex in $D^-([[kG]])$ that has  only finitely many non-zero cohomology
groups, all of which have finite $k$-dimension. 
Let $\hat{F}_{\mathcal{D}}=\hat{F}_{\mathcal{D},V^\bullet}:\hat{\mathcal{C}} \to \mathrm{Sets}$ be the 
deformation functor defined in \cite[Def. 2.10]{bcderived}. We want to show that $\hat{F}_{\mathcal{D}}$ is continuous.

By \cite[Prop. 2.12 and Cor. 3.6(i)]{bcderived} we may assume, without loss of generality,
that there is a closed normal subgroup $\Delta$ of finite index in $G$ such that
$V^\bullet$ is a bounded above complex of abstractly free finitely generated 
$[k(G/\Delta)]$-modules.
Let $R$ be an object of $\hat{\mathcal{C}}$ with maximal ideal $m_R$, and consider the natural map
$$\Gamma_{\mathcal{D}}: \hat{F}_{\mathcal{D}}(R) \to \lim_{\stackrel{\longleftarrow}{i}}
\hat{F}_{\mathcal{D}}(R/m_R^i)
$$
defined by  $\Gamma_{\mathcal{D}}([M^\bullet,\phi])= \{[(R/m_R^i) \hat{\otimes}^{\LL}_R
M^\bullet,\phi_{\pi_i} ]\}_{i=1}^\infty$ when $\pi_i:R\to R/m_R^i$ is the natural surjection for all $i$.
The proof that $\Gamma_{\mathcal{D}}$ is surjective is the same as in the proof of \cite[Prop. 7.2]{bcderived}.
For the proof that  $\Gamma_{\mathcal{D}}$ is injective, we use that  $\hat{F}^{\fl}$
is a subfunctor of $\hat{F}$ by \cite[Prop. 2.12]{bcderived}. Hence it suffices to
show that $\Gamma_{\mathcal{D}}$ is injective in the case when $\hat{F}_{\mathcal{D}}=\hat{F}$. 

To prove this, we modify Claims 1 - 3 in the proof of Proposition \ref{prop:step4} to arrive at the three claims 
Claims $G.1$ - $G.3$ below. Using $f_i$ from Claim $G.3$, instead of the morphism $f_i$ in \cite[Eq. (7.5)]{bcderived}, 
the remainder of the proof of the injectivity of $\Gamma_\mathcal{D}$ follows then as in the proof of
\cite[Prop. 7.2]{bcderived}.

\medskip

\noindent
\textit{Claim $G.1$.} Let $[M^\bullet, \phi]\in \hat{F}(R)$. For all positive integers $i$, define
$$Z_i=\{\zeta_i\in\mathrm{End}_{D^-([[(R/m_R^i)G]])}((R/m_R^i)\hat{\otimes}^{\LL}_R {M}^\bullet)\;
|\;k\hat{\otimes}^{\LL}_{R/m_R^{i}}\zeta_i=0\mbox{ in } D^-([[kG]])\}.$$
Then $Z_i$ is a finitely generated nilpotent $(R/m_R^i)$-module.

\medskip

\noindent
\textit{Claim $G.2$.} Let $[M^\bullet,\phi]$ and $Z_i$ be as in Claim $G.1$. 
For all positive integers $i$ and for all $j\ge i$, let $\tau^j_i: Z_j\to Z_i$ be the map 
that sends $\zeta_j$ to $(R/m_R^i)\hat{\otimes}^{\LL}_{R/m_R^j}\zeta_j$. Then
there exists $N\ge i$ such that for all $j\ge N$, $\tau^j_i(Z_j)=\tau^N_i(Z_N)$. In other words,
the inverse system $\{Z_i, \tau^j_i\}$ satisfies the Mittag-Leffler condition.

\medskip

\noindent
\textit{Claim $G.3$.} Suppose $\hat{F}_{\mathcal{D}}=\hat{F}$ and
$[M^\bullet,\phi], [\widetilde{M}^\bullet,\widetilde{\phi}]
\in \hat{F}(R)$ are such that $\Gamma_\mathcal{D}([M^\bullet,\phi])=
\Gamma_\mathcal{D}([\widetilde{M}^\bullet,\widetilde{\phi}])$. Then for all $i$, there are isomorphisms
$f_i: {(R/m_R^i)\hat{\otimes}^{\LL}_R {M}^\bullet} \to 
{(R/m_R^i)\hat{\otimes}^{\LL}_R \widetilde{M}^\bullet}$ in $D^-([[(R/m_R^i)G]])$ such that
$(R/m_R^i)\hat{\otimes}^{\LL}_{R/m_R^{i}}f_{i+1}= f_i$ in $D^-([[(R/m_R^i)G]])$ and 
$\widetilde{\phi}\circ(k\hat{\otimes}^{\LL}_{R/m_R^{i}}f_{i})=\phi$ in $D^-([[kG]])$.

\medskip

The proofs of Claims $G.1$ - $G.3$ are rather similar to the proofs of Claims 1 - 3 in the proof of
Proposition \ref{prop:step4}. 
We use \cite[Lemma 4.2]{bcderived} to be able to
assume that $M^\bullet$ (and in Claim $G.3$ also $\widetilde{M}^\bullet$) 
is a bounded above complex of topologically free pseudocompact $[[RG]]$-modules.
If $\HH^j(V^\bullet)=0$ for $j<n_1$ and $j>n_2$, then
we use \cite[Thm. 2.10]{obstructions} to argue that $\HH^j((R/m_R^i)\hat{\otimes}_R {M}^\bullet)=0$ for 
$j<n_1$ and $j>n_2$. Moreover, we use \cite[Cor. 3.6(i)]{bcderived} to see that there exists a closed 
normal subgroup $\Delta_i$ of finite index in $G$ and a bounded above complex $N_i^\bullet$ of abstractly free 
finitely generated $[(R/m_R^i)(G/\Delta_i)]$-modules such that there is an isomorphism
$(R/m_R^i)\hat{\otimes}_R {M}^\bullet\to \mathrm{Inf}_{G/\Delta_i}^G(N_i^\bullet)$ in 
$K^{-}_{\fin}([[(R/m_R^i)G]])$. In particular, $\prod_{j=n_1}^{n_2}\mathrm{End}_{R/m_R^i}(N_i^j)$ is then
a free $(R/m_R^i)$-module of finite rank. 
The arguments of the remainder of the proofs of Claims $G.1$ - $G.3$ are basically the same as in the proof
of Proposition \ref{prop:step4}. 
\end{rem}

We now return to the situation in the present paper and, in particular, to the assumptions put forward in Section \ref{s:setup}.

\subsection{Special complexes $V^\bullet$}
\label{s:12terms}

In this section we consider several cases of more special complexes $V^\bullet$. Namely,
we consider one-term and two-term complexes, and completely split complexes.
The results are adapted from \cite[Sects. 9 and 11]{bcderived}. 

We first recall the deformation functor associated to a finitely generated $\Lambda$-module 
that was studied in \cite{blehervelez}.

\begin{rem}
\label{rem:deformmodules}
Let $C$ be a finitely generated $\Lambda$-module, and let $R$ be an object of $\hat{\mathcal{C}}$. According to \cite[Sect. 2]{blehervelez}, a lift of $C$ over $R$ is a pair $(L,\lambda)$ consisting of 
a finitely generated $R\Lambda$-module $L$ that is abstractly free as an $R$-module together with a $\Lambda$-module isomorphism $\lambda:k\otimes_R M \to V$.
Two lifts $(L,\lambda)$ and $(L',\lambda')$ of $C$ over $R$ are said to be  isomorphic if there is an $R\Lambda$-module isomorphism 
$f:L\to L'$ with $\lambda'\circ (k\otimes_R f) =\lambda$. A deformation of $C$ over $R$ is an isomorphism class of lifts of $C$ over $R$. We denote the
deformation of $C$ over $R$ represented by $(L,\lambda)$ by $[L,\lambda]$.
The deformation functor $\hat{F}_C : \hat{\mathcal{C}}\to \mathrm{Sets}$ is defined to be the covariant functor
that sends an object $R$ of $\hat{\mathcal{C}}$ to the set $\hat{F}_C(R)$ of all deformations of $C$ over $R$, and that sends
a morphism $\alpha:R\to R'$ in $\hat{\mathcal{C}}$ to the set map
$\hat{F}_C(R)\to \hat{F}_C(R')$ given by $[L,\lambda] \mapsto [R'\otimes_{R,\alpha}L,\lambda_\alpha]$ 
where $\lambda_\alpha$ is the composition $k\otimes_{R'}(R'\otimes_{R,\alpha}L)\cong k\otimes_R L\xrightarrow{\lambda} C$.
Let $F_C:\mathcal{C}\to\mathrm{Sets}$  be the restriction of $\hat{F}_C$ to the full subcategory $\mathcal{C}$ of Artinian rings in $\hat{\mathcal{C}}$.

It was shown in \cite[Prop. 2.1]{blehervelez} that $F_C$ always has a pro-representable 
hull $R(\Lambda,C)$ in $\hat{\mathcal{C}}$, called the versal deformation ring of $C$, and that $\hat{F}_C$ is continuous. Moreover, if $\mathrm{End}_\Lambda(V)= k$,
it was shown that $\hat{F}_C$ is represented by $R(\Lambda,C)$, in which case $R(\Lambda,C)$ is called the universal deformation ring of $C$.
\end{rem}

\begin{prop}
\label{prop:modulecase}
Assume Hypothesis $\ref{hypo:fincoh}$, and let ${F}_{\mathcal{D}}$ and
$\hat{F}_{\mathcal{D}}$ be as in Definition $\ref{def:functordef}$.
Suppose $V^\bullet$ has exactly one non-zero cohomology group $C$, which has finite 
$k$-dimension. Let $\hat{F}_C$ be the deformation functor considered in \cite{blehervelez} $($see Remark $\ref{rem:deformmodules}$$)$.
Then $\hat{F}_{\mathcal{D}}$ and $\hat{F}_C$ are naturally isomorphic, and 
$R(\Lambda,V^\bullet)$ is isomorphic to the versal deformation ring $R(\Lambda,C)$ of $C$.
In particular, $R(\Lambda,V^\bullet)=R^{\fl}(\Lambda,V^\bullet)$.
The groups $\mathrm{Hom}_{D^-(\Lambda)}(V^\bullet,V^\bullet)$
and  $\mathrm{Hom}_{\Lambda}(C,C)$ are isomorphic.
\end{prop}

\begin{proof}
The proof of Proposition \ref{prop:modulecase} closely follows the proof of \cite[Prop. 9.1]{bcderived}.
It will suffice to show that the functor $\hat {F}:\hat {\mathcal{C}} \to \mathrm{Sets}$ defined
by the isomorphism classes of quasi-lifts of $V^\bullet$ is naturally isomorphic to the functor 
$\hat{F}_C$ from \cite{blehervelez} (see Remark \ref{rem:deformmodules}).
Because $\hat {F}$ and $\hat{F}_C$
are continuous (see Proposition \ref{prop:step4} and \cite[Prop. 2.1]{blehervelez}), it will
suffice to show that the restrictions $F$ and $F_C$ of these functors to $\mathcal{C}$
are naturally isomorphic.   Without loss of generality we can assume $C = {\HH}^0(V^\bullet)$
and  $V^\bullet$ is the one-term complex with $V^0=C$ and $V^i=0$ for $i\neq 0$
(see Remark \ref{rem:functor}).

If $R$ is an Artinian ring in $\mathrm{Ob}(\mathcal{C})$ and $(M^\bullet,\phi)$ is a quasi-lift
of $V^\bullet$ over $R$, then $M^\bullet$ has non-zero cohomology only in degree $0$
by Lemma \ref{lem:lemma3.1}. Hence $(M^\bullet,\phi)$ is isomorphic to a quasi-lift
$({M'}^\bullet,\phi')$ where ${M'}^0={\HH}^0(M^\bullet)$ and
${M'}^i=0$ otherwise.  Since, by Lemma \ref{lem:lemma3.8}, $M^\bullet$ has finite pseudocompact
$R$-tor dimension at $0$, it follows by Remark \ref{rem:dumbdumb} and Lemma \ref{lem:lemma3.1} that
${M'}^0$ has projective dimension $0$ as abstract $R$-module. 
Hence ${M'}^0$ is an abstractly free $R$-module because $R$ is local Artinian. Therefore
$k \otimes_R {M'}^0=k \hat{\otimes}_R {M'}^0$ is isomorphic to $C = {\HH}^0(V^\bullet)$ since
$k \hat{\otimes}^{\LL}_R M^\bullet$ is isomorphic to $V^\bullet$ in the derived category.
We have now shown that $({M'}^0,{\phi'}^0)$ is a lift of $C$ over $R$, in the sense of \cite{blehervelez}, 
and the isomorphism class of $({M'}^0,{\phi'}^0)$ as a lift of $C$ over $R$ determines the isomorphism class of
$(M^\bullet,\phi)$ as a quasi-lift of $V^\bullet$ over $R$.  
Conversely, suppose  $(L,\lambda)$ is a lift of $C$ over $R$, in the sense of \cite{blehervelez}.
Then $L$ is an abstractly free  
$R$-module of rank equal to $\mathrm{dim}_k\,C$. Since $R$ is Artinian,
this implies that $L$ is a discrete $R\Lambda$-module of finite length, and hence a
pseudocompact $R\Lambda$-module. Thus the complex $L^\bullet$ with
$L^0=L$ and $L^i=0$ for $i\neq 0$ together with the morphism $\psi: k\otimes_R L^\bullet=k\hat{\otimes}_R L^\bullet \to V^\bullet$
in $C^-(\Lambda)$ given by $\psi^0=\lambda$ and $\psi^i=0$ for $i\neq 0$ defines a quasi-lift $(L^\bullet,\psi)$
of $V^\bullet$ over $R$.  This shows $F$ and $F_C$ are naturally isomorphic functors.
\end{proof} 

\begin{rem}
\label{rem:2term}
Suppose $V^\bullet$ has precisely two non-zero cohomology groups. Without
loss of generality, we can assume these groups are $U_0={\HH}^0(V^\bullet)$ and
$U_{-n}={\HH}^{-n}(V^\bullet)$ for some $n > 0$,  both of finite $k$-dimension by Hypothesis
\ref{hypo:fincoh}. 
We will also regard $U_0,U_{-n}$ as complexes concentrated in degree $0$. 
In particular, we obtain a distinguished triangle  in $D^-(\Lambda)$ of the form
\begin{equation}
\label{eq:anotherone}
U_{-n}[n] \xrightarrow{\iota} V^\bullet \xrightarrow{\pi} U_0
\xrightarrow{\beta} U_{-n}[n+1]\,.
\end{equation}
By the triangle axioms, there exists a complex $C(\beta)^\bullet$
in $D^-(\Lambda)$, which is unique up to isomorphism, such that 
$$U_0 \xrightarrow{\beta}  U_{-n}[n+1] \to C(\beta)^\bullet \to U_0[1]$$
is a distinguished triangle in $D^-(\Lambda)$. In other words, 
$V^\bullet \cong C(\beta)^\bullet[-1]$ in $D^-(\Lambda)$.

The statements and proofs of 
\cite[Sect. 9]{bcderived}
can be adapted to our situation. Here 
is a summary of some of the main results:

\begin{enumerate}
\item[(i)] If
the endomorphism rings of $U_0$ and $U_{-n}$ are both given
by scalars, then $\mathrm{Hom}_{D^-(\Lambda)}(V^\bullet,V^\bullet)= k$
if and only if
\begin{enumerate}
\item[(a)] $\mathrm{Ext}^n_{\Lambda}(U_0,U_{-n})=0$, and
\item[(b)] there exists a nontrivial element
$\beta\in\mathrm{Ext}^{n+1}_{\Lambda}(U_0,U_{-n})\cong
\mathrm{Hom}_{D^-(\Lambda)}(U_0,U_{-n}[n+1])$
such that $V^\bullet\cong C(\beta)^\bullet[-1]$
in $D^-(\Lambda)$.
\end{enumerate}

\item[(ii)]
We have the following description of the tangent space of the proflat deformation
functor $\hat{F}^{\fl}$ using Remark \ref{rem:bigthm}(i):

If $n\geq 2$, then $t_{F^{\fl}}=t_F$.
If $n=1$, then $t_{F^{\fl}}$ is the subspace of
$t_F=\mathrm{Ext}^1_{D^-(\Lambda)}(V^\bullet,V^\bullet)=
\mathrm{Hom}_{D^-(\Lambda)}(V^\bullet,V^\bullet[1])$
consisting of those elements $f\in t_F$ satisfying
$\pi[1]\circ f\circ \iota=0$ in $\mathrm{Hom}_{D^-(\Lambda)}(U_{-1}[1],U_0[1])$,
where $\iota$ and $\pi$ are as in (\ref{eq:anotherone}) for $n=1$.

\item[(iii)]
Suppose that $U_{-n}$ and $U_0$ have universal deformation rings $R_{-n}$ and $R_0$
and universal deformations $[X_{-n},\psi_{-n}]$ and $[X_0,\psi_0]$, respectively,
in the sense of \cite{blehervelez}, and that
$$\mathrm{dim}_k\, \mathrm{Ext}^1_{\Lambda}(U_{-n},U_{-n}) +
\mathrm{dim}_k \,\mathrm{Ext}^1_{\Lambda}(U_{0},U_{0}) =
\mathrm{dim}_k \;t_{F^{\fl}}\,.$$
Suppose furthermore that there exists a proflat quasi-lift
$(M^\bullet,\phi)$ of $V^\bullet$ over $R_{-n}\hat{\otimes}_k R_0$
such that
$${\HH}^{-n}(M^\bullet)\cong (R_{-n}\hat{\otimes}_k R_0)\hat{\otimes}_{R_{-n}}X_{-n}
\quad \mbox{ and } \quad
{\HH}^0(M^\bullet)\cong (R_{-n}\hat{\otimes}_k R_0)\hat{\otimes}_{R_0} X_0\,.$$
Then the versal proflat deformation ring $R^{\fl}(\Lambda,V^\bullet)$ is universal and
isomorphic to $R_{-n}\hat{\otimes}_k R_0$.
\end{enumerate}
\end{rem}

\begin{rem}
\label{rem:split}
Suppose $V^\bullet$ is isomorphic in
$D^-(\Lambda)$ to a complex whose differentials are trivial.  Thus
in $D^-(\Lambda)$, $V^\bullet$ is isomorphic to the direct sum
$\bigoplus_i {\HH}^{i}(V^\bullet)[-i]$, where there are only finitely
many non-zero terms in this sum and all terms have finite $k$-dimension by Hypothesis
\ref{hypo:fincoh}.  Without loss of generality,
we can assume that all the differentials of $V^\bullet$ are trivial,
so ${\HH}^{i}(V^\bullet) = V^{i}$ for all $i$.

In this situation, a \emph{split quasi-lift}
of $V^\bullet$ over an object $R$ of
$\hat{\mathcal{C}}$ is a proflat quasi-lift $(M^\bullet,\phi)$ of $V^\bullet$ over $R$ such that $M^\bullet$ is
isomorphic in $D^-(R\Lambda)$ to a complex whose differentials are trivial.
A split deformation is the isomorphism class of a split quasi-lift.
Let $\hat{F}^{\spl} = \hat{F}^{\spl}_{V^\bullet}:\hat{\mathcal{C}} \to \mathrm{Sets}$
be the functor that sends
each object $R$ of $\hat{\mathcal{C}}$ to the set
$\hat{F}^{\spl}(R)$ of all split deformations of $V^\bullet $ over $R$.

The statements and proofs of 
\cite[Sect. 11]{bcderived}
can be adapted to our situation. Here 
is a summary of some of the main results:

\begin{enumerate}
\item[(i)]
The functor $\hat{F}^{\spl}$ is isomorphic to the product of the functors on $\hat{\mathcal {C}}$ 
associated to deformations, in the sense of \cite{blehervelez},
of the non-zero cohomology groups of $V^\bullet$ considered as $\Lambda$-modules.
Moreover, the functor $\hat{F}^{\spl}$  is naturally isomorphic to the functor
$\hat{F}^{\fl}$.

\item[(ii)]
The versal split deformation ring $R^{\spl}(\Lambda,V^\bullet)$ associated to
the split deformation functor $\hat{F}^{\spl}$
is the tensor product $\hat{\bigotimes}_i R(\Lambda,V^i)$ over $k$ of the versal deformation
rings, in the sense of \cite{blehervelez}, of the non-zero cohomology groups of $V^\bullet$.
Suppose $[U(\Lambda,V^i), \phi_i]$ is the versal deformation of $V^i$, in the sense of \cite{blehervelez}, when $V^i={\HH}^i(V^\bullet) \ne \{0\}$.
A versal split deformation of $V^\bullet$ is represented by the direct sum
$$U^{\spl}(\Lambda,V^\bullet) = \bigoplus_i R^{\spl}(\Lambda,V^\bullet)
\hat{\otimes}_{R(\Lambda,V^i)} U(\Lambda,V^i)[-i]$$
where $i$ runs over those integers for which ${\HH}^i(V^\bullet) \ne \{0\}$, together with the morphism
$\phi_U:k\otimes_{R^{\spl}(\Lambda,V^\bullet)}U^{\spl}(\Lambda,V^\bullet) \to V^\bullet$
in $C^-(\Lambda)$ given by $\phi_U^i=\phi_i$ for all $i$ with ${\HH}^i(V^\bullet) \ne \{0\}$ and $\phi_U^i=0$ for all other $i$.

\item[(iii)] 
The natural map on tangent spaces $\tau: F^{\spl}(k[\varepsilon]) \to F(k[\varepsilon])$ may be
identified with the natural inclusion
$$\iota: \bigoplus_i
 \mathrm{Ext}^1_{\Lambda}(V^i, V^i)
\to \mathrm{Ext}^1_{D^{-}(\Lambda)}(V^\bullet,V^\bullet) = \bigoplus_{i,j}\mathrm{Ext}^{1+i-j}_{
\Lambda}(V^i,V^j).
$$

\item[(iv)]
We have a non-canonical surjective continuous $k$-algebra homomorphism 
$f_{\spl}:R(\Lambda,V^\bullet) \to R^{\spl}(\Lambda,V^\bullet)$.
If the versal deformation ring $R(\Lambda,V^i)$ is universal for all $i$,
then $R^{\spl}(\Lambda,V^\bullet)$ is a universal split deformation ring.  If in addition
$f_{\spl}$ is an isomorphism, then $R(\Lambda,V^\bullet)$ is a universal deformation ring for $V^\bullet$.
\end{enumerate}
\end{rem}


\section{Derived equivalences and stable equivalences of Morita type}
\label{s:derivedequivalences}
\setcounter{equation}{0}

In \cite{derivedeq}, the first author proved that if $k$ is a field of positive characteristic, and
$A$ and $B$ are block algebras of finite groups over a complete local commutative 
Noetherian ring with residue field $k$, then a split-endomorphism
two-sided tilting complex (as introduced by Rickard \cite{rickard2}) for the derived categories
of bounded complexes of finitely generated modules over $A$, resp. $B$,
preserves the versal deformation rings of bounded complexes of finitely generated
modules over $kA$, resp. $kB$.

It is the goal of Section \ref{s:derivedeq} to prove an analogous result, Theorem \ref{thm:deformations}, 
when $k$ is  an arbitrary field and $A$ and $B$ are replaced by arbitrary finite dimensional $k$-algebras.
In Section \ref{s:stableeq}, we will then study the connection to stable equivalences of Morita type for
self-injective algebras; see Propositions \ref{prop:stable1} and \ref{prop:stabmordef}. We assume the notation of Section \ref{s:udr}.

If $S$ is a ring,  $S\mbox{-mod}$ denotes
the category of finitely generated left  $S$-modules. 
Let $C^b(S\mbox{-mod})$  be the category of bounded complexes in $S\mbox{-mod}$,
let $K^b(S\mbox{-mod})$ be the homotopy category of $C^b(S\mbox{-mod})$,
and let $D^b(S\mbox{-mod})$  be the derived category of $K^b(S\mbox{-mod})$.


\subsection{Deformations of complexes and derived equivalences}
\label{s:derivedeq}

Recall from Section \ref{s:udr} that we view the finite dimensional $k$-algebra $\Lambda$
as a pseudocompact $k$-algebra with the discrete topology, and every finitely generated
$\Lambda$-module as a pseudocompact $\Lambda$-module with the discrete topology.
In particular, $D^b(\Lambda\mbox{-mod})$ can be identified with a full subcategory of
the derived category $D^-(\Lambda)$ of bounded above complexes of pseudocompact 
$\Lambda$-modules.

Suppose $\Gamma$ is another finite dimensional $k$-algebra, and $R\in\mathrm{Ob}(\hat{\mathcal{C}})$
is arbitrary. Then $R\Lambda$ and $R\Gamma$ are free $R$-modules of finite rank.
Rickard proved in \cite{rickard1} that the derived categories $D^b(R\Lambda\mbox{-mod})$ and
$D^b(R\Gamma\mbox{-mod})$ are equivalent as triangulated categories if
and only if there is a bounded complex $P_R^\bullet$ of finitely
generated $R\Gamma$-$R\Lambda$-bimodules and a bounded complex $Q_R^\bullet$ of finitely
generated $R\Lambda$-$R\Gamma$-bimodules such that
\begin{eqnarray}
\label{eq:2tilting}
Q_R^\bullet \otimes_{R\Gamma}^{\LL}P_R^\bullet &\cong& R\Lambda \quad
\,\mbox{in $D^b((R\Lambda\otimes_R R\Lambda^{op})\mbox{-mod})$, and}\\
P_R^\bullet \otimes^{\LL}_{R\Lambda}Q_R^\bullet
&\cong& R\Gamma \quad
\mbox{in $D^b((R\Gamma\otimes_R R\Gamma^{op})\mbox{-mod})$.} \nonumber
\end{eqnarray}
If $P_R^\bullet$ and $Q_R^\bullet$ exist, then the functors
\begin{eqnarray}
\label{eq:2tiltfunct} 
P_R^\bullet \otimes^{\LL}_{R\Lambda} -: &D^b(R\Lambda\mbox{-mod}) \to
D^b(R\Gamma\mbox{-mod}) &\quad \mbox{ and}\\
Q_R^\bullet \otimes^{\LL}_{R\Gamma} -: &D^b(R\Gamma\mbox{-mod})
\to D^b(R\Lambda\mbox{-mod}) &\nonumber
\end{eqnarray}
are equivalences of derived categories, and $Q_R^\bullet$ is
isomorphic to $\mathbf{R}\mathrm{Hom}_{R\Gamma}(P_R^\bullet,R\Gamma)$ in the
derived category of $R\Lambda$-$R\Gamma$-bimodules. The complexes $P_R^\bullet$
and $Q_R^\bullet$ are called \emph{two-sided tilting complexes} (see \cite[Def. 4.2]{rickard1}).

By \cite[Prop. 3.1]{rickard1}, $P_R^\bullet$ is isomorphic in $D^b(R\Gamma\mbox{-mod})$ 
(resp. in $D^b(R\Lambda^{op}\mbox{-mod})$) to a bounded complex $X^\bullet$ of finitely generated
projective left $R\Gamma$-modules (resp. a bounded complex $Y^\bullet$ of
finitely generated projective right $R\Lambda$-modules). 
Since $R\Gamma$ and $R\Lambda$ are free $R$-modules,
projective bimodules for these algebras 
are projective as left and right modules. Let 
$$T^\bullet: \qquad \cdots \to T^{n-1}\xrightarrow{\delta^{n-1}_T} T^{n}\xrightarrow{\delta^{n}_T}
T^{n+1}\to \cdots$$
be a projective
$R\Gamma$-$R\Lambda$-bimodule resolution of $P_R^\bullet$ such that all terms of $T^\bullet$
are finitely generated projective $R\Gamma$-$R\Lambda$-bimodules. By adding an acyclic complex
of finitely generated projective $R\Gamma$-$R\Lambda$-bimodules to $T^\bullet$ if necessary,
we can find quasi-isomorphisms
\begin{eqnarray}
\label{eq:qissurjective}
f:\quad &T^\bullet \to X^\bullet&\mbox{in $C^-(R\Gamma\mbox{-mod})$ and}\\
g:\quad &T^\bullet \to Y^\bullet&\mbox{in }C^-(R\Lambda^{op}\mbox{-mod})
\nonumber
\end{eqnarray}
that are surjective on terms.
More precisely, we construct the acyclic complex of finitely generated projective $R\Gamma$-$R\Lambda$-bimodules 
to be added to $T^\bullet$ inductively, starting from the right. Without loss of generality, assume $T^i=0=X^i=Y^i$ for $i>0$. 
Let $F^{-1}$ be a finitely generated free $R\Gamma$-$R\Lambda$-bimodule such that there exists a surjective $R\Lambda$-module homomorphism
$\phi_X^{-1}:F^{-1}\to X^{-1}$ and a surjective $R\Gamma$-module homomorphism $\phi_Y^{-1}:F^{-1}\to Y^{-1}$. 
Let $\psi_X^0:F^{-1}\to X^0$ be the composition $\delta_X^{-1}\circ\phi_X^{-1}$, and let 
$\psi_Y^0:F^{-1}\to Y^0$ be the composition $\delta_Y^{-1}\circ\phi_Y^{-1}$. Since $f^0:T^0\to X^0$ and $g^0:T^0\to Y^0$ induce
isomorphisms on the zero-th cohomology groups, it follows that $\tilde{f}^0=(\psi_X^0,f^0):F^{-1}\oplus T^0 \to X^0$ and $\tilde{g}^0=(\psi_Y^0,g^0): F^{-1}\oplus T^0\to Y^0$
are surjective. In other words, we have added the acyclic complex $F^{-1}\xrightarrow{\mathrm{id}}F^{-1}$, concentrated in the degrees $-1$ and $0$, to $T^\bullet$
to ensure that the resulting homomorphisms $\tilde{f}^0$ and $\tilde{g}^0$ are surjective. Inductively, we now work our way to the left and add acyclic complexes of
finitely generated free $R\Gamma$-$R\Lambda$-bimodules of the form $F^{i-1}\xrightarrow{\mathrm{id}}F^{i-1}$, concentrated in the degrees $i-1$ and $i$, to $T^\bullet$,
for $i\le -1$ such that there exists a surjective $R\Lambda$-module homomorphism $\phi_X^{i-1}:F^{i-1}\to X^{i-1}$ and a surjective $R\Gamma$-module homomorphism 
$\phi_Y^{i-1}:F^{i-1}\to Y^{i-1}$. We also let $\psi_X^i:F^{i-1}\to X^i$ be the composition $\delta_X^{i-1}\circ\phi_X^{i-1}$, and we let 
$\psi_Y^i:F^{i-1}\to Y^i$ be the composition $\delta_Y^{i-1}\circ\phi_Y^{i-1}$, to receive surjective $R\Gamma$-module homomorphisms 
$\tilde{f}^i=(\psi_X^i,\phi_X^i,f^i):F^{i-1}\oplus F^i\oplus T^i\to X^i$ and surjective $R\Lambda$-module homomorphisms 
$\tilde{g}^i=(\psi_Y^i,\phi_Y^i,g^i):F^{i-1}\oplus F^i\oplus T^i\to Y^i$ for all $i\le -1$. Replacing $f$ by $\tilde{f}$ and
$g$ by $\tilde{g}$, if necessary, we arrive at quasi-isomorphisms $f$, $g$ as in (\ref{eq:qissurjective}) that are 
surjective on terms.

Since $\mathrm{Ker}(f)$ (resp. $\mathrm{Ker}(g)$) is an acyclic complex of projective
left $R\Gamma$-modules (resp. projective right $R\Lambda$-modules), it splits completely.
If $X^i=0$ and $Y^i=0$ for all $i\le n$, it follows that 
$\mathrm{Ker}(\delta^{n}_T)$ is isomorphic to the kernel of the $n$-th differential of
$\mathrm{Ker}(f)$ (resp. $\mathrm{Ker}(g)$) as a left $R\Gamma$-module (resp. right $R\Lambda$-module).
But the latter kernel is projective as a left $R\Gamma$-module (resp. right $R\Lambda$-module).
Hence we can truncate $T^\bullet$ at $n$:
$$\cdots \to 0\to \mathrm{Ker}(\delta^{n}_T)\to T^n\xrightarrow{\delta^{n}_T}
T^{n+1}\to \cdots$$
to produce a bounded complex that is isomorphic to $P_R^\bullet$ 
in $D^b((R\Gamma\otimes_R R\Lambda^{op})\mbox{-mod})$ with the additional property that all the
terms of this complex are projective as left and as right modules and that all terms, except possibly
the leftmost non-zero term, are actually projective as bimodules.
Similarly, we can assume that all terms of
$Q_R^\bullet$ are projective as left $R\Lambda$-modules and as right
$R\Gamma$-modules and that all  terms, except possibly
the leftmost non-zero term, are actually
projective as $R\Lambda$-$R\Gamma$-bimodules. In particular, we can take
$Q_R^\bullet$ to be the $R\Gamma$-dual
\begin{equation}
\label{eq:lambdadual}
\widetilde{P}_R^\bullet = \mathrm{Hom}_{R\Gamma}(P_R^\bullet,R\Gamma).
\end{equation}
In this situation, (\ref{eq:2tilting}) is equivalent to
\begin{eqnarray}
\label{eq:2tiltbetter} 
\widetilde{P}_R^\bullet \otimes_{R\Gamma}P_R^\bullet 
&\cong& R\Lambda \qquad\,
\mbox{in $D^b((R\Lambda\otimes_R R\Lambda^{op})\mbox{-mod})$, and}\\
P_R^\bullet \otimes_{R\Lambda}\widetilde{P}_R^\bullet
&\cong&  R\Gamma \qquad
\mbox{in $D^b((R\Gamma\otimes_R R\Gamma^{op})\mbox{-mod})$.} \nonumber
\end{eqnarray}

\begin{dfn}
\label{def:nice2sidedtilting}
Let $R\in\mathrm{Ob}(\hat{\mathcal{C}})$. Suppose $P_R^\bullet$ is a bounded complex of finitely 
generated $R\Gamma$-$R\Lambda$-bimodules.
\begin{enumerate}
\item[(a)]
We call $P_R^\bullet$
a \emph{nice two-sided tilting
complex}, if all terms of $P_R^\bullet$ are projective as left $R\Gamma$-modules and
as right $R\Lambda$-modules and (\ref{eq:2tiltbetter}) is satisfied for $\widetilde{P}_R^\bullet$ as in
(\ref{eq:lambdadual}).

\item[(b)]
If $R=k$, then we write $P^\bullet=P_R^\bullet$ and $\widetilde{P}^\bullet = \mathrm{Hom}_{\Gamma}(P^\bullet,\Gamma)$. In other words, $P^\bullet$ is a \emph{nice two-sided tilting complex} if $P^\bullet$ is a 
bounded complex of finitely generated $\Gamma$-$\Lambda$-bimodules such that all terms of
$P^\bullet$ are projective as left $\Gamma$-modules and as right $\Lambda$-modules and such that
$\widetilde{P}^\bullet \otimes_{\Gamma}P^\bullet \cong \Lambda$ in 
$D^b((\Lambda\otimes_k \Lambda^{op})\mbox{-mod})$, and
$P^\bullet \otimes_{\Lambda}\widetilde{P}^\bullet \cong \Gamma$ in $D^b((\Gamma\otimes_k \Gamma^{op})\mbox{-mod})$.
\end{enumerate}
\end{dfn}

\begin{rem}
\label{rem:splitendo}
If $\Lambda$ and $\Gamma$ are symmetric $k$-algebras, then 
$R\Lambda$ and $R\Gamma$ are symmetric $R$-algebras for all
$R\in\mathrm{Ob}(\hat{\mathcal{C}})$, in the sense that $R\Lambda$ (resp. $R\Gamma$),
considered as a bimodule over itself, is isomorphic to its $R$-linear dual
$\mathrm{Hom}_R(R\Lambda,R)$ (resp. $\mathrm{Hom}_R(R\Gamma,R)$). In particular,
the functors $\mathrm{Hom}_{R\Gamma}(-,R\Gamma)$ and $\mathrm{Hom}_R(-,R)$ are
then naturally isomorphic. Therefore, under the assumptions made prior to (\ref{eq:lambdadual}), we may take $Q_R^\bullet$
to be the $R$-dual
\begin{equation}
\label{eq:kdual}
\widecheck{P}_R^\bullet = \mathrm{Hom}_R(P_R^\bullet,R).
\end{equation}
Rickard calls a bounded complex $P_R^\bullet$ of finitely generated
$R\Gamma$-$R\Lambda$-bimodules a \emph{split-endomorphism two-sided tilting
complex}, if all terms of $P_R^\bullet$ are projective as left $R\Gamma$-modules and
as right $R\Lambda$-modules and (\ref{eq:2tiltbetter}) is satisfied with 
$\widetilde{P}_R^\bullet$ replaced by $\widecheck{P}_R^\bullet$ (see \cite[p. 336]{rickard2}).
\end{rem}

\begin{lemma}
\label{lem:lifting} 
Suppose $\Lambda$ and $\Gamma$ are finite dimensional $k$-algebras,
$P^\bullet$ is a nice two-sided tilting complex in 
$D^b((\Gamma\otimes_k\Lambda^{op})\mbox{-$\mathrm{mod}$})$ 
as in Definition $\ref{def:nice2sidedtilting}$, and
$R\in\mathrm{Ob}(\hat{\mathcal{C}})$. Then $P_R^\bullet=R\otimes_k
P^\bullet$ is a nice two-sided tilting complex in 
$D^b((R\Gamma\otimes_R R\Lambda^{op})\mbox{-$\mathrm{mod}$})$.
Moreover, 
\begin{eqnarray}
\label{eq:useful} 
\widetilde{P}_R^\bullet \hat{\otimes}_{R\Gamma}P_R^\bullet
&\cong&  R\Lambda \qquad
\mbox{in $D^-(R\Lambda\otimes_R R\Lambda^{op})$, and}\\
P_R^\bullet \hat{\otimes}_{R\Lambda}\widetilde{P}_R^\bullet 
&\cong& R\Gamma \qquad\,
\mbox{in $D^-(R\Gamma\otimes_R R\Gamma^{op})$.} \nonumber
\end{eqnarray}
In particular, the functors $P_R^\bullet\hat{\otimes}_{R\Lambda}-$ and
$\widetilde{P}_R^\bullet\hat{\otimes}_{R\Gamma}-$ provide quasi-inverse equivalences: 
\begin{eqnarray}
\label{eq:quasiinverses}
P_R^\bullet \hat{\otimes}_{R\Lambda} -: &D^-(R\Lambda) \to
D^-(R\Gamma)\,,& \mbox{and}\\
\widetilde{P}_R^\bullet \hat{\otimes}_{R\Gamma} -: &D^-(R\Gamma)
\to D^-(R\Lambda)\,.& \nonumber
\end{eqnarray}
\end{lemma}

\begin{proof}
We have
$$R\otimes_k(\Lambda\otimes_k \Lambda^{op})\cong R\Lambda\otimes_R R\Lambda^{op}$$
and 
$$R\otimes_k(\Gamma\otimes_k\Gamma^{op})\cong R\Gamma\otimes_R R\Gamma^{op}$$
as pseudocompact $R$-algebras. Therefore,
it follows from \cite[Lemma 4.3]{rickard1} that $P_R^\bullet=R\otimes_kP^\bullet$ is a two-sided tilting
complex in $D^b((R\Gamma\otimes_R R\Lambda^{op})\mbox{-$\mathrm{mod}$})$.
Note that $R\otimes_k\widetilde{P}^\bullet=R\otimes_k\mathrm{Hom}_{\Gamma}(P^\bullet,\Gamma)
\cong \mathrm{Hom}_{R\Gamma}(P_R^\bullet,R\Gamma)=\widetilde{P}_R^\bullet$ 
as complexes of $R\Lambda$-$R\Gamma$-bimodules.
Since all terms of $P^\bullet$ are projective as left $\Gamma$-modules and as right $\Lambda$-modules,
it follows that all terms of $P_R^\bullet$ are projective as left $R\Gamma$-modules and
as right $R\Lambda$-modules.
Since $R\otimes_k(\widetilde{P}^\bullet \otimes_{\Gamma}P^\bullet)
\cong \widetilde{P}_R^\bullet \otimes_{R\Gamma}P_R^\bullet$
in $C^b((R\Lambda\otimes_R R\Lambda^{op})\mbox{-mod})$
and since $R\otimes_k(P^\bullet \otimes_{\Lambda}\widetilde{P}^\bullet )
\cong P_R^\bullet \otimes_{R\Lambda}\widetilde{P}_R^\bullet$ 
in $C^b((R\Gamma\otimes_R R\Gamma^{op})\mbox{-mod})$, we obtain
(\ref{eq:2tiltbetter}). In other words, $P_R^\bullet$ is a nice two-sided tilting complex in 
$D^b((R\Gamma\otimes_R R\Lambda^{op})\mbox{-$\mathrm{mod}$})$.

To prove the remaining statements of Lemma \ref{lem:lifting}, we consider the isomorphisms
in (\ref{eq:2tiltbetter}) more closely. First, we note that since the terms of 
$P_R^\bullet$ are finitely generated projective left $R\Gamma$-modules and finitely generated
projective right $R\Lambda$-modules and since the terms of $\widetilde{P}_R^\bullet$ are finitely 
generated projective left $R\Lambda$-modules and finitely generated
projective right $R\Gamma$-modules, we can replace the tensor products by completed tensor
products. Moreover, since $\widetilde{P}^\bullet \otimes_{\Gamma}P^\bullet\cong  \Lambda$
in $D^b((\Lambda\otimes_k \Lambda^{op})\mbox{-mod})$, it follows that this is also true
in the derived category $D^-(\Lambda\otimes_k \Lambda^{op})$ of bounded above
pseudocompact $\Lambda$-$\Lambda$-bimodules. Similarly, 
$P^\bullet \otimes_{\Lambda}\widetilde{P}^\bullet\cong  \Gamma$ in the derived category 
$D^-(\Gamma\otimes_k \Gamma^{op})$ of bounded above
pseudocompact $\Gamma$-$\Gamma$-bimodules.  Therefore, we obtain
\begin{eqnarray}
\label{eq:oyoyoy} 
R\otimes_k(\widetilde{P}^\bullet \otimes_{\Gamma}P^\bullet )
&\cong&  R\otimes_k\Lambda=R\Lambda \qquad
\mbox{in $D^-(R\Lambda\otimes_R R\Lambda^{op})$, and}\\
R\otimes_k(P^\bullet \otimes_{\Lambda}\widetilde{P}^\bullet)
&\cong& R\otimes_k\Gamma\,=R\Gamma \qquad
\mbox{in $D^-(R\Gamma\otimes_R R\Gamma^{op}).$} \nonumber
\end{eqnarray}
Since 
$$R\otimes_k(\widetilde{P}^\bullet \otimes_{\Gamma}P^\bullet )
\cong \widetilde{P}_R^\bullet \otimes_{R\Gamma}P_R^\bullet
\cong \widetilde{P}_R^\bullet \hat{\otimes}_{R\Gamma}P_R^\bullet$$ 
in $C^-(R\Lambda\otimes_R R\Lambda^{op})$
and since 
$$R\otimes_k(P^\bullet \otimes_{\Lambda}\widetilde{P}^\bullet)
\cong P_R^\bullet \otimes_{R\Lambda}\widetilde{P}_R^\bullet
\cong P_R^\bullet \hat{\otimes}_{R\Lambda}\widetilde{P}_R^\bullet$$
in $C^-(R\Gamma\otimes_R R\Gamma^{op})$, 
the isomorphisms in (\ref{eq:useful}) follow.
This proves that the functors in (\ref{eq:quasiinverses}) are quasi-inverses
between the derived categories of bounded above complexes of pseudocompact 
$R\Lambda$- and $R\Gamma$-modules, completing the proof of Lemma \ref{lem:lifting}.
\end{proof}

\begin{rem}
\label{rem:leftderivedtensor2}
Suppose $\Lambda$, $\Gamma$, $P^\bullet$ and $R$ are as in Lemma \ref{lem:lifting}. 
Then Lemma \ref{lem:lifting} implies, in particular, that $P_R^\bullet$ is a bounded complex 
of finitely generated $R\Gamma$-$R\Lambda$-bimodules whose terms are all projective
as abstract right $R\Lambda$-modules. Hence we can use Remark \ref{rem:topflat} to see that
all terms of $P_R^\bullet$ are projective objects in the category of pseudocompact right 
$R\Lambda$-modules. Therefore, the functor $P_R^\bullet\hat{\otimes}_{R\Lambda}-:K^-(R\Lambda) \to K^-(R\Gamma)$
sends acyclic complexes to acyclic complexes (see, for example, \cite[Acyclic Assembly Lemma 2.7.3]{Weibel}).
This implies that $P_R^\bullet\hat{\otimes}_{R\Lambda}-$ is its own left derived functor, i.e.
$$P_R^\bullet\hat{\otimes}^{\LL}_{R\Lambda}-\;=\; P_R^\bullet\hat{\otimes}_{R\Lambda}-\; :\quad D^-(R\Lambda)\;\to\; D^-(R\Gamma)$$
(see, for example, \cite[Ex. 10.5.5]{Weibel}).
Similarly, we see that
$$\widetilde{P}_R^\bullet\hat{\otimes}^{\LL}_{R\Gamma}-\;=\; \widetilde{P}_R^\bullet\hat{\otimes}_{R\Gamma}-\; :\quad D^-(R\Gamma)\;\to\; D^-(R\Lambda)\,.$$
In particular, we can identify the functors $P^\bullet\hat{\otimes}^{\LL}_\Lambda-$, $P^\bullet\hat{\otimes}_\Lambda-$ and
$P^\bullet\otimes_\Lambda-$ as functors from $D^-(\Lambda)$ to $D^-(\Gamma)$.
\end{rem}

\begin{thm}
\label{thm:deformations}
Suppose $\Lambda$ and $\Gamma$ are finite dimensional $k$-algebras,
and $P^\bullet$ is a nice two-sided tilting complex in 
$D^b((\Gamma\otimes_k\Lambda^{op})\mbox{-$\mathrm{mod}$})$
as in Definition $\ref{def:nice2sidedtilting}$. 
Let $V^\bullet$ be  a complex in $D^-(\Lambda)$ satisfying Hypothesis 
$\ref{hypo:fincoh}$, and let ${V'}^\bullet =P^\bullet\hat{\otimes}^{\LL}_\Lambda
V^\bullet=P^\bullet\otimes_\Lambda V^\bullet$. Then the deformation functors $\hat{F}_{V^\bullet}$ and 
$\hat{F}_{{V'}^\bullet}$ are naturally isomorphic.
In particular, the versal deformation rings $R(\Lambda,V^\bullet)$ and $R(\Gamma,{V'}^\bullet)$ are isomorphic
in $\hat{\mathcal{C}}$, and $R(\Lambda,V^\bullet)$ is a universal deformation ring of $V^\bullet$ if and only if $R(\Gamma,{V'}^\bullet)$ is a 
universal deformation ring of ${V'}^\bullet$.
\end{thm}

\begin{proof}
By Remarks \ref{rem:functor} and \ref{rem:leftderivedtensor2},
we may assume, without loss of generality, that $V^\bullet$ is a bounded above complex of  topologically free pseudocompact $\Lambda$-modules.
We use the notation from Lemma \ref{lem:lifting}.
Let $R\in\mathrm{Ob}(\hat{\mathcal{C}})$ and let $(M^\bullet, \phi)$ be a quasi-lift of
$V^\bullet$ over $R$. By Remark \ref{rem:profree}(i), we can assume that the terms of
$M^\bullet$ are topologically free pseudocompact $R\Lambda$-modules. Since $M^\bullet$ has finite
pseudocompact $R$-tor dimension, we can truncate $M^\bullet$ to obtain 
a complex $N^\bullet$ that is isomorphic to $M^\bullet$ in $D^-(R\Gamma)$ such that $N^\bullet$
is a bounded complex of $R\Lambda$-modules, all of which are topologically free as $R$-modules. 

Define ${M'}^\bullet=P_R^\bullet \hat{\otimes}_{R\Lambda}M^\bullet$ and
${N'}^\bullet=P_R^\bullet \hat{\otimes}_{R\Lambda}N^\bullet$,
so  ${M'}^\bullet$ and ${N'}^\bullet$ are isomorphic objects of $D^-(R\Gamma)$.
Since $P_R^\bullet$ is a bounded complex of finitely generated projective right $R\Lambda$-modules and 
since $N^\bullet$ is a bounded complex of topologically free
$R$-modules, it follows that ${N'}^\bullet$ is also a bounded complex of topologically free $R$-modules.
But this means that there exists an integer $n$ such that for all pseudocompact $R$-modules $S$
and all integers $i< n$ we have $\HH^i(S\hat{\otimes}_R{N'}^\bullet)=
\HH^i(S\hat{\otimes}^{\LL}_R{N'}^\bullet)=0$. Therefore ${N'}^\bullet$, and hence ${M'}^\bullet$, both have 
finite pseudocompact $R$-tor dimension.

Next we note that we can view $K^-(\Lambda)$ as the full subcategory of $K^-(R\Lambda)$
consisting of bounded above complexes of pseudocompact $R\Lambda$-modules
on which the maximal ideal $m_R$ of $R$ acts trivially. Moreover, on $K^-(\Lambda)$ the functor 
$P_R^\bullet \hat{\otimes}_{R\Lambda} -$ coincides with the functor 
$P^\bullet \hat{\otimes}_{\Lambda} -:K^-(\Lambda)\to K^-(\Gamma)$. 
Viewing both functors as functors on $K^-(\Lambda)$, their corresponding left derived functors
$L^-(P_R^\bullet \hat{\otimes}_{R\Lambda} -)$ and  $L^-(P^\bullet \hat{\otimes}_{\Lambda} -)= P^\bullet\hat{\otimes}^{\LL}_\Lambda-=
P^\bullet \hat{\otimes}_{\Lambda}-$ coincide as functors from $D^-(\Lambda)$ to $D^-(\Gamma)$.
Define $\phi'=L^-(P_R^\bullet \hat{\otimes}_{R\Lambda}-)(\phi)$, so $\phi'$ is an isomorphism in 
$D^-(\Gamma)$. Since $M^\bullet$ is a bounded above complex of topologically free pseudocompact
$R\Lambda$-modules and since we assumed that $V^\bullet$ is a bounded above complex of 
topologically free pseudocompact $\Lambda$-modules, we obtain 
$\phi'=P_R^\bullet \hat{\otimes}_{R\Lambda}\phi$. Therefore, 
$${M'}^\bullet\hat{\otimes}_R^{\LL}k ={M'}^\bullet\hat{\otimes}_Rk =
(P_R^\bullet \hat{\otimes}_{R\Lambda}M^\bullet)\hat{\otimes}_Rk=
P_R^\bullet \hat{\otimes}_{R\Lambda}(M^\bullet\hat{\otimes}_Rk)
\xrightarrow{\phi'} P_R^\bullet \hat{\otimes}_{R\Lambda}V^\bullet
= {V'}^\bullet$$
which means $({M'}^\bullet,\phi')$ is a quasi-lift of ${V'}^\bullet$. 
It follows that for each $R\in\mathrm{Ob}(\hat{\mathcal{C}})$,
the functor $P_R^\bullet \hat{\otimes}_{R\Lambda}-$ induces a bijection
$\tau_R$ from the set of deformations of $V^\bullet$ over $R$ onto the set of deformations of 
${V'}^\bullet$ over $R$. 

It remains to show that the maps $\tau_R$ are natural with respect to morphisms $\alpha:R\to R'$ in 
$\hat{\mathcal{C}}$. Considering $(M^\bullet, \phi)$ and $({M'}^\bullet, \phi')$ as above, it suffices to
show that $(R'\hat{\otimes}_{R,\alpha}{M'}^\bullet,(\phi')_\alpha)$ and 
$(P_{R'}^\bullet \hat{\otimes}_{R'\Lambda} (R'\hat{\otimes}_{R,\alpha}M^\bullet),P_{R'}^\bullet \hat{\otimes}_{R'\Lambda}(\phi_\alpha))$ 
are isomorphic as quasi-lifts of ${V'}^\bullet$ over $R'$.
Since all the terms of $P_R^\bullet$ are finitely generated projective right $R\Lambda$-modules and
since all the terms of $M^\bullet$ are topologically free as $R$-modules, it follows
that there is a natural isomorphism
$$f\;:\quad {M'}^\bullet \hat{\otimes}_{R,\alpha}R'= 
(P_R^\bullet \hat{\otimes}_{R\Lambda}M^\bullet)\hat{\otimes}_{R,\alpha}R'
\to P_{R'}^\bullet \hat{\otimes}_{R'\Lambda}(M^\bullet\hat{\otimes}_{R,\alpha} R')$$
in $D^-(R'\Gamma)$ (in fact in $C^-(R'\Gamma)$). Moreover, the diagram
$$
\xymatrix{
\left({M'}^\bullet\hat{\otimes}_{R,\alpha}R'\right)\hat{\otimes}_{R'}k
\ar[rr]^(.45){f\hat{\otimes}_{R'}k} \ar[d]_{\cong}& &
\left(P_{R'}^\bullet \hat{\otimes}_{R'\Lambda}(M^\bullet\hat{\otimes}_{R,\alpha} R')\right)\hat{\otimes}_{R'}k
\ar[d]^{\cong}\\
{M'}^\bullet\hat{\otimes}_Rk\ar[rd]_{\phi'} && P_{R'}^\bullet \hat{\otimes}_{R'\Lambda}(M^\bullet\hat{\otimes}_R k)\ar[ld]^{P_{R'}^\bullet \hat{\otimes}_{R'\Lambda}\phi}
\\&{V'}^\bullet&
}$$
commutes in $D^-(\Gamma)$. Hence $(R'\hat{\otimes}_{R,\alpha}{M'}^\bullet,(\phi')_\alpha) \cong (P_{R'}^\bullet \hat{\otimes}_{R'\Lambda} (R'\hat{\otimes}_{R,\alpha}M^\bullet),P_{R'}^\bullet \hat{\otimes}_{R'\Lambda}(\phi_\alpha))$.

This means that the functors $\hat{F}_{V^\bullet}$ and $\hat{F}_{{V'}^\bullet}$ are naturally 
isomorphic, which implies Theorem \ref{thm:deformations}.
\end{proof}


\subsection{Stable equivalences of Morita type for self-injective algebras}
\label{s:stableeq}

We use the notation introduced in Section \ref{s:derivedeq}. Moreover, we assume throughout this section
that both $\Lambda$ and $\Gamma$ are self-injective finite dimensional $k$-algebras.

If $V$ is any finitely generated $\Lambda$-module or $\Gamma$-module, let $\hat{F}_V:\hat{\mathcal{C}}\to\mathrm{Sets}$ be the
deformation functor considered in \cite{blehervelez} (see also Remark \ref{rem:deformmodules}), and let
$F_V$ be its restriction to the full subcategory $\mathcal{C}$ of $\hat{\mathcal{C}}$ consisting of Artinian rings.

We first collect some useful facts, which were proved as Claims 1, 2 and 6
in the proof of \cite[Thm. 2.6]{blehervelez}, only using the assumption that $\Lambda$ is a self-injective
finite dimensional $k$-algebra. 
Note that we need to add the assumption that $M$ (resp. $M_0$) is free over $R$ (resp. $R_0$) in Claims 1 and 2
in the proof of \cite[Thm. 2.6]{blehervelez}. This makes no difference in the overall proof of \cite[Thm. 2.6]{blehervelez}
since these claims were only used under this assumption.

\begin{rem}
\label{rem:josepaper}
Suppose $\Lambda$ is a self-injective finite dimensional $k$-algebra.
Let $R,R_0$ be Artinian rings in $\mathcal{C}$, and let $\pi:R\to R_0$ be a surjection in $\mathcal{C}$.
Let $M$, $Q$ (resp. $M_0$, $Q_0$) be finitely generated $R\Lambda$-modules
(resp. $R_0\Lambda$-modules) that are free over $R$ (resp. $R_0$),
and assume that $Q$ (resp. $Q_0$) is projective. 
Suppose there are
$R_0\Lambda$-module isomorphisms $g:R_0\otimes_{R,\pi} M \to M_0$,
$h:R_0\otimes_{R,\pi} Q\to Q_0$. 
\begin{enumerate}
\item[(i)] 
If $\nu_0\in \mathrm{Hom}_{R_0\Lambda}(M_0, Q_0)$, then there
exists $\nu\in \mathrm{Hom}_{R\Lambda}(M, Q)$ with $\nu_0=
h\circ(R_0\otimes_{R,\pi}\nu)\circ g^{-1}$.
\item[(ii)]
If $\sigma_0\in \mathrm{End}_{\Lambda}(M_0)$  factors through a projective 
$R_0 \Lambda$-module, then there exists $\sigma
\in  \mathrm{End}_{R\Lambda}(M)$ such that $\sigma$ factors through a projective
$R\Lambda$-module and $\sigma_0=g\circ (R_0\otimes_{R,\pi}\sigma)\circ g^{-1}$.
\end{enumerate}
Let $P$ be a finitely generated projective $\Lambda$-module. 
Let $\iota_R:k\to R$ be the unique morphism in $\mathcal{C}$ endowing $R$ with a $k$-algebra structure,
and let $\pi_R:R\to k$ be the morphism from $R$ to its residue field $k$ in $\mathcal{C}$. Then 
$\pi_R\circ \iota_R$ is the identity on $k$, and $P_R=R\otimes_{k,\iota_R} P$ is a
projective $R\Lambda$-module cover of $P$, which is unique up to isomorphism. 
In particular, $(P_R,\pi_{R,P})$ is a lift of $P$ over $R$, where $\pi_{R,P}$ is the natural
isomorphism $k\otimes_{R,\pi_R}(R\otimes_{k,\iota_R}P)\to P$ of $\Lambda$-modules.
\begin{enumerate}
\item[(iii)]
Suppose there is a commutative
diagram of finitely generated $R\Lambda$-modules
\begin{equation}
\label{eq:whatever}
\xymatrix{
0\ar[r]&P_R\ar[d]\ar[r]^{g}&T\ar[d]\ar[r]^{h}& C\ar[d]\ar[r]&0\\
0\ar[r]&P\ar[r]^{\overline{g}}&k\otimes_R T\ar[r]^{\overline{h}}&k\otimes_R C\ar[r]&0}
\end{equation}
in which $T$ and $C$ are free over $R$ and the bottom row arises by tensoring the top row with 
$k$ over $R$ and using the $\Lambda$-module isomorphism $\pi_{R,P}:k\otimes_{R,\pi_R} P_R\to P$. 
Then the top row of $(\ref{eq:whatever})$ splits as a sequence of $R\Lambda$-modules.
\end{enumerate}
\end{rem}

We need the following result, which is a generalization of \cite[Thm. 2.6(iii)]{blehervelez}.

\begin{lemma}
\label{lem:addproj}
Suppose $\Lambda$ is a self-injective finite dimensional $k$-algebra and that 
$V$ and $P$ are finitely generated non-zero $\Lambda$-modules and $P$ is projective. 
Then $P$ has a universal deformation ring $R(\Lambda,P)$ and $R(\Lambda,P)\cong k$.
The deformation functors $\hat{F}_V$ and $\hat{F}_{V\oplus P}$ are naturally isomorphic.
In particular, the versal deformation rings $R(\Lambda,V)$ and $R(\Lambda,V\oplus P)$ are isomorphic
in $\hat{\mathcal{C}}$, and $R(\Lambda,V)$ is universal if and only if $R(\Lambda,V\oplus P)$ is universal. 
\end{lemma}

\begin{proof}
Since $P$ is a projective $\Lambda$-module, it follows that $\mathrm{Ext}^1_\Lambda(P,P)=0$, which
implies by \cite[Prop. 2.1]{blehervelez} that the versal deformation ring of $P$ is isomorphic to $k$. 
For each $R\in\mathrm{Ob}(\hat{\mathcal{C}})$, let $\iota_R:k\to R$ be the unique morphism in 
$\hat{\mathcal{C}}$ endowing $R$ with a $k$-algebra structure, and let $\pi_R:R\to k$ be the morphism 
from $R$ to its residue field $k$ in $\hat{\mathcal{C}}$. Then $\pi_R\circ \iota_R$ is the identity morphism of $k$. 
This implies that $k$ is the universal deformation ring of $P$.

Let $R\in\mathrm{Ob}(\mathcal{C})$ be Artinian. Define a map
\begin{eqnarray}
\label{eq:bijproj}
F_V(R)&\to&F_{V\oplus P}(R)\,,\\
{[M,\phi]} &\mapsto&[M\oplus P_R,\phi\oplus \pi_{R,P}]\nonumber
\end{eqnarray}
where $P_R=R\otimes_{k,\iota_R}P$ and $\pi_{R,P}:k\otimes_{R,\pi_R} P_R\to P$ are as in 
Remark \ref{rem:josepaper}. Then (\ref{eq:bijproj}) is a well-defined map that is natural with respect to morphisms
$\alpha:R\to R'$ in $\mathcal{C}$, since $\alpha\circ\iota_R=\iota_{R'}$ and
$\pi_{R'}\circ\alpha=\pi_R$. Since the deformation functors $\hat{F}_V$ and $\hat{F}_{V\oplus P}$
are continuous, it suffices to show that the
map (\ref{eq:bijproj}) is bijective for all $R\in\mathrm{Ob}(\mathcal{C})$ to complete the proof of Lemma \ref{lem:addproj}. 

Suppose first that $(M,\phi)$ and $(M',\phi')$ are two lifts of $V$ over $R$ such that
there exists an $R\Lambda$-module isomorphism 
$f=\left(\begin{array}{cc}f_{11}&f_{12}\\f_{21}&f_{22}\end{array}\right): M\oplus P_R\to M'\oplus P_R$ with
$(\phi'\oplus \pi_{R,P})\circ(k\otimes f) = \phi\oplus \pi_{R,P}$. In particular, $\phi'\circ(k\otimes f_{11})=\phi$
and $k\otimes f_{22}$ is the identity morphism on $P$. By Nakayama's Lemma, this implies that $f_{11}$ and $f_{22}$ 
are $R\Lambda$-module isomorphisms, since $M$ and $M'$ are free $R$-modules of the same finite rank.
Therefore, $[M,\phi]=[M',\phi']$ and the map $(\ref{eq:bijproj})$ is injective.

We now show that the map $(\ref{eq:bijproj})$ is surjective. Let $(T,\tau)$ be a lift of $V\oplus P$ over $R$. Since
$P_R$ is a projective $R\Lambda$-module, there exists an $R\Lambda$-module homomorphism $g$ that makes
the following diagram commute
\begin{equation}
\label{eq:need1}
\xymatrix{ P_R\ar[r]^{g} \ar[d]_{\footnotesize{\left(\begin{array}{c}0\\\mathrm{pr}\end{array}\right)}}& T\ar[d]\\
V\oplus P \ar[r]^{\tau^{-1}}&k\otimes_R T}
\end{equation}
where $\mathrm{pr}$ is obtained by first tensoring $P_R$ with
$k$ over $R$ and then using the $\Lambda$-module isomorphism 
$\pi_{R,P}:k\otimes_{R,\pi_R} P_R\to P$. Then $g$ induces an injective $\Lambda$-module homomorphism
$g'$ and a commutative diagram of $\Lambda$-modules 
\begin{equation}
\label{eq:need2}
\xymatrix{P_R/m_R\,P_R\ar[r]^{g'} \ar[d]_{\pi_{R,P}}& T/m_R T\ar@{=}[dd]\\
P\ar[d]_{\iota_P}&\\
V\oplus P&k\otimes_R T\ar[l]^{\tau}}\\
\end{equation}
where $\iota_P: P \to V\oplus P$ is the natural injection and, 
as before, $m_R$ denotes the maximal ideal of $R$. Using Nakayama's Lemma, it follows that $g$
is injective. Letting $C$ be the cokernel of $g$, we obtain a commutative diagram (\ref{eq:whatever}). Since
$P_R$ and $T$ are free $R$-modules and since $g$ and $g'$ (and hence $\overline{g}=g'\circ \pi_{R,P}^{-1}$) are injective, an
elementary Nakayama's Lemma argument shows that $C$ is also free as an $R$-module. 
Therefore, by Remark \ref{rem:josepaper}(iii), the top row of (\ref{eq:whatever}) splits as a sequence of 
$R\Lambda$-modules. Let $j:C\to T$ be an $R\Lambda$-module splitting of $h$. By tensoring with $k$ over $R$,
we obtain a $\Lambda$-module splitting $\overline{j}:k\otimes_RC \to k\otimes_R T$ of $\overline{h}$.
Consider the $R\Lambda$-module isomorphism
$$(j,g):\quad  C\oplus P_R \to T\,.$$
Since $\tau\circ\overline{g}=\tau\circ g'\circ \pi_{R,P}^{-1}=\iota_P$ by (\ref{eq:need2}),
there exists a $\Lambda$-module isomorphism $\xi:k\otimes_R C\to V$ such that
$p_V\circ\tau = \xi\circ \overline{h}$, where $p_V:V\oplus P \to V$ is the natural projection onto $V$. 
Letting $p_P:V\oplus P \to P$ be the natural projection onto $P$, we have
\begin{eqnarray}
\label{eq:needthis4}
p_P\circ\tau\circ(k\otimes_R g) &=& p_P\circ\tau \circ g' \;= \;p_P\circ\iota_P\circ\pi_{R,P}\;=\;\pi_{R,P},\\
p_V\circ\tau\circ(k\otimes_R g) &=& p_V\circ\tau\circ g'\;=\;p_V\circ\iota_P\circ\pi_{R,P}\;=\;0,\nonumber\\
p_V\circ\tau\circ(k\otimes_R j) &=& p_V\circ\tau\circ\overline{j} \;=\;\xi\circ\overline{h}\circ\overline{j}\;=\; \xi,\mbox{ and}\nonumber\\
p_P\circ\tau\circ(k\otimes_R j) &=& p_P\circ\tau\circ\overline{j} \,.\nonumber
\end{eqnarray}
By Remark \ref{rem:josepaper}(i), there exists an $R\Lambda$-module homomorphism $\lambda:C\to P_R$ such that
$$k\otimes_R\lambda=\pi_{R,P}^{-1}\circ p_P\circ\tau\circ\overline{j}\,.$$
Letting $j_1=j-g\circ \lambda$ shows that $j_1$ is also an $R\Lambda$-module splitting of $h$ and
$\overline{j_1}=k\otimes_R j_1$ is also a $\Lambda$-module splitting of $\overline{h}$. Hence, on replacing $j$ by
$j_1$ and using the equations (\ref{eq:needthis4}), we see that
$p_V\circ\tau\circ(k\otimes_R j_1) = \xi$ and
$$p_P\circ\tau\circ(k\otimes_R j_1) = p_P\circ\tau\circ\overline{j}-p_P\circ\tau\circ (k\otimes_Rg)\circ(k\otimes_R\lambda)=0\,.$$
This means that $(j_1,g):  C\oplus P_R \to T$ provides an isomorphism between the lifts
$(C\oplus P_R,\xi\oplus \pi_{R,P})$ and $(T,\tau)$ of $V\oplus P$ over $R$. 
Hence the map (\ref{eq:bijproj}) is bijective for all $R\in\mathrm{Ob}(\mathcal{C})$, which 
completes the proof of Lemma \ref{lem:addproj}. 
\end{proof}

The following definition of stable equivalence of Morita type goes back to Brou\'{e} \cite{broue1}.

\begin{dfn}
\label{def:stabeq} 
Suppose $\Lambda$ and $\Gamma$ are self-injective finite dimensional $k$-algebras.
Let $X$ be a $\Gamma$-$\Lambda$-bimodule and let $Y$ be a $\Lambda$-$\Gamma$-bimodule. 
We say $X$ and $Y$ induce a \emph{stable equivalence of Morita type} between $\Lambda$ and 
$\Gamma$, if $X$ and $Y$ are projective both as left and as right modules, and if 
\begin{eqnarray}
\label{eq:stab}
Y\otimes_{\Gamma}X&\cong& \Lambda\oplus P \quad\mbox{ as $\Lambda$-$\Lambda$-bimodules, and} \\ 
\nonumber
X\otimes_{\Lambda}Y&\cong& \Gamma\oplus Q \quad\mbox{ as $\Gamma$-$\Gamma$-bimodules}, 
\end{eqnarray}
where $P$ is a projective $\Lambda$-$\Lambda$-bimodule, and $Q$ is a projective 
$\Gamma$-$\Gamma$-bimodule.
In particular, $X\otimes_{\Lambda}-$ and $Y\otimes_{\Gamma}-$ induce mutually inverse 
equivalences between the stable module categories $\Lambda\mbox{-\underline{mod}}$ and 
$\Gamma\mbox{-\underline{mod}}$.
\end{dfn}

\begin{rem}
\label{rem:rickardagain}
It follows from a result by Rickard (see \cite[Cor. 5.5]{rickard1} and \cite[Prop. 6.3.8]{kozim}) that a derived 
equivalence between $D^b(\Lambda\mbox{-{mod}})$ and $D^b(\Gamma\mbox{-{mod}})$ induces
a stable equivalence of Morita type between $\Lambda$ and $\Gamma$.

More precisely,
let $K^{b}(\Lambda\mbox{-proj})$ be the full subcategory of $D^b(\Lambda\mbox{-{mod}})$
consisting of all objects isomorphic to bounded complexes of finitely generated projective
$\Lambda$-modules. Then $K^{b}(\Lambda\mbox{-proj})$ is a thick subcategory of
$D^b(\Lambda\mbox{-{mod}})$ and we can build the Verdier quotient
$$D^b(\Lambda\mbox{-{mod}})/K^{b}(\Lambda\mbox{-proj})$$
(see, for example, \cite[Sect. 4.6]{krause2008} for the construction of this quotient).
Rickard proved in \cite[Thm. 2.1]{rickardJPAA1989} that this Verdier quotient is equivalent
as a triangulated category to the stable module category $\Lambda\mbox{-\underline{mod}}$.

Suppose now that there is a derived equivalence between $D^b(\Lambda\mbox{-{mod}})$ and 
$D^b(\Gamma\mbox{-{mod}})$. As in Section \ref{s:derivedeq}, there exists 
a nice two-sided tilting complex $P^\bullet$ (see Definition \ref{def:nice2sidedtilting}(b)
in the case where $R=k$). Following the proof of \cite[Cor. 5.5]{rickard1}, let $T^\bullet$ be a projective 
$\Gamma$-$\Lambda$-bimodule resolution of $P^\bullet$ such that all terms of $T^\bullet$ are
finitely generated projective $\Gamma$-$\Lambda$-bimodules. For large $n>0$, we can truncate
$T^\bullet$ to obtain a bounded complex
$$S^\bullet:\quad \cdots \to 0\to S^{-n}\to  T^{-n+1}\to  T^{-n+2}\to \cdots$$
that is isomorphic to $P^\bullet$ in $D^b((\Gamma\otimes\Lambda^{op})\mbox{-{mod}})$, 
where all terms but $S^{-n}$ are projective $\Gamma$-$\Lambda$-bimodules and 
$S^{-n}$ is projective as a left $\Gamma$-module and as a right $\Lambda$-module.
If we let $X=\Omega_{\Gamma\Lambda}^{-n}(S^{-n})$, the
$(-n)$-th syzygy as a $\Gamma$-$\Lambda$-bimodule, then $S^\bullet$ is isomorphic
to the one-term complex $X$ concentrated in degree 0 in 
$D^b((\Gamma\otimes_k\Lambda^{op})\mbox{-{mod}})/
K^{b}(\Gamma\otimes_k\Lambda^{op})\mbox{-proj})$.
Moreover, we have that $\widetilde{S}^\bullet=\mathrm{Hom}_\Gamma(S^\bullet,\Gamma)$ has the form
$$\widetilde{S}^\bullet:\quad \cdots \to
\widetilde{T}^{n-2}\to \widetilde{T}^{n-1}\to \widetilde{S}^{n}\to 0 \to\cdots$$
where $\widetilde{T}^i=\mathrm{Hom}_\Gamma(T^{-i},\Gamma)$ for $i<n$ and 
$\widetilde{S}^{n}=\mathrm{Hom}_\Gamma(S^{-n},\Gamma)$. 
If we let $Y=\Omega_{\Lambda\Gamma}^{n}(\widetilde{S}^{n})$, the
$n$-th syzygy as a $\Lambda$-$\Gamma$-bimodule, then $\widetilde{S}^\bullet$ is isomorphic
to the one-term complex $Y$ concentrated in degree 0 in 
$D^b((\Lambda\otimes_k\Gamma^{op})\mbox{-{mod}})/K^{b}(\Lambda\otimes_k\Gamma^{op})\mbox{-proj})$.
Using (\ref{eq:2tiltbetter}), we obtain
\begin{eqnarray*}
Y\otimes_\Gamma X &\cong& \widetilde{S}^\bullet\otimes_{\Gamma}S^\bullet\;\cong\;
\widetilde{P}^\bullet \otimes_{\Gamma}P^\bullet \;\cong \;
\Lambda \qquad\,
\mbox{in $D^b((\Lambda\otimes_k\Lambda^{op})\mbox{-mod})/K^{b}((\Lambda\otimes_k\Lambda^{op})\mbox{-proj})$, and}\\
X\otimes_\Lambda Y &\cong& S^\bullet\otimes_{\Gamma}\widetilde{S}^\bullet\;\cong\;
P^\bullet \otimes_{\Lambda}\widetilde{P}^\bullet\;\cong\;\Gamma \qquad
\mbox{in $D^b((\Gamma\otimes_k\Gamma^{op})\mbox{-mod})/K^{b}((\Gamma\otimes_k\Gamma^{op})\mbox{-proj})$.}
\end{eqnarray*}
Using that $D^b((\Lambda\otimes_k\Lambda^{op})\mbox{-{mod}})/
K^{b}((\Lambda\otimes_k\Lambda^{op})\mbox{-proj})$ is equivalent as a triangulated
category to $(\Lambda\otimes_k\Lambda^{op})\mbox{-\underline{mod}}$ and that
$D^b((\Gamma\otimes_k\Gamma^{op})\mbox{-{mod}})/
K^{b}((\Gamma\otimes_k\Gamma^{op})\mbox{-proj})$ is equivalent as a triangulated
category to $(\Gamma\otimes_k\Gamma^{op})\mbox{-\underline{mod}}$, we obtain
that $X$ and $Y$ satisfy (\ref{eq:stab}). In other words, $X$ and $Y$ induce  a stable equivalence 
of Morita type between $\Lambda$ and  $\Gamma$.
\end{rem}

The following result is proved using Theorem \ref{thm:deformations} and
Lemma \ref{lem:addproj}.

\begin{prop}
\label{prop:stable1}
Suppose $\Lambda$ and $\Gamma$ are finite dimensional self-injective $k$-algebras.
Let $P^\bullet$ be a nice two-sided tilting complex in 
$D^b((\Gamma\otimes_k\Lambda^{op})\mbox{-$\mathrm{mod}$})$ such that $P^\bullet$ is isomorphic
in $D^b((\Gamma\otimes_k\Lambda^{op})\mbox{-$\mathrm{mod}$})/
K^{b}((\Gamma\otimes_k\Lambda^{op})\mbox{-$\mathrm{proj}$})$
to a one-term complex $X$ concentrated in degree 0, as in Remark $\ref{rem:rickardagain}$.
Let $V$ be a finitely generated $\Lambda$-module, and let ${V'} = X\otimes_\Lambda V$, so ${V'}$
is a finitely generated $\Gamma$-module. Then 
$R(\Lambda,V)$ and $R(\Gamma,{V'})$ are isomorphic in $\hat{\mathcal{C}}$.
\end{prop}

\begin{proof}
If we view $V$ as a one-term complex concentrated in degree 0, then it follows from
Proposition \ref{prop:modulecase} and Theorem \ref{thm:deformations} that $R(\Lambda,V)\cong
R(\Gamma,P^\bullet\otimes_\Lambda V)$ in $\hat{\mathcal{C}}$.
We have that 
$$P^\bullet\otimes_\Lambda V\cong S^\bullet\otimes_\Lambda V\cong X\otimes_\Lambda V$$
in $D^b(\Gamma\mbox{-{mod}})/K^{b}(\Gamma\mbox{-$\mathrm{proj}$})$, where $S^\bullet$ is as in
Remark \ref{rem:rickardagain}. By  \cite[Thm. 2.1]{rickardJPAA1989},
$$D^b(\Gamma\mbox{-{mod}})/K^{b}(\Gamma\mbox{-$\mathrm{proj}$})\cong \Gamma\mbox{-\underline{mod}}$$
as triangulated categories. 
By Definition \ref{def:nice2sidedtilting}(b) and by (\ref{eq:stab}), we have
\begin{eqnarray*}
\widetilde{P}^\bullet \otimes_{\Gamma}(P^\bullet\otimes_\Lambda V) &\cong& V\qquad\qquad\qquad\;\,
\mbox{in $D^-(\Lambda)$, and}\\
Y\otimes_\Gamma (X\otimes_\Lambda V) &\cong& V\oplus( P\otimes_\Lambda V)\quad
\mbox{in $\Lambda$-mod.}
\end{eqnarray*}
Since moreover $R(\Lambda,V)\cong R(\Lambda, V\oplus (P\otimes_\Lambda V))$ by Lemma 
\ref{lem:addproj}, we obtain that
$$R(\Gamma,P^\bullet\otimes_\Lambda V)\cong R(\Gamma,X\otimes_\Lambda V)$$
in $\hat{\mathcal{C}}$.
Therefore,
$R(\Lambda,V)\cong R(\Gamma,P^\bullet\otimes_\Lambda V)\cong R(\Gamma,X\otimes_\Lambda V)$
in $\hat{\mathcal{C}}$.
\end{proof}

Note that not every stable equivalence of Morita type between self-injective algebras is induced by a 
derived equivalence (see, for example, \cite{dugas} and its references).
Therefore, we next show that an arbitrary stable equivalence of Morita type between self-injective algebras
preserves versal deformation rings.
The arguments are very similar to those used in \cite[Sect. 2.2]{3sim}. 

\begin{prop}
\label{prop:stabmordef}
Suppose $\Lambda$ and $\Gamma$ are self-injective finite dimensional $k$-algebras.
Suppose $X$ is a $\Gamma$-$\Lambda$-bimodule and $Y$ is a $\Lambda$-$\Gamma$-bimodule
that induce a stable equivalence of Morita type between $\Lambda$ and $\Gamma$. Let $V$ be a 
finitely generated $\Lambda$-module, and define $V'=X\otimes_{\Lambda}V$. Then 
the deformation functors $\hat{F}_V$ and $\hat{F}_{V'}$ are naturally isomorphic.
In particular, the versal deformation rings $R(\Lambda,V)$ and $R(\Gamma,V')$ are isomorphic
in $\hat{\mathcal{C}}$, and $R(\Lambda,V)$ is a universal deformation ring of $V$ if and only if 
$R(\Gamma,V')$ is a universal deformation ring of $V'$.
\end{prop}

\begin{proof}
Let $R\in\mathrm{Ob}(\mathcal{C})$ be Artinian. Then $X_R=R\otimes_kX$ is projective as left 
$R\Gamma$-module and as right $R\Lambda$-module, and $Y_R=R\otimes_k Y$ is projective as left 
$R\Lambda$-module and as right $R\Gamma$-module.
Since $X_R\otimes_{R\Lambda} (Y_R) \cong R\otimes_k(X\otimes_\Lambda Y)$, we have, using 
(\ref{eq:stab}),
\begin{eqnarray*}
Y_R\otimes_{R\Gamma} X_R&\cong& R\Lambda \oplus P_R \quad\mbox{ as 
$R\Lambda$-$R\Lambda$-bimodules, and}\\
X_R\otimes_{R\Lambda} Y_R&\cong &R\Gamma\oplus Q_R \quad\mbox{ as 
$R\Gamma$-$R\Gamma$-bimodules},
\end{eqnarray*}
where $P_R=R\otimes_k P$ is a projective $R\Lambda$-$R\Lambda$-bimodule and
$Q_R=R\otimes_k Q$ is a projective $R\Gamma$-$R\Gamma$-bimodule.

Since $P$  is a projective $\Lambda$-$\Lambda$-bimodule, it follows that $P\otimes_\Lambda V$
is a projective left $\Lambda$-module. By Lemma \ref{lem:addproj} it follows that 
$P\otimes_\Lambda V$ has a universal deformation ring, which is isomorphic to $k$.
In particular, every lift of $P\otimes_\Lambda V$ over $R$ is isomorphic to 
$\left(R\otimes_k (P\otimes_\Lambda V), \pi_{R,P\otimes_\Lambda V}\right)$, where
$\pi_{R,P\otimes_\Lambda V}:k\otimes_R \left(R\otimes_k (P\otimes_\Lambda V)\right)
\to P\otimes_\Lambda V$ is the natural isomorphism of $\Lambda$-modules.

Let now $(M,\phi)$ be a lift of $V$ over $R$. Then $M$ is a finitely generated $R\Lambda$-module.
Define $M'=X_R\otimes_{R\Lambda}M$.
Since $X_R$ is a finitely generated projective right $R\Lambda$-module and since $M$ is a finitely 
generated abstractly free $R$-module, it follows that $M'$ is a finitely generated projective, 
and hence abstractly free, $R$-module. 

Next we note that we can view $\Lambda\mbox{-mod}$ as the full subcategory of $R\Lambda\mbox{-mod}$
consisting of all finitely generated $R\Lambda$-modules on which the maximal ideal $m_R$ of $R$ acts 
trivially. Moreover, on $\Lambda\mbox{-mod}$ the functor $X_R\otimes_{R\Lambda} -$
coincides with the functor $X\otimes_{\Lambda} -$.
Define $\phi'=X_R\otimes_{R\Lambda}\phi$, so $\phi'$ is a $\Gamma$-module isomorphism, since
$X_R$ is projective as a right $R\Lambda$-module. 
Then
\begin{equation}
\label{eq:lala}
M'\otimes_Rk=(X_R\otimes_{R\Lambda}M)\otimes_Rk = X_R\otimes_{R\Lambda}(M\otimes_Rk)
\xrightarrow{\phi'} X_R\otimes_{R\Lambda}V =V'
\end{equation}
which means $(M',\phi')$ is a lift of $V'$ over $R$. We therefore obtain for all $R\in\mathrm{Ob}(\mathcal{C})$ a well-defined map 
\begin{eqnarray*}
\tau_R:\quad F_V(R)&\to&F_{V'}(R)\,,\\
{[M,\phi]}&\mapsto&[M',\phi']=[X_R\otimes_{R\Lambda}M,X_R\otimes_{R\Lambda}\phi]\,.
\end{eqnarray*}

We need to show that $\tau_R$ is bijective.
Arguing as in (\ref{eq:lala}), we see that $(Y_R\otimes_{R\Gamma}M',Y_R\otimes_{R\Gamma}\phi')$
is a lift of $Y\otimes_\Gamma V'\cong V \oplus (P\otimes_\Lambda V)$ over $R$. Moreover, 
\begin{eqnarray}
\label{eq:olala}
(Y_R\otimes_{R\Gamma}M',Y_R\otimes_{R\Gamma}\phi')
&\cong&((R\Lambda\oplus P_R)\otimes_{R\Lambda}M, (R\Lambda\oplus P_R)\otimes_{R\Lambda}\phi) \\
&\cong&(M\oplus (P_R\otimes_{R\Lambda}M), \phi\oplus(P_R\otimes_{R\Lambda}\phi)). \nonumber
\end{eqnarray}
Since $(P_R\otimes_{R\Lambda}M,P_R\otimes_{R\Lambda}\phi)$ is
a lift of the projective $\Lambda$-module $P\otimes_\Lambda V$ over $R$,
it follows from Lemma \ref{lem:addproj} that $\tau_R$ is injective.

Now let $(L,\psi)$ be a lift of $V'=X\otimes_{\Lambda}V$ over $R$.  Then 
$(L',\psi')= (Y_R\otimes_{R\Gamma}L,Y_R\otimes_{R\Gamma}\psi)$ is a lift of 
$V''=Y\otimes_{\Gamma}V'\cong V \oplus (P\otimes_\Lambda V)$ over $R$. By Lemma \ref{lem:addproj}, there exists a lift 
$(M,\phi)$ of $V$ over $R$ such that $(L',\psi')$ is isomorphic to the lift
$\left(M\oplus (R\otimes_k (P\otimes_\Lambda V)),\phi\oplus\pi_{R,P\otimes_\Lambda V}\right)$.
Arguing similarly as in (\ref{eq:olala}), we then have that $(L',\psi')$ is isomorphic to $(M'',\phi'')=
(Y_R\otimes_{R\Gamma}M',Y_R\otimes_{R\Gamma}\phi')$ where $(M',\phi')=(X_R \otimes_{R\Lambda}M,X_R\otimes_{R\Lambda}\phi)$. Therefore, $(X_R\otimes_{R\Lambda}L' ,X_R\otimes_{R\Lambda}\psi')\cong (X_R\otimes_{R\Lambda}M'',X_R\otimes_{R\Lambda}\phi'')$.
Arguing again similarly as in (\ref{eq:olala}) and using Lemma \ref{lem:addproj}, we have 
\begin{eqnarray*}
(X_R\otimes_{R\Lambda}L' ,X_R\otimes_{R\Lambda}\psi') 
&\cong&\left(L\oplus (R\otimes_k (Q\otimes_\Gamma V')),\psi\oplus \pi_{R,Q\otimes_\Gamma V'}\right),
\mbox{ and }\\
(X_R\otimes_{R\Lambda}M'',X_R\otimes_{R\Lambda}\phi'')
&\cong&\left(M' \oplus (R\otimes_k (Q\otimes_\Gamma V')),\phi'\oplus \pi_{R,Q\otimes_\Gamma V'}\right).
\end{eqnarray*}
Thus by Lemma \ref{lem:addproj}, it follows that $(L,\psi)\cong (M',\phi')$, i.e. $\tau_R$ is surjective.

To show that the maps $\tau_R$ are natural with respect to morphisms $\alpha:R\to R'$ in 
$\mathcal{C}$, consider $(M,\phi)$ and $(M',\phi')$ as above.
Since $X_R$ is a projective right $R\Lambda$-module and
$M$ is a free $R$-module, there exists a natural isomorphism
$$f:\quad R' \otimes_{R,\alpha}M'= 
R'\otimes_{R,\alpha}(X_R\otimes_{R\Lambda}M)
\to X_{R'}\otimes_{R'\Lambda}(R'\otimes_{R,\alpha} M)$$
of $R'\Gamma$-modules. It is straightforward to see that $f$ provides an isomorphism between the
lifts $(R'\otimes_{R,\alpha}M',(\phi')_\alpha)$ and 
$(X_{R'} \otimes_{R'\Lambda} (R'\otimes_{R,\alpha}M),X_{R'}\otimes_{R'\Lambda}(\phi_\alpha))$
of $V'$ over $R'$.

Since the deformation functors $\hat{F}_V$ and $\hat{F}_{V'}$ are continuous, this implies that they are naturally isomorphic. Hence the versal deformation rings $R(\Lambda,V)$ and $R(\Gamma,V')$ are isomorphic in $\hat{\mathcal{C}}$.
Moreover, $R(\Lambda,V)$ is universal if and only if  $R(\Gamma,V')$ is universal.
\end{proof}

The following two remarks will be important in Section \ref{s:examples}.

\begin{rem}
\label{rem:stabmor}
Using the notation of Proposition \ref{prop:stabmordef}, suppose that the stable endomorphism ring 
$\underline{\mathrm{End}}_{\Lambda}(V)$ is isomorphic to $k$. Then it follows that $\underline{\mathrm{End}}_{\Gamma}(V')$ is also isomorphic to $k$. 
Because $\Lambda$ and $\Gamma$ are self-injective, this implies by \cite[Thm. 2.6]{blehervelez} that 
the versal deformation rings $R(\Lambda,V)$ and $R(\Gamma,V')$ are in fact universal.
Since $\underline{\mathrm{End}}_{\Gamma}(V')=k$, 
there exists a non-projective indecomposable $\Gamma$-module $V'_0$ (unique up to isomorphism) that is a 
direct summand of  $V'$ with $\underline{\mathrm{End}}_{\Gamma}(V'_0)=k$ and $R(\Gamma,V')\cong R(\Gamma,V'_0)$
(see Lemma \ref{lem:addproj}). 
It follows that $R(\Lambda,V)\cong R(\Gamma,V'_0)$. 
\end{rem}

\begin{rem}
\label{rem:stablegood}
Using the notation of Proposition \ref{prop:stabmordef}, let $G:\Lambda\mbox{-\underline{mod}}\to \Gamma\mbox{-\underline{mod}}$ denote the stable equivalence
of Morita type induced by $X\otimes_\Lambda-$. Denote by $\mathrm{mod}_{\mathcal{P}}(\Lambda)$ (resp. $\mathrm{mod}_{\mathcal{P}}(\Gamma)$)
the full subcategory of $\Lambda\mbox{-mod}$ (resp. $\Gamma\mbox{-mod}$) whose objects are the modules that have no non-zero projective summands,
and denote the correspondence between $\mathrm{mod}_{\mathcal{P}}(\Lambda)$
and $\mathrm{mod}_{\mathcal{P}}(\Gamma)$ that is induced by $G$ again by $G$.
Let
$$0\to A \xrightarrow{\genfrac{(}{)}{0pt}{}{\mbox{\tiny $f$}}{\mbox{\tiny $s$}}} B \oplus P \xrightarrow{(g,t)} C\to 0$$
be an almost split sequence in $\Lambda\mbox{-mod}$ where $A,B,C$ are in $\mathrm{mod}_{\mathcal{P}}(\Lambda)$, $B$ is non-zero 
and $P$ is projective. Then, by \cite[Prop. X.1.6]{ars}, for any morphism $g':G(B)\to G(C)$ with $G(g)=g'$ in 
$\Gamma\mbox{-\underline{mod}}$,  there is an almost split sequence
$$0\to G(A) \xrightarrow{\genfrac{(}{)}{0pt}{}{\mbox{\tiny $f'$}}{\mbox{\tiny $u$}}} G(B)  \oplus P' \xrightarrow{(g',v)} G(C)\to 0$$
in $\Gamma\mbox{-mod}$ where $P'$ is projective and $G(f)=f'$ in $\Gamma\mbox{-\underline{mod}}$.

Suppose now additionally that $\Lambda$ and $\Gamma$ are symmetric algebras with no blocks of Loewy length $2$. Then it follows
by \cite[Cor. X.1.9 and Prop. X.1.12]{ars} that the stable Auslander-Reiten quivers of $\Lambda$ and $\Gamma$ 
are isomorphic stable translation quivers, and $G$ commutes with the syzygy functor $\Omega$. 

If $V$ and $V'$ are as in Proposition \ref{prop:stabmordef}, this means, in particular, that the stable Auslander-Reiten quiver components of $V$ and $V'$
match up, including the relative positions of $V$ and $V'$ in these components. 
We will see in Section \ref{s:examples} how to use this to transfer results about universal deformation rings of $\Lambda$-modules $V$
with $\underline{\mathrm{End}}_{\Lambda}(V)=k$ to results about the universal deformation rings of the corresponding $\Gamma$-modules $V'$.
\end{rem}


\section{A family of examples}
\label{s:examples}
\setcounter{equation}{0}

We use the notation from Section \ref{s:derivedequivalences}. Moreover, we assume that $k$ is an algebraically closed field.
In this section, we consider the derived equivalence classes of the family of symmetric $k$-algebras
$D(3\mathcal{R})$ introduced by Erdmann in \cite{erd}.  In Section \ref{s:family}, we describe these derived equivalence classes,
which were obtained by Holm in \cite[Sect. 3.2]{holm}. In Section \ref{s:exampleresults}, we apply the results from Section
\ref{s:derivedequivalences} together with the results in \cite{blehervelez} and \cite{velez2015} to
obtain universal deformation rings of $\Lambda$-modules for
other algebras $\Lambda$ of dihedral type that are derived equivalent to members of the family $D(3\mathcal{R})$.


\subsection{The derived equivalence classes of the algebras in the family $D(3\mathcal{R})$}
\label{s:family}

In \cite{erd}, Erdmann introduced the symmetric $k$-algebras of dihedral type
$D(3\mathcal{R})^{a,b,c,d}$, for integers $a\ge 1$, $b,c,d\ge 2$, see Figure \ref{fig:algebra3R}.
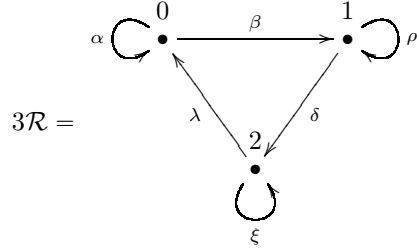
\begin{figure}[ht] \hrule \caption{\label{fig:algebra3R} The family $D(3\mathcal{R})^{a,b,c,d}=k[3\mathcal{R}]/I_{3\mathcal{R},a,b,c,d}$
for $a\ge 1$; $b,c,d\ge 2$.}
$$3\mathcal{R}=\vcenter{\xymatrix @R=-.1pc {
0&&1\\
\ar@(ul,dl)_{\alpha} \bullet \ar[rr]^{\beta}  &&\bullet\ar[ldddddddd]^{\delta} \ar@(ur,dr)^{\rho} 
\\&&\\&&\\&&\\&&\\&&\\&&\\&2&\\ 
&
\bullet\ar[uuuuuuuul]^{\lambda}
\ar@(ld,rd)_{\xi} & }}$$
$$I_{3\mathcal{R},a,b,c,d}=\langle \alpha\lambda,\lambda\xi,
\xi\delta,\delta\rho,\rho\beta,\beta\alpha,
\alpha^b-(\lambda\delta\beta)^a,\rho^c-(\beta\lambda\delta)^a,\xi^d-(\delta\beta\lambda)^a\rangle.$$
\vspace{1ex}
\hrule
\end{figure}

In \cite[Sect. 3.2]{holm}, Holm determined all algebras of dihedral type with precisely three isomorphism
classes of simple modules that are derived equivalent to $D(3\mathcal{R})^{a,b,c,d}$ for various $a,b\ge  1$; $c,d\ge 2$. 
Note that for $a\ge 1$; $c,d\ge 2$, we define the algebra $D(3\mathcal{R})^{a,1,c,d}$ as in Figure \ref{fig:algebra3R'}.
\begin{figure}[ht] \hrule \caption{\label{fig:algebra3R'} The family $D(3\mathcal{R})^{a,1,c,d}=k[3\mathcal{R'}]/I_{3\mathcal{R}',a,c,d}$
for $a\ge 1$; $c,d\ge 2$.}
$$3\mathcal{R}'=\vcenter{\xymatrix @R=-.1pc {
0&&1\\
\bullet \ar[rr]^{\beta}  &&\bullet\ar[ldddddddd]^{\delta} \ar@(ur,dr)^{\rho} 
\\&&\\&&\\&&\\&&\\&&\\&&\\&2&\\ 
&
\bullet\ar[uuuuuuuul]^{\lambda}
\ar@(ld,rd)_{\xi} & }}$$
$$I_{3\mathcal{R}',a,c,d}= \langle \lambda\xi,\xi\delta,\delta\rho,\rho\beta,
\rho^c-(\beta\lambda\delta)^a,\xi^d-(\delta\beta\lambda)^a\rangle.$$
\vspace{1ex}
\hrule
\end{figure}
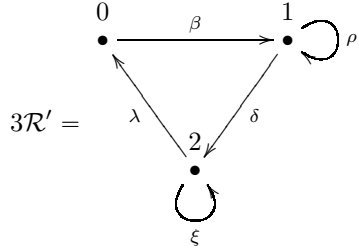

Holm showed in \cite[Thm. 3.4]{holm} that no block of a group algebra with dihedral defect groups is derived equivalent to $D(3\mathcal{R})^{a,b,c,d}$
for any $a,b\ge 1$; $c,d\ge 2$. Note that up to derived equivalence, we can order $1\le a\le b\le c\le d$; $2\le c$ in $D(3\mathcal{R})^{a,b,c,d}$.
By \cite[Sect. 3.2]{holm}, there are precisely five additional families of Morita equivalence classes of algebras of dihedral type 
that are derived equivalent to the algebras in the family $D(3\mathcal{R})^{a,b,c,d}$, $a,b\ge 1$; $c,d\ge 2$. 
We list these Morita equivalence classes in Figure \ref{fig:derived3R}. 
\begin{figure}[ht] \hrule \caption{\label{fig:derived3R} Morita equivalence classes of algebras of dihedral type that are derived
equivalent to $D(3\mathcal{R})^{a,b,c,d}$ for various $a,b\ge 1$; $c,d\ge 2$ (see \cite[Sect. 3.2]{holm}).}
\begin{enumerate}
\item[(A)] The family $D(3\mathcal{Q})^{b,c,d}=k[3\mathcal{Q}]/I_{3\mathcal{Q},b,c,d}$,
$b\ge 1$; $c,d\ge 2$, which is derived equivalent to $D(3\mathcal{R})^{1,b,c,d}\;$:
$$3\mathcal{Q}=\vcenter{\xymatrix @R=-.1pc {
0&&1\\
\ar@(ul,dl)_{\alpha} \bullet \ar[rr]^{\beta}  &&\bullet\ar[ldddddddd]^{\delta} \ar@(ur,dr)^{\rho} 
\\&&\\&&\\&&\\&&\\&&\\&&\\&&\\ 
&
\bullet\ar[uuuuuuuul]^{\lambda}\\
&2&}}$$
$$I_{3\mathcal{Q},b,c,d}= \langle \alpha\lambda,\delta\rho,\rho\beta,\beta\alpha,
\alpha^c-(\lambda\delta\beta)^b,\rho^d-(\beta\lambda\delta)^b\rangle.$$
\vspace{2ex}
\item[(B)]
The family $D(3\mathcal{L})^{c,d}=k[3\mathcal{L}]/I_{3\mathcal{L},c,d}$,
$c,d\ge 2$, which is derived equivalent to $D(3\mathcal{R})^{1,1,c,d}\;$:
$$3\mathcal{L}=\vcenter{\xymatrix @R=-.1pc {
0&&1\\
\ar@(ul,dl)_{\alpha} \bullet \ar[rr]^{\beta}  &&\bullet\ar[ldddddddd]^{\delta} 
\\&&\\&&\\&&\\&&\\&&\\&&\\&&\\ 
&
\bullet\ar[uuuuuuuul]^{\lambda}\\
&2&}}$$
$$I_{3\mathcal{L},c,d}=\langle \alpha\lambda,\beta\alpha,
\alpha^d-(\lambda\delta\beta)^c,\delta(\beta\lambda\delta)^c\rangle.$$
\vspace{2ex}
\item[(C)] 
The family $D(3\mathcal{A})_2^{c,d}=k[3\mathcal{A}]/I_{(3\mathcal{A})_2,c,d}$,
$c\ge d\ge 2$, which is derived equivalent to $D(3\mathcal{R})^{1,1,c,d}\;$:
$$3\mathcal{A}=\vcenter{\xymatrix @R=-.1pc {
&0&\\
1\; \bullet \ar@<.8ex>[r]^(.56){\beta} \ar@<1ex>[r];[]^(.44){\gamma}
& \bullet \ar@<.8ex>[r]^(.44){\delta} \ar@<1ex>[r];[]^(.56){\eta} & \bullet\; 2}}$$
$$I_{(3\mathcal{A})_2,c,d}=\langle \gamma\eta,\delta\beta,(\beta\gamma)^c-(\eta\delta)^d\rangle.$$
\vspace{2ex}
\item[(D)] 
The family $D(3\mathcal{B})_2^{b,c,d}=k[3\mathcal{B}]/I_{(3\mathcal{B})_2,b,c,d}$,
$b,c\ge 1$ $(b+c> 2)$; $d\ge 2$ which is derived equivalent to $D(3\mathcal{R})^{1,b,c,d}$ if $c\ge 2$ and to
$D(3\mathcal{R})^{1,c,b,d}$ if $c=1$ (and hence $b\ge 2$)\;:
$$3\mathcal{B}=\vcenter{\xymatrix @R=-.1pc {
&1&0&\\
&\ar@(ul,dl)_{\alpha} \bullet \ar@<.8ex>[r]^{\beta} \ar@<.9ex>[r];[]^{\gamma}
& \bullet \ar@<.8ex>[r]^(.46){\delta} \ar@<.9ex>[r];[]^(.54){\eta} & \bullet\;2}}$$
$$I_{(3\mathcal{B})_2,b,c,d}=\langle \alpha\gamma,\beta\alpha,\gamma\eta,\delta\beta,
\alpha^d-(\gamma\beta)^b,(\beta\gamma)^b-(\eta\delta)^c\rangle.$$
\vspace{2ex}
\item[(E)] 
The family $D(3\mathcal{D})_2^{a,b,c,d}=k[3\mathcal{D}]/I_{(3\mathcal{D})_2,a,b,c,d}$,
$a,b\ge 1$; $c,d\ge 2$ which is derived equivalent to $D(3\mathcal{R})^{a,b,c,d}$:
$$3\mathcal{D}=\vcenter{\xymatrix @R=-.1pc {
&1&0&2\\
&\ar@(ul,dl)_{\alpha} \bullet \ar@<.8ex>[r]^{\beta} \ar@<.9ex>[r];[]^{\gamma}
& \bullet \ar@<.8ex>[r]^(.46){\delta} \ar@<.9ex>[r];[]^(.54){\eta} & \bullet \ar@(ur,dr)^{\xi}}}$$
$$I_{(3\mathcal{D})_2,a,b,c,d}=\langle \alpha\gamma,\beta\alpha,\gamma\eta,\delta\beta,
\eta\xi,\xi\delta,\alpha^c-(\gamma\beta)^a,(\beta\gamma)^a-(\eta\delta)^b,\xi^d-(\delta\eta)^b\rangle.$$
\end{enumerate}
\vspace{1ex}
\hrule
\end{figure}


\subsection{Universal deformation rings for algebras in the derived equivalence classes of members of $D(3\mathcal{R})$}
\label{s:exampleresults}

In \cite{blehertalbott}, \cite{blehervelez} and \cite{velez2015}, the universal deformation rings of certain modules of
$D(3\mathcal{R})^{a,b,c,d}$ were determined for various $a,b\ge 1$; $c,d\ge 2$. Recall that it follows from \cite[Thm. 2.6]{blehervelez}
(see also Remark \ref{rem:stabmor}) that if $\Lambda$ is a self-injective finite dimensional 
$k$-algebra and $V$ is a finitely generated $\Lambda$-module
whose stable endomorphism ring is isomorphic to $k$ then the versal deformation ring $R(\Lambda,V)$ is universal.

\begin{rem}
\label{rem:derivedex1}
In \cite{blehertalbott}, the first author and S. Talbott studied the case $D(3\mathcal{R})^{1,1,2,2}$, which is of polynomial growth, 
and determined all indecomposable $D(3\mathcal{R})^{1,1,2,2}$-modules whose stable endomorphism rings are isomorphic to $k$,
together with their universal deformation rings. 

Using Figure \ref{fig:derived3R}, we obtain the following list of Morita equivalence classes of
algebras of dihedral type that are derived equivalent to
$D(3\mathcal{R})^{1,1,2,2}$:
$$D(3\mathcal{R})^{1,1,2,2}, D(3\mathcal{Q})^{1,2,2}, D(3\mathcal{L})^{2,2},
D(3\mathcal{A})_2^{2,2}, D(3\mathcal{B})_2^{1,2,2}, D(3\mathcal{B})_2^{2,1,2}, D(3\mathcal{D})_2^{1,1,2,2}.$$
By Proposition \ref{prop:stable1}, it follows 
that if $\Lambda$ is any of the algebras in this derived equivalence class and $V$ is an indecomposable $\Lambda$-module
with $\underline{\mathrm{End}}_\Lambda(V)=k$, then $R(\Lambda,V)\cong R(D(3\mathcal{R})^{1,1,2,2},V')$ if $V$
corresponds to $V'$ under the stable  equivalence of Morita type between $\Lambda$ and $D(3\mathcal{R})^{1,1,2,2}$ that is
induced by the derived equivalence. By 
Remark \ref{rem:stablegood}, we see additionally that the stable Auslander-Reiten quiver components of $V$ and $V'$
match up, including the relative positions of $V$ and $V'$ in these components. This reaffirms the results in \cite[Thm. 1.1]{blehertalbott}
(see also  \cite[Props. 3.1-3.3]{blehertalbott}) for these algebras.
\end{rem}

In \cite[Sect. 3]{blehervelez}, the authors studied the algebra $D(3\mathcal{R})^{1,2,2,2}$  and determined all
indecomposable $D(3\mathcal{R})^{1,2,2,2}$-modules whose stable endomorphism rings are isomorphic to $k$,
together with their universal deformation rings. 

Using Proposition \ref{prop:stable1}, Remark \ref{rem:stablegood} and 
\cite[Thm. 1.2]{blehervelez} (see also \cite[Thm. 3.8, Props. 3.9-3.11, Thm. 3.16, Prop. 3.17]{blehervelez}), 
we obtain the following result for all algebras of dihedral type  in the derived equivalence class of $D(3\mathcal{R})^{1,2,2,2}$.

\begin{thm}
\label{thm:derivedex2}
Let $\Lambda$ be one of the following algebras:
\begin{equation}
\label{eq:1222}
D(3\mathcal{R})^{1,2,2,2}, D(3\mathcal{Q})^{2,2,2}, D(3\mathcal{B})_2^{2,2,2}, D(3\mathcal{D})_2^{1,2,2,2}.
\end{equation}
Suppose $\mathfrak{C}$ is a component of the stable Auslander-Reiten quiver $\Gamma_s(\Lambda)$.
\begin{itemize}
\item[(i)] If $\mathfrak{C}$ is one of the two $3$-tubes, then $\Omega(\mathfrak{C})$ is the 
other $3$-tube. There are exactly three $\Omega^2$-orbits of modules
in $\mathfrak{C}$ whose stable endomorphism rings are isomorphic to $k$. 
If $U_0$ is a module that belongs to the boundary of $\mathfrak{C}$, then these three 
$\Omega^2$-orbits are represented by $U_0$, by a successor $U_1$ of $U_0$, and by a 
successor $U_2$ of $U_1$ that does not lie in the $\Omega^2$-orbit of $U_0$. The universal
deformation rings are 
$$R(\Lambda,U_0)\cong R(\Lambda,U_1)\cong k,\quad R(\Lambda,U_2)\cong k[[t]].$$

\item[(ii)] There are infinitely many components of $\Gamma_s(\Lambda)$ of type $\mathbb{Z}
A_\infty^\infty$ that each contain a module whose stable endomorphism ring is isomorphic to $k$.
If $\mathfrak{C}$ is such a component, then $\mathfrak{C}=\Omega(\mathfrak{C})$ and there
are exactly six  $\Omega^2$-orbits $($resp. exactly three $\Omega$-orbits$)$ of modules 
in $\mathfrak{C}$ whose stable endomorphism rings are isomorphic to $k$. These three $\Omega$-orbits are 
represented by a module $V_0$, by a successor $V_1$ of $V_0$ that does not lie in the $\Omega$-orbit of $V_0$, 
and by a successor $V_2$ of $V_1$ that does not lie in the $\Omega^2$-orbit of $V_0$. 
The universal deformation rings are 
$$R(\Lambda,V_0)\cong k[[t]]/(t^2),\quad R(\Lambda,V_1)\cong k,\quad R(\Lambda,V_2)\cong k[[t]].$$

\item[(iii)] There are infinitely many $1$-tubes of $\Gamma_s(\Lambda)$ that each contain a module 
whose stable endomorphism ring is isomorphic to $k$. If $\mathfrak{C}$ is such a component, 
then there is exactly one $\Omega^2$-orbit of modules in $\mathfrak{C}$ whose stable 
endomorphism ring is isomorphic to $k$, represented by a module $W_0$ belonging to the 
boundary of $\mathfrak{C}$. The universal deformation ring of $W_0$ is 
$$R(\Lambda,W_0)\cong k[[t]].$$
\end{itemize}
\end{thm}

\begin{rem}
\label{rem:moredetail2}
Suppose $\Lambda$ is one of the algebras in (\ref{eq:1222}), and suppose $V$ is a $\Lambda$-module with
$\mathrm{End}_\Lambda(V)=k$. 
\begin{enumerate}
\item[(i)] If $\Lambda=D(3\mathcal{R})^{1,2,2,2}$, then $V$ is one of the modules
$$S_0,S_1,S_2,\begin{array}{c}0\\1\end{array},\begin{array}{c}1\\2\end{array}, \begin{array}{c}2\\0\end{array},
\begin{array}{c}0\\1\\2\end{array}, \begin{array}{c}1\\2\\0\end{array}, \begin{array}{c}2\\0\\1\end{array}.$$
Moreover, $R(\Lambda,V)\cong k$ if  $V$ has composition series length 2 or 3, and
$R(\Lambda,V)\cong k[[t]]/(t^2)$ if $V$ is simple.

\item[(ii)] If $\Lambda=D(3\mathcal{Q})^{2,2,2}$, then $V$ is one of the modules
$$S_0,S_1,S_2,\begin{array}{c}0\\1\end{array},\begin{array}{c}1\\2\end{array}, \begin{array}{c}2\\0\end{array},
\begin{array}{c}0\\1\\2\end{array}, \begin{array}{c}1\\2\\0\end{array}, \begin{array}{c}2\\0\\1\end{array}.$$
Moreover, $R(\Lambda,V)\cong k$ if $V=S_2$ or $V$ has composition series length 2, and
$R(\Lambda,V)\cong k[[t]]/(t^2)$ if $V\in\{S_0, S_1\}$ or $V$ has composition series length 3.

\item[(iii)] If $\Lambda=D(3\mathcal{B})_2^{2,2,2}$, then $V$ is one of the modules
$$S_0,S_1,S_2,\begin{array}{c}0\\1\end{array},\begin{array}{c}1\\0\end{array}, \begin{array}{c}0\\2\end{array}, \begin{array}{c}2\\0\end{array},
\begin{array}{cc}\multicolumn{2}{c}{0}\\1&2\end{array}, \begin{array}{cc}1&2\\ \multicolumn{2}{c}{0}\end{array}.$$
Moreover, $R(\Lambda,V)\cong k$ if $V\in\{S_0,S_2\}$ or $V$ has composition series length 3, and
$R(\Lambda,V)\cong k[[t]]/(t^2)$ if $V=S_1$ or $V$ has composition series length 2.

\item[(iv)] If $\Lambda=D(3\mathcal{D})_2^{1,2,2,2}$, then $V$ is one of the modules
$$S_0,S_1,S_2,\begin{array}{c}0\\1\end{array},\begin{array}{c}1\\0\end{array}, \begin{array}{c}0\\2\end{array}, \begin{array}{c}2\\0\end{array},
\begin{array}{cc}\multicolumn{2}{c}{0}\\1&2\end{array}, \begin{array}{cc}1&2\\ \multicolumn{2}{c}{0}\end{array}.$$
Moreover, $R(\Lambda,V)\cong k$ if $V\in\left\{S_0,\begin{array}{c}0\\1\end{array},\begin{array}{c}1\\0\end{array}\right\}$ or $V$ has 
composition series length 3, and
$R(\Lambda,V)\cong k[[t]]/(t^2)$ if $V\in\left\{S_1,S_2,\begin{array}{c}0\\2\end{array},\begin{array}{c}2\\0\end{array}\right\}$.
\end{enumerate}
\end{rem}

In \cite{velez2015}, the second author considered all possible $a\ge 1$, $b,c,d\ge 2$
and determined all indecomposable $D(3\mathcal{R})^{a,b,c,d}$-modules whose (usual) endomorphism rings are isomorphic to $k$. Moreover,
he looked at their components in the stable Auslander-Reiten quiver of $D(3\mathcal{R})^{a,b,c,d}$ and determined all
modules in these components whose stable endomorphism rings are isomorphic to $k$, together with their universal deformation rings. 

We use \cite[Thm. 1.1(iv)]{velez2015} to obtain a result concerning 3-tubes for all allowed parameters
$a,b,c,d$. By using similar arguments as in the proof of \cite[Prop. 4.4]{velez2015}, we can 
include the case of $D(3\mathcal{R})^{a,1,c,d}$ for $a\ge 1$ and $c,d\ge 2$.

\begin{thm}
\label{thm:derivedex3}
Let $\Lambda$ be one of the following algebras:
\begin{equation}
\label{eq:rstu}
D(3\mathcal{R})^{a,b,c,d}, D(3\mathcal{Q})^{b,c,d}, D(3\mathcal{L})^{c,d}, D(3\mathcal{A})_2^{c,d}, D(3\mathcal{B})_2^{b,c,d}, D(3\mathcal{B})_2^{c,b,d}, D(3\mathcal{D})_2^{a,b,c,d},
\end{equation}
where $a,b\ge 1$, $c,d\ge 2$ are allowed parameters according to Figure $\ref{fig:derived3R}$
such that $\Lambda$ is not of polynomial growth.
Suppose $\mathfrak{T}$ is one of the two $3$-tubes of the stable Auslander-Reiten quiver $\Gamma_s(\Lambda)$.
Then $\Omega(\mathfrak{T})$ is the other $3$-tube, and there are precisely three $\Omega$-orbits of modules in 
$\mathfrak{T}\cup \Omega(\mathfrak{T})$ whose stable endomorphism rings are isomorphic to $k$. 
If $U_0$ is a module that belongs to the boundary of $\mathfrak{T}$, then these three 
$\Omega$-orbits are represented by $U_0$, by a successor $U_1$ of $U_0$, and by a 
successor $U_2$ of $U_1$ that does not lie in the $\Omega^2$-orbit of $U_0$. The universal
deformation rings are 
$$R(\Lambda,U_0)\cong R(\Lambda,U_1)\cong k,\quad R(\Lambda,U_2)\cong k[[t]].$$
\end{thm}

Since
in \cite{velez2015} only those components of the stable Auslander-Reiten quiver of $D(3\mathcal{R})^{a,b,c,d}$,
for all $a\ge 1$, $b,c,d\ge 2$, were studied that contain modules whose (usual) endomorphism rings are isomorphic to $k$,
we can only say something about finitely many components of the stable Auslander-Reiten quiver of type
$\mathbb{Z}A_\infty^\infty$, as far as universal deformation rings are concerned. 
Using \cite[Thm. 1.1(i)-(iii)]{velez2015}, we obtain the following result for all allowed parameters
$a,b,c,d$. By using similar arguments as in the proof of \cite[Props. 4.1-4.3]{velez2015}, we can 
include the case of $D(3\mathcal{R})^{a,1,c,d}$ for $a\ge 1$ and $c,d\ge 2$.

\begin{prop}
\label{prop:derivedexvelez}
Let $\Lambda$ be one of the algebras in $(\ref{eq:rstu})$, where $a,b\ge 1$, $c,d\ge 2$ are allowed parameters 
according to Figure $\ref{fig:derived3R}$ such that $\Lambda$ is not of polynomial growth. 
If the quiver of $\Lambda$ is $3\mathcal{Q}$ or $3\mathcal{B}$ we set $a=1$, and if the quiver of $\Lambda$ is 
$3\mathcal{L}$ or $3\mathcal{A}$ we set $a=b=1$.
Let $\mathfrak{C}$ be a component of $\Gamma_s(\Lambda)$ of type $\mathbb{Z}A_\infty^\infty$ 
containing a module whose stable endomorphism ring is $k$.

\begin{enumerate}
\item[(i)] Suppose $a=1=b$. Then there is at least one component $\mathfrak{C}$ such that the following is true:
There are precisely three $\Omega$-orbits of modules in $\mathfrak{C}\cup\Omega(\mathfrak{C})$ whose stable endomorphism rings are 
isomorphic to $k$, represented by  $V_0,V_1,V_2$ such that $V_1$ is a successor of $V_0$ that does not lie in the $\Omega$-orbit of $V_0$, 
$V_2$ is a successor of $V_1$ that does not lie in the $\Omega^2$-orbit of $V_0$, and
$$R(\Lambda,V_0)\cong k[[t]]/(t^c), \quad R(\Lambda,V_1)\cong k, \quad R(\Lambda,V_2)\cong k[[t]]/(t^d).$$
Moreover, $\mathfrak{C}=\Omega(\mathfrak{C})$ if and only if $c=2$ or $d=2$.

\item[(ii)] Suppose $a=1$ and $b\ge 2$. Then there are at least three components $\mathfrak{C}=\mathfrak{C}_{i,j}$
for $(i,j)\in\{(1,b),(2,c),(3,d)\}$ such that the following is true:
There are precisely three $\Omega$-orbits of modules in $\mathfrak{C}_{i,j}\cup\Omega(\mathfrak{C}_{i,j})$ whose stable endomorphism rings 
are isomorphic to $k$, represented by $V_{i,j,0},V_{i,j,1},V_{i,j,2}$ such that $V_{i,j,1}$ is a successor of $V_{i,j,0}$ that does not lie in the 
$\Omega$-orbit of $V_{i,j,0}$, $V_{i,j,2}$ is a successor of $V_{i,j,1}$ that does not lie in the $\Omega^2$-orbit of $V_{i,j,0}$, and
$$R(\Lambda,V_{i,j,0})\cong k[[t]]/(t^j), \quad R(\Lambda,V_{i,j,1})\cong k, \quad R(\Lambda,V_{i,j,2})\cong k[[t]].$$
Moreover, $\mathfrak{C}_{i,j}=\Omega(\mathfrak{C}_{i,j})$ if and only if $j=2$.

\item[(iii)] Suppose $a\ge 2$ and $b=1$. Then there are at least four components $\mathfrak{C}=\mathfrak{C}_{i,j}$
for $(i,j)\in\{(1,a),(2,a),(3,c),(4,d)\}$ such that the following is true:
There are precisely three $\Omega$-orbits of modules in $\mathfrak{C}_{i,j}\cup\Omega(\mathfrak{C}_{i,j})$ whose stable endomorphism rings 
are isomorphic to $k$, represented by $V_{i,j,0},V_{i,j,1},V_{i,j,2}$ such that $V_{i,j,1}$ is a successor of $V_{i,j,0}$ that does not lie in the 
$\Omega$-orbit of $V_{i,j,0}$, $V_{i,j,2}$ is a successor of $V_{i,j,1}$ that does not lie in the $\Omega^2$-orbit of $V_{i,j,0}$, and
$$R(\Lambda,V_{i,j,0})\cong k[[t]]/(t^j), \quad R(\Lambda,V_{i,j,1})\cong k, \quad R(\Lambda,V_{i,j,2})\cong k[[t]].$$
Moreover, $\mathfrak{C}_{i,j}=\Omega(\mathfrak{C}_{i,j})$ if and only if $j=2$.

\item[(iv)] Suppose $a\ge 2$ and $b\ge 2$. Then there are at least nine components $\mathfrak{C}=\mathfrak{C}_{i,j}$
for $(i,j)\in\{(1,a),(2,a),(3,a),(4,b),(5,c),(6,d),(7,\infty),(8,\infty),(9,\infty)\}$ such that the following is true:
There are precisely three $\Omega$-orbits of modules in $\mathfrak{C}_{i,j}\cup\Omega(\mathfrak{C}_{i,j})$ whose stable endomorphism rings 
are isomorphic to $k$, represented by $V_{i,j,0},V_{i,j,1},V_{i,j,2}$ such that $V_{i,j,1}$ is a successor of $V_{i,j,0}$ that does not lie in the 
$\Omega$-orbit of $V_{i,j,0}$, $V_{i,j,2}$ is a successor of $V_{i,j,1}$ that does not lie in the $\Omega^2$-orbit of $V_{i,j,0}$, and
$$\qquad \,\quad R(\Lambda,V_{i,j,0})\cong \left\{\begin{array}{c@{\quad:\quad}l}k[[t]]/(t^j)&j\neq\infty\\k[[t]]&j=\infty\end{array}\right\}, \quad R(\Lambda,V_{i,j,1})\cong k, \quad R(\Lambda,V_{i,j,2})\cong k[[t]].$$
Moreover, $\mathfrak{C}_{i,j}=\Omega(\mathfrak{C}_{i,j})$ if and only if $j=2$.
\end{enumerate}
\end{prop}

\begin{rem}
\label{rem:infinite}
Suppose $\Lambda$ is one of the algebras in $(\ref{eq:rstu})$, where $a,b\ge 1$, $c,d\ge 2$ are allowed parameters according to 
Figure $\ref{fig:derived3R}$. Moreover assume $\Lambda$ is not of polynomial growth.
Note that by \cite[Lemma 3.15]{holm}, $D(3\mathcal{R})^{a,b,c,d}$ and $D(3\mathcal{R})^{b,a,c,d}$ are derived equivalent for all
$a,b\ge 1$; $c,d\ge 2$. In view of Theorem \ref{thm:derivedex2}(ii), it seems plausible that there are usually infinitely many components of 
the stable Auslander-Reiten quiver of $\Lambda$ of type $\mathbb{Z}A_\infty^\infty$ that contain modules whose stable endomorphism 
rings are isomorphic to $k$.
\end{rem}


\end{document}